\documentclass[]{amsart}
\usepackage{amsmath}
\usepackage[english]{babel}
\usepackage[utf8]{inputenc}
\usepackage{amsfonts}
\usepackage{amsthm}
\usepackage{color}
\usepackage{esint}
\numberwithin{equation}{section}

\newtheorem{thm}{Theorem}[section]

\newtheorem{prop}[thm]{Proposition}
\newtheorem{cor}[thm]{Corollary}
\newtheorem{lem}[thm]{Lemma}

\theoremstyle{remark}
\newtheorem{rmk}[thm]{Remark}
\theoremstyle{definition}
\newtheorem{defn}{Definition}[section]

\DeclareMathOperator{\N}{\mathbb{N}}
\DeclareMathOperator{\R}{\mathbb{R}}

\DeclareMathOperator{\cY}{\mathcal{Y}}
\DeclareMathOperator{\cX}{\mathcal{X}}
\DeclareMathOperator{\cI}{\mathcal{I}}

\DeclareMathOperator{\cA}{\mathcal{A}}

\DeclareMathOperator{\cK}{\mathcal{K}}
\DeclareMathOperator{\cL}{\mathcal{L}}
\DeclareMathOperator{\cH}{\mathcal{H}}
\DeclareMathOperator{\cB}{\mathcal{B}}
\DeclareMathOperator{\cS}{\mathcal{S}}

\DeclareMathOperator{\fE}{\mathfrak{E}}

\DeclareMathOperator{\fD}{\mathfrak{D}}
\DeclareMathOperator{\fG}{\mathfrak{G}}

\DeclareMathOperator{\bP}{\mathbb{P}}

\DeclareMathOperator{\divg}{div}

\newcommand{\der}[2]{\frac{d #1}{d #2}}

\newcommand{\Norm}[2]{\left\Vert #1 \right\Vert_{#2}}

\title[Stability of a Riesz-type inequality]{A spherical rearrangement proof of the stability of a Riesz-type inequality and an application to an isoperimetric type problem.}
\author{Giacomo Ascione}
\thanks{The author is supported by MIUR-PRIN 2017, project Stochastic Models for Complex Systems, no. 2017JFFHSH and by Gruppo Nazionale per l’Analisi Matematica, la Probabilit\'a e le loro Applicazioni (GNAMPA-INdAM)}
\address{Dipartimento di Matematica e Applicazioni ``Renato Caccioppoli'', Università degli Studi di Napoli Federico II, 80126 Napoli, Italy}
\email{giacomo.ascione@unina.it}
\begin{document}
	\maketitle
\begin{abstract}
	We prove the stability of the ball as global minimizer of an attractive shape functional under volume constraint, by means of mass transportation arguments. The stability exponent is $1/2$ and it is sharp. Moreover, we use such stability result together with the quantitative (possibly fractional) isoperimetric inequality to prove that the ball is a global minimizer of a shape functional involving both an attractive and a repulsive term with a sufficiently large fixed volume and with a suitable (possibly fractional) perimeter penalization.
\end{abstract}
\section{Introduction}
In recent times different works focused on shape functionals of the form
\begin{equation*}
\fE(E)=P(E)+V_{\alpha}(E),
\end{equation*}
where $P$ is the (de Giorgi) perimeter and
\begin{equation}\label{Rieszpot}
V_\alpha(E)=\int_{E}\int_{E}\frac{1}{|x-y|^{N-\alpha}}dxdy
\end{equation}
is the Riesz potential with $\alpha \in (0,N)$. The most famous case is given by $N=3$ and $\alpha=2$. It is linked with Gamow's liquid drop model (see \cite{gamow1930mass}) for the stability of atomic nuclei (see \cite{choksi2017old} for a review on the problem). 
Such problem has been studied independently in \cite{knupfer2013isoperimetric,muratov2014isoperimetric} up to dimension $7$ and in \cite{julin2014isoperimetric} for any dimension and a Coulombic potential (i.e. $\alpha=1$). In all those papers, the authors prove that, up to a critical volume, the ball is the minimizer of the mixed energy described by Gamow's liquid drop model. In \cite{bonacini2014local} the same result is obtained for general Riesz potential in any dimension. On the other hand, a non-existence result in the Coulombian case in dimension $3$ has been shown in \cite{knupfer2013isoperimetric,muratov2014isoperimetric,lu2014nonexistence} for sufficiently big volumes. Let us recall that in \cite{bonacini2014local} it is shown that there exists a critical volume $m_1$ such that the ball is the unique minimizer of the mixed energy for $|E|<m_1$ and is not a minimizer for $|E|>m_1$, while in \cite{lu2014nonexistence} another critical volume $m_2$ is found such that $\fE$ does not admit a minimizer for $m>m_2$. From now on, let us denote by $m_1$ and $m_2$ the optimal constants such that for $|E|<m_1$ the minimizer is a ball and for $|E|>m_2$ the minimizer does not exist. In \cite{choksi2011small} the two constants $m_1$ and $m_2$ are conjectured to be equal. Up to now, it has been shown in \cite{frank2015compactness} that, as $m \le m_2$, a minimizer still exists (but it is not proved to be a ball). The best known lower bound on $m_2$ has been achieved in \cite{frank2016nonexistence}. The non-existence results have been generalized for different values of $\alpha \in (0,N)$. As far as we know, the best bound on $\alpha$ for which such result holds has been provided in \cite{frank2021nonexistence}.\\
In \cite{figalli2015isoperimetry}, the isoperimetric inequality for fractional perimeter has been used to prove the existence of a minimizer as $|E|<m_2$, with some critical volume $m_2>0$, for the functional
\begin{equation*}
\fE(E)=P_s(E)+V_\alpha(E),
\end{equation*}
where, for $s \in (0,1)$,
\begin{equation}\label{fracper}
P_s(E)=\frac{1-s}{\omega_{N-1}}\int_{E}\int_{\R^N\setminus E}\frac{1}{|x-y|^{N+s}}dxdy
\end{equation}
is the fractional perimeter and $\alpha \in (0,N)$, while, for $s=1$, $P_1:=P$ is the classical perimeter. Moreover, the authors prove that there exists a critical volume $m_1$ such that the unique minimizer of $\fE(E)$ is a ball if $|E|<m_1$. We remark that their results provide uniform bounds for $s$ varying in a compact subset of $(0,1]$. Let us also recall that some variations of the liquid drop problem in presence of an attractive term have been also considered in \cite{lu2015isoperimetric,la2019isoperimetric,frank2018ionization}.\\

Another similar problem, with both an attractive and a repulsive term, is given by the spherical flocking model. Such problem consists in the minimization of the mixed energy
\begin{equation*}
\fE(E)=\fG_\beta(E)+V_\alpha(E),
\end{equation*}
where
\begin{equation}\label{Gbeta}
\fG_\beta(E)=\int_{E}\int_{E}|x-y|^\beta dxdy
\end{equation}
for $\beta>0$, with fixed volume $|E|=m>0$. A first approach to such problem, for $\beta=1$, is given in \cite{burchard2015nonlocal}. As for the liquid drop model, the energy functional in this case presents a repulsive term $V_\alpha$ which is maximized by the ball and an attractive term $\fG_\beta$ that is minimized by the ball (by Riesz's rearrangement inequality, see also \cite{pfiefer1990maximum} for more general inequalities of such type). The interaction between attractive and repulsive potentials has been widely studied. The existence of a global minimizer for the energy $\fE$ when it is extended to $L^1$ functions has been proved in \cite{choksi2015minimizers}. On the other hand, in \cite{burchard2015nonlocal}, the existence of a minimizer is shown for $|E|>m_1$ where $m_1$ is a certain critical volume, while it is also shown that if $|E|<m_0$, for a certain critical volume $m_0>0$, $\fE$ does not admit any minimizer. Finally, in \cite{frank2019proof}, it has been shown that if $|E|>m_2$, for a certain critical mass $m_2>0$, then the ball is a global minimizer of $\fE$ for $\alpha \in (1,N)$, while the ball is not even a critical point if $\alpha \in (0,1]$. Last result is proved by using Christ's theorem \cite{christ2017sharpened} to achieve quantitative versions of the inequalities involving the Riesz potential $V_\alpha$ and the shape functional $\fG_\beta$. \\

Another strategy to prove a quantitative version of the isoperimetric inequality involving the Riesz potential $V_\alpha$ has been exploited in \cite{fusco2019sharp}. Such strategy is made up of two main steps:
\begin{itemize}
	\item Proving the result for nearly-spherical sets (this part is common with \cite{frank2019proof} and has been exploited in \cite{frank2019note});
	\item Using mass transportation arguments to show that, for any measurable set $E$, $V_\alpha(E)$ can be increased by \textit{transporting} it into a nearly spherical set $E_*$ with the same mass.
\end{itemize}
In this paper we want to prove, by using a similar strategy, a quantitative version of the isoperimetric inequality concerning $\fG_\beta$:
\begin{equation}\label{isopineq}
\fG_\beta(E)\ge \fG_\beta(B[|E|]),
\end{equation}
where $B[|E|]$ is a ball with the same volume as $E$. In particular, for a measurable set $E$ with $|E|=\omega_N$, denoting the Fraenkel asymmetry by $$\delta(E)=\min_{x \in \R^N}|E \Delta B(x)|,$$ where $B(x)$ is the unit ball centered in $x$, and the deficit by $$\fD_\beta(E)=\fG_\beta(E)-\fG_\beta(B),$$
where $B$ is a unit ball, we propose a different proof of the following result (that is \cite[Theorem $5$]{frank2019proof}):
\begin{thm}\label{thm:quantrandisop}
	For any $N \ge 2$ and $\beta>0$ there exists a constant $C(N,\beta)>0$ such that for any measurable set $E \subseteq \R^N$ with $|E|=\omega_N$ it holds
	\begin{equation}\label{quantisop}
	\delta(E)\le C(N,\beta)\sqrt{\fD_\beta(E)}.
	\end{equation}
\end{thm}
Although we follow closely the strategy of \cite{fusco2019sharp}, the generalization to our case is not achieved via a direct \textit{mirror-symmetric} argument. Indeed, there are two main difficulties that arise in such adaptation. First, a preliminary detailed study of the properties of the eigenvalues of a Marchaud-type integral defined on the sphere is needed. Such a study, as far as we know, was not available in the literature. Let us recall that for the Riesz operator and the hypersingular Riesz operator on the sphere, the spectral study is carried on in full details in \cite[Chapter $6$]{samko2001hypersingular} while the case $\beta=0$ is considered in \cite[Section $5$]{frank2019note}. The latter is used to obtain the quantitative version of the Riesz inequality in \cite{frank2019proof}. On the other hand, both the Riesz operator and the hypersingular Riesz operator on the sphere have been used in \cite{figalli2015isoperimetry} and \cite{fusco2019sharp} to prove, respectively, the fractional isoperimetric inequality and Riesz inequality with the Riesz potential in the nearly-spherical setting. In these cases, the authors rely on the monotonicity of the eigenvalues of the involved operators. However, the ones of the Marchaud-type integral are not generally monotone. Hence, not only we have to explicitly evaluate them (by means of a standard application of the Funk-Hecke formula), but we also need to better understand their behaviour to overcome the lack of monotonicity. Second, in the mass transportation argument an extra difficulty in order to reduce to the nearly spherical case is due to the fact that the integrand of $\fG_\beta$ is an increasing function of the distance between points. Despite most of the arguments of \cite{fusco2019sharp} can be mirrored, some of them are needed without reverting the inequality. This is solved either by proving the desired upper bounds on bounded sets, or by relying on a milder form of weak$^*$ continuity. Clearly, an additional step in the mass transportation argument is required. \\
Let us also stress out that the exponent $1/2$ is sharp, and it is achieved, as in \cite{burchard2015geometric}, by considering sets constructed starting from the unit ball, by removing an annulus whose outer boundary is the unit sphere and then adding an annulus whose inner boundary is the unit sphere and whose volume is the same of the removed part.

With this result in mind, we then consider a mixed energy with  two different attractive terms:
\begin{equation*}
\fE(E)=\fG_\beta(E)+V_\alpha(E)+\varepsilon P_s(E)
\end{equation*}
with $s \in (0,1]$, $\alpha \in (0,N)$ and $\beta>0$. Using the result proved in \cite{frank2019proof}, we have in particular that if $\alpha \in (1,N)$, there exists a critical mass $m_1>0$ such that, for any $\varepsilon>0$ and any $m>m_1$, $\fE$ admits a minimizer among measurable sets with fixed volume $m$. Moreover, there exists another critical mass $m_2>m_1$ such that, if $m>m_2$, the ball is the unique minimizer of $\fE$ under the volume constraint. On the other hand, by using the result proved in \cite{figalli2015isoperimetry} we have that, for any $\alpha \in (0,N)$ and $\varepsilon>0$, there exists a mass $m_0>0$ such that if $m<m_0$ the ball is a minimizer of $\fE$. Here, we complete the above picture by proving the following result:
\begin{thm}\label{thm:minmix1}
	Let $N \ge 2$, $\alpha \in (0,N)$, $\beta>0$, $s \in (0,1]$. Then there exists a positive constant $m_1>0$ (depending on $\alpha,\beta,s,N$) such that for any $m > m_1$ the ball $B[m]$ of volume $m$ is the minimizer of the shape functional
	\begin{equation*} 
	\fE=\fG_\beta+V_\alpha+\varepsilon P_s,
	\end{equation*}
	under the volume constraint $|E|=m$ for $\varepsilon \ge \varepsilon_0>0$, where $\varepsilon_0$ depends on $m,\beta,s,N$.
\end{thm}
To prove such result, we will follow the line of \cite{figalli2015isoperimetry}. \\

After some preliminaries on trasport maps, given in Section \ref{Sec2}, next section is devoted to a thorough study of the eigenvalue problem for a Marchaud-type fractional integral on the sphere. Then, Section \ref{Sec4} contains a stability result for the functional $\fG_\beta$ in the case of nearly-spherical sets, while in Section \ref{Sec5} we complete the proof of Theorem \ref{thm:quantrandisop} using a mass transportation argument. Finally, in Section \ref{Sec6}, we prove Theorem \ref{thm:minmix1} for the aforementioned mixed energy.
\section{Preliminaries and notations on transport maps}\label{Sec2}
Let us introduce some notions that will be useful in what follows. First of all, let us recall the definition of transportation map between two probability measure, as given in \cite{villani2008optimal}.
\begin{defn}
	Let $(\cX,\mu)$, $(\cY,\nu)$ and $(\Omega, \bP)$ be probability spaces. A coupling of $\mu$ and $\nu$ on $(\Omega, \bP)$ is any random variable $(X,Y)$ defined on $(\Omega \times \Omega, \bP \times \bP)$ with $X:\Omega \to \cX$ and $Y:\Omega \to \cY$ such that the law of $X$ is $\mu$ and the law of $Y$ is $\nu$.\\
	A coupling is said to be deterministic if there exists a measurable function $T:\cX \to \cY$ such that $T(X)=Y$ and $\nu=T_{\sharp}\mu$ (where $T_{\sharp}$ is the push-forward of measures induced by $T$). In such case, $T$ is said to be a transport map between $\mu$ and $\nu$.\\
	Without referring to random variables, this definition can be easily extended to any couple of finite measure spaces $(\cX,\mu)$ and $(\cY,\nu)$.
\end{defn}
A particular case is given by finite measures on $\R^N$ (for some $N$) that are absolutely continuous with respect to the Lebesgue measure. Let us denote by $\cL^N$ the Lebesgue measure, while we will denote as $|E|$ the measure of any Borel set $E$. Let us consider then two non-negative $L^1$ Borel functions $f,g$ and the measures $\mu=fd\cL^N$ and $\nu=gd\cL^N$.\\
We say that a Borel function $T:\R^N \to \R^N$ is a transport map between $f$ and $g$ if it is a transport map between $\mu$ and $\nu$. In particular $T$ is a transport map between $f$ and $g$ if and only if for any continuous non-negative function $\varphi:\R^N \to \R$ it holds
\begin{equation}\label{changeofvar}
\int_{\R^N}\varphi(z)g(z)dz=\int_{\R^N}\varphi(T(y))f(y)dy.
\end{equation}
Formula \eqref{changeofvar} is called a change of variable formula induced by the transport map $T$.\\
Finally, we say that for any two finite measure Borel sets $H$ and $K$ of equal measure, $T:\R^N \to \R^N$ is a transport map between $H$ and $K$ if and only if it is a transport map between the characteristic functions $f=\chi_{K}$ and $g=\chi_{H}$. Actually, $T$ is also a transport map between the uniform distributions on $K$ and $H$, as they are defined as $\bP_K:=\frac{\chi_{K}}{|K|}d\cL^N$ and $\bP_H:=\frac{\chi_{H}}{|H|}d\cL^N$.\\
Now let us focus on invertibility of transport maps, referring to \cite{ambrosio2003lecture}
\begin{defn}
A transport map $T:\R^N \to \R^N$ between two non-negative $L^1$ Borel functions $f,g$ is said to be invertible if there exists a Borel function $T^{-1}:\R^N \to \R^N$ with the property that $T(T^{-1}(z))=z$ for almost any point $z$ such that $g(z)>0$ and $T^{-1}(T(y))=y$ for almost any point $y$ such that $f(y)>0$.\\
Given any two non-negative $L^1$ Borel functions $f,g$, there always exists at least one invertible transport map (see Sudakov's Theorem in \cite{ambrosio2003lecture}).
\end{defn}
Let us give an example of construction of a transport map between two finite measures on $\R^2$. Let us consider $\mu$ and $\nu$ two finite measures on $\R^2$ such that $\mu$ is absolutely continuous with respect to $\cL^2$. First one can define the first marginal of $\mu$ and $\nu$ as, for any Borel set $A \subseteq \R$,
\begin{equation*}
\mu_1(A)=\mu(A \times \R) \qquad \nu_1(A)=\nu(A \times \R).
\end{equation*}
Since $\mu$ is absolutely continuous with respect to $\cL^2$, then $\mu_1$ does not admit atoms. Let us consider the distribution functions $F_1$ and $G_1$ of $\mu_1$ and $\nu_1$ and let us construct $T_1=G_1^{\leftarrow}\circ F_1$, where $G_1^{\leftarrow}$ is the right-continuous inverse of $G_1$. Then $T_1$ is a transport map between $\mu_1$ and $\nu_1$ called the increasing rearrangement of $\mu_1$ over $\nu_1$.\\
As next step, by disintegration of measures (see, for instance, \cite{dellacherie1980probabilities}), we can construct the families of conditional measures $\mu_2(\cdot|x_1)$ and $\nu_2(\cdot|y_1)$ such that for any Borel set $A \subseteq \R^2$, denoting by $A_{x_1}=\{x_2 \in \R: \ (x_1,x_2)\in A\}$ the section of $A$ for fixed $x_1$,
\begin{equation*}
\mu(A)=\int_{\R}\int_{A_{x_1}}\mu_2(dx_2|x_1)\mu_1(dx_1) \qquad \nu(A)=\int_{\R}\int_{A_{y_1}}\nu_2(dy_2|y_1)\nu_1(dy_1).
\end{equation*}
Finally, for any $x_1 \in \R$ fix $y_1=T_1(x_1)$ and define $T_2(\cdot|x_1)$ as the increasing rearrangement of $\mu_2(\cdot|x_1)$ over $\nu_2(\cdot|y_1)$. Then the map
$T(x_1,x_2)=(T_1(x_1),T_2(x_2|x_1))$ is a transport map between $\mu$ and $\nu$, called the Knothe-Rosenblatt rearrangement.\\
This kind of construction can be reproduced also in a more general context than $\R \times \R$, as we will define in what follows a sort of radial Knothe-Rosenblatt rearrangement.
\section{The fractional integral on the sphere and its eigenvalues}\label{Sec3}
As we stated in the introduction, we want to exploit an alternative proof (with respect to \cite{frank2019proof}) to the stability of balls as volume-contrained minimizers of the shape functional $\fG_\beta$ defined in Equation \eqref{Gbeta} for $\beta>0$, which makes use of some Fuglede-type results and mass transportation. As one can see from \cite{figalli2015isoperimetry} and \cite{fusco2019sharp}, some estimates on the eigenvalues of an integral transform on the sphere are mandatory.\\
Let us recall the definition of fractional integral (of conformal type) on a sphere $S^{N-1}\subset \R^N$ as given in \cite{rubin1992inversion}. For any function $u \in L^p(S^{N})$ (in our case $p=+\infty$) we define the fractional integral of $u$ over $S^{N-1}$ of order $\beta+N-1$ as
\begin{equation*}
\cK_\beta[u](\omega)=2^{-\frac{\beta}{2}}\int_{S^{N-1}}|\omega-\xi|^{\beta}u(\xi)d\cH^{N-1}(\xi)
\end{equation*}
for any $\beta \in (1-N,+\infty)$. A study of the eigenvalues of $\cK_\beta$ when $\beta \in (1-N,0)$ has been already carried out in \cite{figalli2015isoperimetry}, so let us focus on the case $\beta>0$. In this case we have the following result.
\begin{prop}
	Let us denote by $\cS_k$ the space of the $k$-th spherical harmonics. Then for any $\beta>0$ and $Y_k \in \cS_k$ with $k \ge 0$ it holds
	\begin{equation*}
	\cK_\beta[Y_k](\omega)=\theta_{k,\beta}Y_k(\omega)
	\end{equation*}
	where
	\begin{equation}\label{thetagen}
	\theta_{k,\beta}=(N-1)\omega_{N-1}(-1)^k\frac{2^{\frac{\beta+2N-4}{2}}\Gamma\left(\frac{\beta+N-1}{2}\right)\Gamma\left(\frac{N-1}{2}\right)\Gamma\left(\frac{\beta+2}{2}\right)}{\Gamma\left(\frac{\beta+2}{2}-k\right)\Gamma\left(\frac{\beta+2N-2}{2}+k\right)}
	\end{equation}
	and $\omega_{N-1}$ is the measure of the $N-1$-dimensional unit ball.
\end{prop}
\begin{proof}
Let us first observe that for any $\omega,\xi \in S^{N-1}$ it holds
\begin{equation*}
|\omega-\xi|^\beta=2^{\frac{\beta}{2}}(1-\omega \cdot \xi)^{\frac{\beta}{2}}
\end{equation*}
so that for any function $u \in L^\infty(S^{N-1})$ it holds
\begin{equation}\label{Kernexpr}
\cK_\beta[u](\omega)=\int_{S^{N-1}}K_\beta(\omega \cdot \xi)u(\xi)d\cH^{N-1}(\xi)
\end{equation}
where
\begin{equation}\label{eq:Kernbeta}
K_\beta(t)=(1-t)^{\frac{\beta}{2}}.
\end{equation}
Since we have expressed $\cK_\beta$ in terms of an integral kernel that depends only on the scalar product between two points of the sphere, we can use Funk-Hecke formula (see, for instance, \cite[Theorem A.34]{rubin2015introduction} or \cite{estrada2017funk}) to state that the eigenfunctions of $\cK_\beta$ are the spherical harmonics and to determine the eigenvalues. Precisely, for any $k \ge 0$ there exists a $\theta_{k,\beta}$ such that for any $Y_k \in \cS_k$ it holds $\cK_\beta[Y_k]=\theta_{k,\beta}Y_k$.\\
First of all, let us determine $\theta_{0,\beta}$. To do this, we can use \cite[Formula $7.311.3$]{ryzhik1965table} to obtain, recalling that $\cS_0$ is the space of constant functions on the sphere
\begin{align*}
\theta_{0,\beta}&=\int_{S^{N-1}}K_\beta(\omega \cdot \xi)d\cH^{N-1}(\xi)\\
&=(N-1)\omega_{N-1}\int_{-1}^{1}(1-t)^{\frac{\beta}{2}}(1-t^2)^{\frac{N-3}{2}}dt\\
&=(N-1)\omega_{N-1}\frac{2^{\frac{\beta+2N-4}{2}}\Gamma\left(\frac{\beta+N-1}{2}\right)\Gamma\left(\frac{N-1}{2}\right)}{\Gamma\left(\frac{\beta+2N-2}{2}\right)}.
\end{align*}
To determine the eigenvalues $\theta_{k,\beta}$ for $k \ge 1$ we have to introduce the spherical polynomials $P_k:[-1,1]\to \R$, $k=0,1,\dots$, defined by the Rodrigues' formula (see \cite[Corollary A.28]{rubin2015introduction})
\begin{equation}\label{eq:sph}
	P_k(t)=\left(-\frac{1}{2}\right)^k \frac{\Gamma\left(\frac{N-1}{2}\right)}{\Gamma\left(k+\frac{N-1}{2}\right)}(1-t^2)^{-\frac{N-3}{2}}\left(\der{}{t}\right)^k (1-t^2)^{k+\frac{N-3}{2}}.
\end{equation}
Precisely, the family $\{P_k, \ k \ge 0\}$ of the spherical polynomials is the unique family of functions on $[-1,1]$ satisfying the following properties (see \cite[Proposition A.27]{rubin2015introduction}):
\begin{enumerate}
	\item $P_k$ is a polynomial of degree $k$;
	\item $P_k(1)=1$;
	\item For any $k_1 \not = k_2$, $P_{k_1}$ and $P_{k_2}$ are orthogonal with respect to the weigthed Lebesgue measure $(1-t^2)^{\frac{N-3}{2}}dt$, i.e.
	\begin{equation*}
		\int_{-1}^{1}P_{k_1}(t)P_{k_2}(t)(1-t^2)^{\frac{N-3}{2}}dt=0.
	\end{equation*}
\end{enumerate}
The eigevalues $\theta_{k,\beta}$ are then achieved via the following formula (see \cite[Equation $(A.7.2)$]{rubin2015introduction})
\begin{equation}\label{eq:FH}
	\theta_{k,\beta}=(N-1)\omega_{N-1}\int_{-1}^{1}K_\beta(t)(1-t^2)^{\frac{N-3}{2}}P_k(t)dt.
\end{equation}
To evaluate the previous integral, we need to distinguish between the cases $N=2$ and $N \ge 3$. Let us start with the second one. In this case, let us introduce the Gegenbauer polynomials $C_k^\lambda:[-1,1] \to \R$, for $k=0,1,\dots$ and $\lambda>0$, defined by the Rodrigues' formula (see \cite[Equation $(1.6.9)$]{rubin2015introduction})
\begin{equation*}
	C_k^\lambda(t)=\left(-\frac{1}{2}\right)^k\frac{\Gamma(2\lambda+k)\Gamma\left(\lambda+\frac{1}{2}\right)}{k!\Gamma(2\lambda)\Gamma\left(\lambda+k+\frac{1}{2}\right)}(1-t^2)^{\frac{1}{2}-\lambda}\left(\der{}{t}\right)^k(1-t^2)^{\lambda+k-\frac{1}{2}}.
\end{equation*}
Comparing the previous formula with Formula \eqref{eq:sph} we get (see also \cite[Equation (A.6.13)]{rubin2015introduction})
\begin{equation*}
	P_{k}(t)=\frac{k!(N-3)!}{(N+k-3)!}C_k^{\frac{N-2}{2}}(t).
\end{equation*}
By using the relation $C_k^{\frac{N-2}{2}}(-t)=(-1)^kC_k^{\frac{N-2}{2}}(t)$ (see \cite[Equation $(1.6.10)$]{rubin2015introduction}) and \cite[Formula $7.311.3$]{ryzhik1965table} we achieve
\begin{align*}
	\theta_{k,\beta}&=(N-1)\omega_{N-1}\frac{k!(N-3)!}{(N+k-3)!}\int_{-1}^{1}(1-t)^{\frac{\beta}{2}}(1-t^2)^{\frac{N-3}{2}}C_k^{\frac{N-2}{2}}(t)dt\\
	&=(N-1)\omega_{N-1}(-1)^k\frac{2^{\frac{\beta+2N-4}{2}}\Gamma\left(\frac{\beta+N-1}{2}\right)\Gamma\left(\frac{N-1}{2}\right)\Gamma\left(\frac{\beta+2}{2}\right)}{\Gamma\left(\frac{\beta+2}{2}-k\right)\Gamma\left(\frac{\beta+2N-2}{2}+k\right)}.
\end{align*}
Concerning the case $N=2$, let us define the Chebyshev polynomials $T_k(t)$ by the Rodrigues' formula (see \cite[Equation $(1.6.16)$]{rubin2015introduction})
\begin{equation*}
	T_k(t)=\left(-\frac{1}{2}\right)^k \frac{\sqrt{\pi}}{\Gamma\left(k+\frac{1}{2}\right)}\sqrt{1-t^2}\left(\der{}{t}\right)(1-t^2)^{m-\frac{1}{2}}.
\end{equation*}
Again, comparing the previous formula with Formula \eqref{eq:sph} we get $P_k(t)=T_k(t)$ (see also \cite[Equation (A.6.13)]{rubin2015introduction}). Now let us consider $\beta \not = 2m$ for any $m \in \N$. By using \cite[Formula $7.354.6$]{ryzhik1965table} one achieves
\begin{align}\label{thetaN2}
\begin{split}
\theta_{k,\beta}&=2\pi\int_{-1}^{1}\frac{(1-t)^{\frac{\beta}{2}}}{\sqrt{1+t^2}}T_k(t)dt\\
&=2\pi \frac{\sqrt{\pi}2^{\frac{\beta}{2}}\Gamma\left(\frac{\beta+1}{2}\right)}{\Gamma\left(\frac{\beta+2}{2}\right)}{}_4F_3\left(-k,k,\frac{\beta+1}{2},\frac{\beta+2}{2};\frac{1}{2},\frac{\beta+2}{2},\frac{\beta+2}{2};1\right)
\end{split}
\end{align}
where, according to \cite[Formula $9.14.1$]{ryzhik1965table}, ${}_pF_q$ is the generalized hypergeometric series defined as
\begin{equation*}
{}_pF_q(a_1,\dots,a_p;b_1,\dots,b_q;z)=\sum_{j=0}^{\infty}\frac{\prod_{h=1}^{p}(a_h)_{j}}{\prod_{h=1}^{q}(b_h)_{j}}\frac{z^j}{j!}
\end{equation*}
and $(a)_j$ is the Pochhammer symbol defined as
\begin{equation*}
	(a)_j=\frac{\Gamma(a+j)}{\Gamma(a)}.
\end{equation*}
By using the definition, it is easy to check that
\begin{equation*}
	{}_4F_3\left(-k,k,\frac{\beta+1}{2},\frac{\beta+2}{2};\frac{1}{2},\frac{\beta+2}{2},\frac{\beta+2}{2};1\right)={}_3F_2\left(-k,k,\frac{\beta+1}{2};\frac{\beta+2}{2},\frac{1}{2};1\right).
\end{equation*}
Now, by Saalschutz's Theorem (see \cite[Section $2.3.1$]{slater1966generalized}) and then, by \cite[Formula $2.3.2.5$]{slater1966generalized} we have
\begin{equation*}
{}_3F_2\left(-k,k,\frac{\beta+1}{2};\frac{\beta+2}{2},\frac{1}{2};1\right)=\frac{\left(\frac{1}{2}\right)_k\left(1+\frac{\beta}{2}-k\right)_k}{\left(\frac{1}{2}-k\right)_k\left(1+\frac{\beta}{2}\right)_k}=\frac{\left(-\frac{\beta}{2}\right)_k}{\left(\frac{\beta+2}{2}\right)_k}.
\end{equation*}
By definition, it holds
\begin{equation*}
\left(\frac{\beta+2}{2}\right)_k=\frac{\Gamma\left(\frac{\beta+2}{2}+k\right)}{\Gamma\left(\frac{\beta+2}{2}\right)},
\end{equation*}
while, by using Euler's reflection formula, we get
\begin{equation*}
\left(-\frac{\beta}{2}\right)_k=\frac{\Gamma\left(k-\frac{\beta}{2}\right)}{\Gamma\left(-\frac{\beta}{2}\right)}=(-1)^k\frac{\Gamma\left(\frac{\beta+2}{2}\right)}{\Gamma\left(\frac{\beta+2}{2}-k\right)}.
\end{equation*}
Substituting all these equalities back to \eqref{thetaN2} we achieve \eqref{thetagen}. Finally, we can extend the formula to the case $\beta=2m$ for some $m \in \N$ by continuity (setting $1/\infty=0$).
\end{proof}
\begin{rmk}\label{rmkbetaeven}
	Let us observe that for any $N \ge 2$ if $\beta=2m$ for some $m \in \N$ then $\theta_{k,\beta}=0$ for $k\ge m+1$.
\end{rmk}
By exploiting the behaviour of the involved $\Gamma$ functions, we can show a recursive formula for the eigenvalues $\theta_{k,\beta}$.
\begin{prop}\label{prop:recform}
	It holds
	\begin{equation}\label{eq:recform}
		\theta_{k+1,\beta}=\frac{\frac{\beta}{2}-k}{\frac{\beta+2N-2}{2}+k}\theta_{k,\beta}
	\end{equation}
	and then, for any $k \in \N$, one has $|\theta_{k+1,\beta}|\le |\theta_{k,\beta}|$. In particular it holds $|\theta_{k+1,\beta}|\le \theta_{0,\beta}$
\end{prop}
\begin{proof}
	First of all, by Remark \ref{rmkbetaeven}, we have that if $\beta=2m$ for some $m \in \N$ the sequence $\theta_{k,\beta}$ is definitely $0$. Thus let us suppose that $\beta \not = 2m$ for any $m \in \N$. Let us set
	\begin{align}\label{eq:muC}
		\begin{split}
			C_{\beta}&=(N-1)\omega_{N-1}2^{\frac{\beta+2N-4}{2}}\Gamma\left(\frac{\beta+N-1}{2}\right)\Gamma\left(\frac{N-1}{2}\right)\Gamma\left(\frac{\beta+2}{2}\right)\\
			\mu_{k,\beta}&=\frac{1}{\Gamma\left(\frac{\beta+2}{2}-k\right)\Gamma\left(\frac{\beta+2N-2}{2}+k\right)}
		\end{split}
	\end{align}
	in such a way that $\theta_{k,\beta}=(-1)^kC_{\beta}\mu_{k,\beta}$. Hence, it is only necessary to show the recursive formula and the bounds for $\mu_{k,\beta}$. We have
	\begin{equation*}
		\mu_{k+1,\beta}=\frac{1}{\Gamma\left(\frac{\beta+2}{2}-k-1\right)\Gamma\left(\frac{\beta+2N-2}{2}+k+1\right)}=\frac{\frac{\beta}{2}-k}{\frac{\beta+2N-2}{2}+k}\mu_{k,\beta}.
	\end{equation*}
	Moreover, taking the absolute value, we have
	\begin{equation*}
		|\mu_{k+1,\beta}|=\frac{\left|\frac{\beta}{2}-k\right|}{\frac{\beta+2N-2}{2}+k}|\mu_{k,\beta}|
	\end{equation*}
	and we conclude the proof observing that for any $k \in \N$ it holds $\left|\frac{\beta-2k}{2}\right| \le \frac{\beta+2N-2+2k}{2}$.
\end{proof}
Concerning the asymptotics of $\theta_{k,\beta}$, we can show the following result, as a consequence of the recursive formula \eqref{eq:recform}.
\begin{cor}\label{prop:lim}
	It holds $\lim_{k \to +\infty}\theta_{k,\beta}=0$.
\end{cor}
\begin{proof}
	We have
	\begin{multline*}
		\lim_{k \to +\infty}k\left(\frac{|\theta_{k,\beta}|}{|\theta_{k+1,\beta}|}-1\right)=\lim_{k \to +\infty}k\left(\frac{k+\frac{\beta+2N-2}{2}}{k-\frac{\beta}{2}}-1\right)\\=\lim_{k \to +\infty}\frac{k(\beta+N-1)}{k-\frac{\beta}{2}}=\beta+N-1>1,
	\end{multline*}
	that, by Raabe's test (see, for istance, \cite[Page $33$]{bromwich2005introduction}), implies $|\theta_{k,\beta}|\to 0$.
	\end{proof}
\begin{rmk}
	The previous Corollary is, in general, true for any spherical convolution operator (see \cite[Lemma $6.2$]{samko2001hypersingular}).
\end{rmk}
\subsection{A Marchaud-type fractional integral on the sphere}
Together with $\cK_\beta$, we can consider also the fractional integral on the sphere defined as
\begin{equation*}
\cI_\beta[u](\omega)=2\int_{S^{N-1}}|\omega-\xi|^\beta(u(\omega)-u(\xi))d\cH^{N-1}(\xi).
\end{equation*}
Let us observe that for $\beta \in (1-N,0)$, this operator reminds of Marchaud fractional derivative (see \cite{ferrari2018weyl}) so we will refer to $\cI_\beta$ as a Marchaud-type fractional integral.\\
Since we have already determined the eigenvalues of $\cK_\beta$, it is easy to determine the ones of $\cI_\beta$.
\begin{cor}
	Let us denote by $\cS_k$ the space of the $k$-th spherical harmonics. Then, for each $k \ge 0$ and $Y_k \in \cS_k$ it holds
	\begin{equation*}
		\cI_\beta[Y_k](\omega)=\lambda_{k,\beta}Y_k(\omega)
	\end{equation*} 
	where
	\begin{equation*}
	\lambda_{k,\beta}=2^{1+\frac{\beta}{2}}[\theta_{0,\beta}-\theta_{k,\beta}].
	\end{equation*}
	Hence, in particular, it holds $\lim_{k \to +\infty}\lambda_{k,\beta}=2^{1+\frac{\beta}{2}}\theta_{0,\beta}$.
\end{cor}
\begin{proof}
	Just observe that for any $u \in L^\infty(S^{N-1})$ it holds
	\begin{equation*}
	\cI_\beta[u](\omega)=2^{1+\frac{\beta}{2}}(\theta_{0,\beta}u(\omega)-\cK_\beta[u](\omega)).
	\end{equation*}
	\end{proof}
Now, by using the recursive formula proved in Proposition \ref{prop:recform} for $\mu_{k,\beta}$ defined as in \eqref{eq:muC}, we obtain a formula to obtain an expression of $\lambda_{k,\beta}$ in terms of $\mu_{0,\beta}$.
\begin{lem}\label{lem:eigen}
	Defining $\widetilde{C}_{\beta}=2^{1+\frac{\beta}{2}}C_\beta$, where $C_\beta$ is defined as in \eqref{eq:muC}, it holds
	\begin{equation}\label{eq:eigen}
	\lambda_{k,\beta}=\widetilde{C}_\beta\left(1+(-1)^{k+1}\prod_{j=0}^{k-1}\frac{\beta-2j}{\beta+2N-2+2j}\right)\mu_{0,\beta}.
	\end{equation}
\end{lem} 
\begin{proof}
	By the recursive formula proved in Proposition \ref{prop:recform}, we have
	\begin{equation*}
	\mu_{k,\beta}=\prod_{j=0}^{k-1}\frac{\beta-2j}{\beta+2N-2+2j}\mu_{0,\beta}.
	\end{equation*}
	Thus, by recalling that for any $k \ge 0$ it holds $\theta_{k,\beta}=(-1)^{k}C_{\beta}\mu_{k,\beta}$ and $\lambda_{k,\beta}=2^{1+\frac{\beta}{2}}(\theta_{0,\beta}-\theta_{k,\beta})$ we complete the proof.
\end{proof}
Now let us argue concerning the supremum of the eigenvalues. We have the following result.
\begin{prop}\label{prop:lowbound}
	For any $\beta>0$ it holds $\lambda_{1,\beta}=\max_{k \ge 0}\lambda_{k,\beta}$. Moreover, for any $k \ge 2$ it holds
	\begin{equation}\label{lambdacontr}
	\lambda_{1,\beta}-\lambda_{k,\beta}\ge D_{\beta}>0
	\end{equation}
	where
	\begin{equation}\label{eq:Dbeta}
	D_\beta=\frac{\beta \mu_{0,\beta} \widetilde{C}_\beta}{\beta+2N-2}\begin{cases} 1-\frac{2-\beta}{\beta+2N} & \beta \in (0,2)\\
	1 & \beta \in [2,4]\\
	1-\frac{(\beta-2)(\beta-4)}{(\beta+2N)(\beta+2N+2)} & \beta>4
	\end{cases}
	\end{equation}
and $\widetilde{C}_\beta$ is defined as in Lemma \ref{lem:eigen}.
\end{prop}
\begin{proof}
	Let us first observe that $\lambda_{0,\beta}=0$ for any $\beta>0$. Moreover we also have $\mu_{0,\beta}>0$ and $\widetilde{C}_\beta>0$, thus
	\begin{equation*}
	\lambda_{1,\beta}=\widetilde{C}_\beta\left(1+\frac{\beta}{\beta+2N-2}\right)\mu_{0,\beta}>0
	\end{equation*}
	Moreover, for any $k \ge 2$ we have, recalling \eqref{eq:eigen},
	\begin{align*}
	\lambda_{1,\beta}-\lambda_{k,\beta}&=\widetilde{C}_\beta\frac{\beta}{\beta+2N-2}\mu_{0,\beta}\left(1-(-1)^{k+1}\prod_{j=1}^{k-1}\frac{\beta-2j}{\beta+2N-2+2j}\right)\\&\ge \widetilde{C}_\beta\frac{\beta}{\beta+2N-2}\mu_0^\beta\left(1-\prod_{j=1}^{k-1}\frac{|\beta-2j|}{\beta+2N-2+2j}\right)\ge 0,
	\end{align*}
	since $\frac{|\beta-2j|}{\beta+2N-2+2j}\le 1$ for any $j \le k-1$. Thus we have that $\lambda_{1,\beta}=\max_{k \ge 0}\lambda_{k,\beta}$.\\
	Now let us show relation \eqref{lambdacontr}. First of all, let us work with $\beta<2$. If $\beta \in (0,2)$ then we can show that $(\lambda_{k,\beta})_{k \ge 1}$ is a decreasing sequence. Indeed, we have, since $\beta-2j<0$,
	\begin{multline*}
	\lambda_{k,\beta}-\lambda_{k+1,\beta}=(-1)^{k+1}\widetilde{C}_\beta\frac{\beta}{\beta+2N-2}\left(\prod_{j=1}^{k-1}\frac{\beta-2j}{\beta+2N-2+2j}\right)\\\times\left(1+\frac{2k-\beta}{\beta+2N-2+2k}\right)\mu_0^\beta\ge 0.
	\end{multline*}
	Hence we have that, for $\beta \in (0,2)$, it holds
	\begin{equation*}
	\lambda_{1,\beta}-\lambda_{k,\beta}\ge \lambda_{1,\beta}-\lambda_{2,\beta}=\frac{\beta}{\beta+2N-2}\widetilde{C}_\beta \mu_{0,\beta}\left(1-\frac{2-\beta}{\beta+2N}\right).
	\end{equation*}
	Now let us consider $\beta=2$. We have $\lambda_{k,\beta}=\widetilde{C}_\beta \mu_{0,\beta}$ for any $k \ge 2$ and then
	\begin{equation*}
	\lambda_{1,\beta}-\lambda_{k,\beta}\ge \lambda_{1,\beta}-\widetilde{C}_\beta \mu_{0,\beta}=\frac{1}{N}\widetilde{C}_\beta \mu_{0,\beta}.
	\end{equation*}
	Let us now consider $\beta>2$. Exploiting $\lambda_{2,\beta}$ we have
	\begin{equation*}
	\lambda_{2,\beta}=\widetilde{C}_\beta\left(1-\frac{\beta(\beta-2)}{(\beta+2N)(\beta+2N-2)}\right)\mu_{0,\beta}\le \widetilde{C}_\beta \mu_{0,\beta}.
	\end{equation*}
	Now let us show that the sequence $(\lambda_{k,\beta})_{k \ge 2}$ is increasing. To do this let us observe that
	\begin{multline*}
	\lambda_{k+1,\beta}-\lambda_{k,\beta}=(-1)^{k}\widetilde{C}_\beta\frac{\beta(\beta-2)}{(\beta+2N-2)(\beta+2N)}\left(\prod_{j=2}^{k-1}\frac{\beta-2j}{\beta+2N-2+2j}\right)\\\times\left(1+\frac{2k-\beta}{\beta+2N-2+2k}\right)\mu_0^\beta\ge 0.
	\end{multline*}
	Let us also recall, by Corollary \ref{prop:lim}, that $\lambda_{k,\beta} \to \widetilde{C}_\beta\mu_{0,\beta}$ as $k \to +\infty$ to achieve that, for any $k \ge 2$
	\begin{equation*}
	\lambda_{1,\beta}-\lambda_{k,\beta}\ge \lambda_{1,\beta}-\widetilde{C}_\beta\mu_{0,\beta}=\frac{\beta \mu_{0,\beta}\widetilde{C}_\beta}{\beta+2N-2}.
	\end{equation*}
	Now let us consider $\beta=4$. Then for any $k \ge 3$ it holds $\lambda_{k,\beta}=\widetilde{C}_\beta\mu_{0,\beta}$ obtaining again the previous estimate. Finally, for $\beta>4$, we have
	\begin{equation*}
	\lambda_{3,\beta}=\widetilde{C}_{n,\beta}\left(1+\frac{\beta(\beta+2)(\beta+4)}{(\beta+2N-2)(\beta+2N)(\beta+2N+2)}\right)\mu_{0,\beta}> \widetilde{C}_{n,\beta}\mu_{0,\beta}
	\end{equation*}
	while, with the same strategy as before, we can show that $(\lambda_{k,\beta})_{k \ge 3}$ is a decreasing sequence. Hence we have, for any $k \ge 2$,
	\begin{equation*}
	\lambda_{1,\beta}-\lambda_{k,\beta}\ge \lambda_{1,\beta}-\lambda_{3,\beta}=\frac{\beta \mu_{0,\beta} \widetilde{C}_\beta}{\beta+2N-2}\left(1-\frac{(\beta-2)(\beta-4)}{(\beta+2N)(\beta+2N+2)}\right),
	\end{equation*}
	concluding the proof.
	\end{proof}
\begin{rmk}\label{rmkosc}
	In the proof of the previous Proposition we exploited the initially oscillatory behaviour of the eigenvalues $\lambda_{k,\beta}$ for $\beta \le 4$. To conclude the study of the eigenvalues of $\cI_\beta$, let us show that this initially oscillatory behaviour actually holds for any $\beta>2$ (as opposed to what happens in the case $\beta<0$, see \cite{figalli2015isoperimetry}). Precisely, let $\lambda_{\infty,\beta}:=\widetilde{C}_\beta \mu_{0,\beta}$, where $\widetilde{C}_\beta$ is defined in Lemma \ref{lem:eigen}.  Thus, from Equation \eqref{eq:eigen}, we get
	\begin{equation*}
		\lambda_{k,\beta}-\lambda_{\infty,\beta}=(-1)^{k+1}\mu_{0,\beta}\prod_{j=0}^{k-1}\frac{\beta-2j}{\beta+2N-2+2j}.
	\end{equation*}
	Observing that for $k < \frac{\beta}{2}+1$ it holds $\prod_{j=0}^{k-1}\frac{\beta-2j}{\beta+2N-2+2j} > 0$, we have
	\begin{equation*}
		(-1)^{k}(\lambda_{k,\beta}-\lambda_{\infty,\beta})< 0 , \ \forall k < \frac{\beta}{2}+1,
	\end{equation*}
	that is to say that the first $\lceil \frac{\beta}{2} \rceil$ eigenvalues oscillate around the limit value $\lambda_{\infty,\beta}$. If $\beta=2n$ for some positive integer $n$, then Equation \eqref{eq:eigen} tells us that $\lambda_{k,\beta}=\lambda_{\infty,\beta}$ for any $k\ge \frac{\beta}{2}+1$. Otherwise, for $k>\frac{\beta}{2}+1$, setting $\widetilde{k}(\beta)=\lceil \frac{\beta}{2} \rceil$, one gets
	\begin{multline*}
		\lambda_{k+1,\beta}-\lambda_{k,\beta}=\left((-1)^{\widetilde{k}(\beta)+1}\prod_{j=0}^{\widetilde{k}(\beta)}\frac{\beta-2j}{\beta+2N-2+2j}\right)\\\times \left((-1)^{k-\widetilde{k}(\beta)}\prod_{j=\widetilde{k}(\beta)+1}^{k-1}\frac{\beta-2j}{\beta+2N-2+2j}\right)\left(\frac{2k-\beta}{\beta+2N-2+2k}+1\right).
	\end{multline*}
	Being $\beta-2j<0$ for any $j> \widetilde{k}(\beta)$, it follows that $(-1)^{k-\widetilde{k}(\beta)}\prod_{j=\widetilde{k}(\beta)+1}^{k-1}\frac{\beta-2j}{\beta+2N-2+2j}>0$. Thus, we finally get
	\begin{equation*}
		(-1)^{\widetilde{k}(\beta)+1}(\lambda_{k+1,\beta}-\lambda_{k,\beta})>0,
\end{equation*}
	that implies that the sequence $(\lambda_{k,\beta})_{k > \widetilde{k}(\beta)}$ monotonically (increasing or decreasing depending on the parity of $\widetilde{k}(\beta)$) converges towards $\lambda_{\infty,\beta}$.
\end{rmk}
\section{A Fuglede-type result for $\fG_\beta$}\label{Sec4}
As a first step to prove Theorem \ref{thm:quantrandisop}, we want to obtain a Fuglede-type result for the functional $\fG_\beta$. Let us first recall the definition of nearly-spherical domain (see \cite{fuglede1989stability} for the definition or \cite{fusco2015quantitative} for an almost complete survey on quantitative isoperimetric inequalities).
\begin{defn}\label{def:ns}
	Let us denote by $B$ the unit ball of $\R^N$ centred at the origin. We say that a set $E\subset \R^N$ is nearly spherical if there exists $t \in (0,1)$ and $u \in L^\infty(S^{N-1})$ with $\Norm{u}{L^\infty(S^{N-1})}\le \frac{1}{2}$ such that, up to a translation, $E=E_{t,u}$, where
	\begin{equation}\label{eq:nearlyspherical}
	E_{t,u}=\{x \in \R^N: \ x= \rho z, \ z \in S^{N-1}, \ \rho \in (0,1+tu(z)]\}.
	\end{equation}
\end{defn}
Let us observe that while in the usual definition of nearly spherical set (see \cite{fuglede1989stability}), it is required that $u \in W^{1,\infty}(S^{N-1})$, here we do not need to assume the Lipschitz continuity of the function $u$. Such assumption will be instead needed as we approach the problem with the mixed energy, as we will see later.\\
Concerning nearly-spherical sets, we want to show the following result, where $\fD_\beta(E)=\fG_\beta(E)-\fG_\beta(B)$.
\begin{thm}\label{nrthm}
	Let $\beta>0$. There exist two constants $\varepsilon_0>0$ and $C(N,\beta)>0$ with the property that if $t \in (0,\varepsilon_0)$, $u \in L^\infty(S^{N-1})$ with $\Norm{u}{L^\infty(S^{N-1})}\le \frac{1}{2}$ and $E_{t,u}$ is a nearly spherical set as in \eqref{eq:nearlyspherical}, such that $|E_{t,u}|=\omega_N$ and the baricenter of $E_{t,u}$ is at the origin, then
	\begin{equation*}
	\fD_\beta(E_{t,u})\ge C(N,\beta)t^2\Norm{u}{L^2(S^{N-1})}^2.
	\end{equation*}
\end{thm}
Let us observe that Theorem \ref{nrthm} actually implies Theorem \ref{thm:quantrandisop} in the nearly spherical context. To see this, let us first recall the definition of (non-normalized) Fraenkel asymmetry as, for any measurable set $E$ such that $|E|=\omega_N$, denoting by $B(x)$ the ball of radius $1$ and centre $x \in \R^N$ and by $\Delta$ the symmetric difference,
\begin{equation*}
\delta(E)=\inf_{x \in \R^N}|E \Delta B(x)|
\end{equation*}
Thus, for a nearly spherical set $E_{t,u}$, we have the following chain of inequalities:
\begin{equation*}
\delta(E_{t,u})\le |E_{t,u} \Delta B|=\Norm{tu}{L^1(S^{N-1})} \le t\sqrt{N\omega_N} \Norm{u}{L^2(S^{N-1})},
\end{equation*} 
obtaining the desired implication.\\
Another quantity we will work with is the following semi-norm on $L^\infty(S^{N-1})$: for $\beta>0$ and for any $u \in L^\infty(S^{N-1})$
\begin{equation*}
[u]_\beta^2=\int_{S^{N-1}}\int_{S^{N-1}}|\omega-\xi|^\beta|u(\omega)-u(\xi)|^2d\cH^{N-1}(\omega) d\cH^{N-1}(\xi).
\end{equation*}
Let us observe that for $\beta \in (-N,0)$, this actually reminds of a Besov semi-norm on $S^{N-1}$.\\
The first easy observation we can show gives us the link between this semi-norm and the Marchaud-type fractional integral.
\begin{lem}
	Let $\beta \in (-N,0)$ and $u \in L^\infty(S^{N-1})$. Then
	\begin{equation}\label{sntoint}
	[u]_\beta^2=\int_{S^{N-1}}\cI_\beta[u](\omega)u(\omega)d\cH^{N-1}(\omega)
	\end{equation}
\end{lem}
\begin{proof}
	Let us just observe that
	\begin{align*}
	\int_{S^{N-1}}\int_{S^{N-1}}&|\omega-\xi|^\beta|u(\omega)-u(\xi)|^2d\cH^{N-1}(\omega) d\cH^{N-1}(\xi)\\&=\int_{S^{N-1}}\int_{S^{N-1}}|\omega-\xi|^\beta(u(\omega)-u(\xi))u(\omega)d\cH^{N-1}(\omega) d\cH^{N-1}(\xi)\\
	&+\int_{S^{N-1}}\int_{S^{N-1}}|\omega-\xi|^\beta(u(\xi)-u(\omega))u(\xi)d\cH^{N-1}(\omega) d\cH^{N-1}(\xi)
	\end{align*}
	thus Equation \eqref{sntoint} follows from Fubini's theorem and the definition of $\cI_\beta[u]$.
\end{proof}
To show Theorem \ref{nrthm} we need the following preliminary result that will make use of the aforementioned seminorm.
\begin{lem}
	Fix $\beta>0$. Then it holds
	\begin{equation}\label{lambda1asGB}
	\lambda_{1,\beta}=\frac{(\beta+N)(\beta+2N)}{N\omega_N}\fG_\beta(B).
	\end{equation}
\end{lem}
\begin{proof}
	Let us first observe that it holds
	% by definition of uniform distribution one has, for any compact set $K \subseteq \R^N$
	%\begin{equation}\label{intrep}
	%\cG_\beta(K)=\frac{1}{|K|^2}\int_{K^2}|x-y|^\beta dxdy.
	%\end{equation}
	\begin{equation}\label{eq:eqGbeta}
	\fG_\beta(B)=\omega_N\int_{B}|x-y|^\beta dx
	\end{equation}
	by exploiting the fact that the integral $\int_{B}|x-y|^\beta dx$ is constant with respect to $y$.\\
	Now let us consider $\cS_1$ the space of spherical $1$-harmonic function. A basis for $\cS_1$ is given by the coordinate functions $\omega \mapsto \omega_i$ for $i=1,\dots,N$. In particular, it holds, by Equation \eqref{sntoint},
	\begin{equation*}
	[\omega_i]^2_{\beta}=\lambda_{1,\beta}\int_{S^{N-1}}\omega_i^2d\cH^{N-1}(\omega).
	\end{equation*}
	 On the other hand, by definition,
	 \begin{equation*}
	 [\omega_i]^2_{\beta}=\int_{S^{N-1}}\int_{S^{N-1}}|\omega-\xi|^\beta|\omega_i-\xi_i|^2d\cH^{N-1}(\omega)d\cH^{N-1}(\xi).
	 \end{equation*}
	 Hence, for any $i=1,\dots, N$, we obtain the identity
	 \begin{equation*}
	 \lambda_{1,\beta}\int_{S^{N-1}}\omega_i^2d\cH^{N-1}(\omega)=\int_{S^{N-1}}\int_{S^{N-1}}|\omega-\xi|^\beta|\omega_i-\xi_i|^2d\cH^{N-1}(\omega)d\cH^{N-1}(\xi).
	 \end{equation*}
	 Thus, summing over $i$, we get
	 \begin{equation*}
	 \lambda_{1,\beta}N\omega_N=\int_{S^{N-1}}\int_{S^{N-1}}|\omega-\xi|^{\beta+2}d\cH^{N-1}(\omega)d\cH^{N-1}(\xi).
	 \end{equation*}
	 Now let us define the function $L:S^{N-1}\to \R$ as
	 \begin{equation*}
	 L(\xi)=\int_{S^{N-1}}|\omega-\xi|^{\beta+2}d\cH^{N-1}(\omega).
	 \end{equation*}
	 Let us also define $\ell(z)=\frac{1}{\beta+2}|z|^{\beta+2}$ so that $\nabla \ell(z)=|z|^\beta z$. This leads to
	 \begin{equation}\label{eq:L}
	 L(\xi)=\int_{S^{N-1}}\nabla \ell(\omega-\xi)\cdot \omega d\cH^{N-1}(\omega)-\int_{S^{N-1}}\nabla \ell(\omega-\xi)\cdot \xi d\cH^{N-1}(\omega).
	 \end{equation}
	 Now let us denote by $\nabla_\tau$ the tangential gradient and with $\frac{\partial}{\partial \nu(\omega)}$ the normal derivative. Hence we can split
	 \begin{equation*}
	 \nabla\ell(\omega-\xi)=\nabla_\tau \ell(\omega-\xi)+\frac{\partial \ell}{\partial \nu(\omega)}(\omega-\xi)\omega.
	 \end{equation*}
	 Thus, also recalling that $|\omega|^2=1$ and $\nabla(\omega \cdot \xi)=\xi$, we have, by Equation \eqref{eq:L},
	 \begin{equation}\label{eq:L2}
	 L(\xi)=\int_{S^{N-1}}\frac{\partial \ell }{\partial \nu(\omega)}(\omega-\xi)(1-\omega \cdot \xi) d\cH^{N-1}(\omega)-\int_{S^{N-1}}\nabla_\tau \ell(\omega-\xi)\cdot \nabla_\tau(\omega \cdot\xi) d\cH^{N-1}(\omega).
	 \end{equation}
	 Now we need to study the two integrals separately. Let us define the functions $\cA,\cB:S^{N-1}\to \R$ as
	 \begin{align*}
	\cA(\xi)&=\int_{S^{N-1}}\frac{\partial \ell }{\partial \nu(\omega)}(\omega-\xi)(1-\omega \cdot \xi) d\cH^{N-1}(\omega)\\
	\cB(\xi)&=\int_{S^{N-1}}\nabla_\tau \ell(\omega-\xi)\cdot \nabla_\tau(\omega \cdot\xi) d\cH^{N-1}(\omega).
	\end{align*}
	Concerning $\cA$, we have, by the divergence theorem,
	\begin{equation*}
	\cA(\xi)=\int_{B}\nabla \ell(\omega-\xi)\cdot\nabla(1-\omega\cdot \xi)d\omega+\int_{B}\Delta \ell(\omega-\xi)(1-\omega \cdot \xi)d\omega.
	\end{equation*}
	Observing that $\Delta \ell(\omega-\xi)=(\beta+N)|\omega-\xi|^\beta$, $\nabla (1-\omega \cdot \xi)=-\xi$ and recalling that $\nabla \ell(\omega-\xi)=|\omega-\xi|^\beta(\omega-\xi)$, it holds
	\begin{align*}
	\cA(\xi)&=\int_{B}|\omega-\xi|^{\beta}(1-\omega\cdot \xi)d\omega+(\beta+N)\int_{B}|\omega-\xi|^\beta(1-\omega \cdot \xi)d\omega\\
	&=(\beta+N+1)\int_{B}|\omega-\xi|^\beta(1-\omega \cdot \xi)d\omega.
	\end{align*}
	Concerning $\cB$, by integration by parts on $S^{N-1}$, we have
	\begin{equation*}
	\cB(\xi)=-\int_{S^{N-1}} \ell(\omega-\xi)\Delta_{S^{N-1}}(\omega \cdot \xi)d\cH^{N-1}(\omega)
	\end{equation*}
	where $\Delta_{S^{N-1}}$ is the Laplace-Beltrami operator on $S^{N-1}$. In particular, since the first eigenvalue of $-\Delta_{S^{N-1}}$ is $N-1$ and it is achieved for functions in $\cS_1$, we have $-\Delta_{S^{N-1}}(\omega \cdot \xi)=(N-1)\omega \cdot \xi$. Thus, by also using the definition of $\ell$,
	\begin{equation*}
	\cB(\xi)=\frac{N-1}{\beta+2}\int_{S^{N-1}}|\omega-\xi|^{\beta+2}\omega \cdot \xi d \cH^{N-1}(\omega).
	\end{equation*}
	Hence, by using the fact that $L(\xi)=\cA(\xi)+\cB(\xi)$, we have
	\begin{equation*}
	L(\xi)=(\beta+N+1)\int_{B}|\omega-\xi|^\beta(1-\omega \cdot \xi)d\omega-\frac{N-1}{\beta+2}\int_{S^{N-1}}|\omega-\xi|^{\beta+2}\omega \cdot \xi d\cH^{N-1}(\omega).
	\end{equation*}
	Integrating both sides on $S^{N-1}$ and using Fubini's theorem we have
	\begin{multline*}
	N\omega_N\lambda_{1,\beta}=(\beta+N+1)\int_{B}\int_{S^{N-1}}|\omega-\xi|^\beta(1-\omega \cdot \xi) d\cH^{N-1}(\xi)d \omega\\-\frac{N-1}{\beta+2}\int_{S^{N-1}}\int_{S^{N-1}}|\omega-\xi|^{\beta+2}\omega \cdot \xi d\cH^{N-1}(\omega)d\cH^{N-1}(\xi).
	\end{multline*}
	Concerning the first integral, we achieve, by using the divergence theorem and recalling \eqref{eq:eqGbeta},
	\begin{align*}
	\int_{S^{N-1}}|\omega-\xi|^\beta(1-\omega \cdot \xi) d\cH^{N-1}(\xi)&=\int_{S^{N-1}}\nabla \ell(\xi-\omega)\cdot \xi d\cH^{N-1}(\xi)\\
	&=\int_{B}\Delta \ell(\xi-\omega)d\xi\\
	&=(\beta+N)\int_{B}|\omega-\xi|^\beta d\xi=\frac{(\beta+N)}{\omega_N}\fG_\beta(B)
	\end{align*}
	and then
	\begin{multline}\label{eq:L4}
	\lambda_{1,\beta}=\frac{(\beta+N)(\beta+N+1)}{N\omega_N}\fG_\beta(B)\\-\frac{N-1}{N\omega_N(\beta+2)}\int_{S^{N-1}}\int_{S^{N-1}}|\omega-\xi|^{\beta+2}\omega \cdot \xi d\cH^{N-1}(\omega)d\cH^{N-1}(\xi).
	\end{multline}
	Now we need to evaluate the second integral in terms of $\fG_\beta(B)$. To do this, let us set $G_1(z)=|z|^\beta$ and, observing that $\nabla G_1(z)=\beta |z|^{\beta-2}z$, we achieve, by using the divergence theorem,
	\begin{align*}
	\int_{B}|\omega-\xi|^{\beta}d\omega&=\frac{1}{\beta}\int_{B}\nabla G_1(\omega-\xi)\cdot(\omega-\xi)d\omega\\
	&=-\frac{1}{\beta}\int_{B} G_1(\omega-\xi)\divg(\omega-\xi)d\omega\\
	&\quad+\frac{1}{\beta}\int_{S^{N-1}}G_1(\omega-\xi)(\omega-\xi)\cdot \omega d\cH^{N-1}(\omega)\\
	&=-\frac{N}{\beta}\int_{B}|\omega-\xi|^{\beta}+\frac{1}{\beta}\int_{S^{N-1}}|\omega-\xi|^{\beta}(\omega-\xi)\cdot \omega d\cH^{N-1}(\omega).
	\end{align*}
	Integrating both sides in $B$, multiplying by $\beta$ and using Fubini's theorem, we have
	\begin{equation*}
	(\beta+N)\fG_\beta(B)=\int_{S^{N-1}}\int_{B}|\xi-\omega|^\beta(\omega-\xi)\cdot \omega d\xi d\cH^{N-1}(\omega).
	\end{equation*}
	Setting $G_2(z)=|z|^{\beta+2}$ and arguing as before, by the divergence theorem, we get
	\begin{align*}
	\int_{B}|\xi-\omega|^\beta(\omega-\xi)\cdot \omega d\xi&=-\frac{1}{\beta+2}\int_{B}\nabla G_2(\xi-\omega)\cdot \omega d\xi\\
	&=-\frac{1}{\beta+2}\int_{S^{N-1}}G_2(\xi-\omega)\xi\cdot \omega d\cH^{N-1}(\xi)\\
	&=-\frac{1}{\beta+2}\int_{S^{N-1}}|\xi-\omega|^{\beta+2}\xi\cdot \omega d\cH^{N-1}(\xi).
	\end{align*}
	Hence we finally obtain
	\begin{equation}\label{eq:L3}
	-(\beta+2)(\beta+N)\fG_\beta(B)=\int_{S^{N-1}}\int_{S^{N-1}}|\xi-\omega|^{\beta+2}\xi\cdot \omega d\cH^{N-1}(\xi)d\cH^{N-1}(\omega).
	\end{equation}
	Formula \eqref{lambda1asGB} follows by using identity \eqref{eq:L3} in Equation \eqref{eq:L4}.
\end{proof}
Now we are ready to prove the main Theorem of this section.
\begin{proof}[Proof of Theorem \ref{nrthm}]
	Since in the rest of the proof $u$ will be fixed, we shall simply write $E_t$ instead of $E_{t,u}$. Let $E_t$ be the nearly spherical set in consideration. Then we can write, by coarea formula,
	\begin{align*}
	\fG_\beta(E_t)&=\int_{E_t}\int_{E_t}|x-y|^\beta dxdy\\
	&=\int_{B}\int_{B}\left|x\left(1+tu\left(\frac{x}{|x|}\right)\right)-y\left(1+tu\left(\frac{y}{|y|}\right)\right)\right|^\beta \\&\qquad \times \left(1+tu\left(\frac{x}{|x|}\right)\right)^{N}\left(1+tu\left(\frac{y}{|y|}\right)\right)^{N} dxdy\\
	&=\int_{S^{N-1}} \int_{S^{N-1}} \int_0^1  \int_0^1 |r\omega(1+tu(\omega))-\rho \xi(1+tu(\xi))|^\beta \\
	&\qquad \times(1+tu(\omega))^{N}(1+tu(\xi))^Nr^{N-1}\rho^{N-1} d\rho drd\cH^{N-1}(\xi)d \cH^{N-1}(\omega).
	\end{align*}
	Concerning the one-dimensional integrals, we can use the change of variables $\bar{r}=r(1+tu(\omega))$ and $\bar{\rho}=\rho(1+tu(\xi))$ to achieve
	\begin{multline*}
	\fG_\beta(E_t)=\int_{S^{N-1}} \int_{S^{N-1}} \int_0^{1+tu(\omega)} \int_0^{1+tu(\xi)} |\bar{r}\omega-\bar{\rho} \xi|^\beta\\ \times \bar{r}^{N-1}\bar{\rho}^{N-1} d\bar{\rho} \, d \bar{r} \, d\cH^{N-1}(\xi) \, d \cH^{N-1}(\omega).
	 \end{multline*}
	 Recalling that $\omega^2=\xi^2=1$, it is easy to check that $|\bar{r}\omega-\bar{\rho}\xi|^2=|\bar{r}-\bar{\rho}|^2+\bar{r}\bar{\rho}|\omega-\xi|^2$ and then
	 \begin{multline}\label{GbetaEt}
	 \fG_\beta(E_t)=\int_{S^{N-1}} \int_{S^{N-1}} \int_0^{1+tu(\omega)} \int_0^{1+tu(\xi)} (|\bar{r}-\bar{\rho}|^2+\bar{r}\bar{\rho}|\omega-\xi|^2)^\frac{\beta}{2}\\ \times \bar{r}^{N-1}\bar{\rho}^{N-1} d\bar{\rho} \, d \bar{r} \, d\cH^{N-1}(\xi) \, d \cH^{N-1}(\omega).
	 \end{multline}
	 Now let us recall that for any $a,b>0$ and any symmetric function $f(\bar{r},\bar{\rho})$ one has
	 \begin{equation*}
	 2\int_0^a\int_0^b f(\bar{r},\bar{\rho})d\bar{\rho}d\bar{r}=\int_0^a\int_0^a f(\bar{r},\bar{\rho})d\bar{\rho}d\bar{r}+\int_0^b\int_0^b f(\bar{r},\bar{\rho})d\bar{\rho}d\bar{r}-\int_a^b\int_a^b f(\bar{r},\bar{\rho})d\bar{\rho}d\bar{r}.
	 \end{equation*}
	 The integrand in \eqref{GbetaEt} is symmetric in $\bar{r}$ and $\bar{\rho}$ thus we can apply the previous decomposition to achieve
	 \begin{align*}
	 \fG_\beta(E_t)&=\int_{S^{N-1}} \int_{S^{N-1}} \int_0^{1+tu(\omega)} \int_0^{1+tu(\omega)} (|\bar{r}-\bar{\rho}|^2+\bar{r}\bar{\rho}|\omega-\xi|^2)^\frac{\beta}{2}\\ &\qquad \times \bar{r}^{N-1}\bar{\rho}^{N-1} d\bar{\rho} \, d \bar{r} \, d\cH^{N-1}(\xi) \, d \cH^{N-1}(\omega)\\
	 &\qquad -\frac{1}{2}\int_{S^{N-1}} \int_{S^{N-1}} \int_{1+tu(\omega)}^{1+tu(\xi)} \int_{1+tu(\omega)}^{1+tu(\xi)} (|\bar{r}-\bar{\rho}|^2+\bar{r}\bar{\rho}|\omega-\xi|^2)^\frac{\beta}{2}\\ &\qquad \times \bar{r}^{N-1}\bar{\rho}^{N-1} d\bar{\rho} \, d \bar{r} \, d\cH^{N-1}(\xi) \, d \cH^{N-1}(\omega):=\cA_1(E_t)+\cA_2(E_t).
	 \end{align*}
	 Let us work with $\cA_1(E_t)$. Applying again the change of variables $\bar{r}=r(1+tu(\omega))$ and $\bar{\rho}=\rho(1+tu(\xi))$, we obtain
	 \begin{multline*}
	 \cA_1(E_t)=\int_{S^{N-1}}(1+tu(\omega))^{\beta+2N} \int_{S^{N-1}} \int_0^{1} \int_0^{1} (|r-\rho|^2+r\rho|\omega-\xi|^2)^\frac{\beta}{2}\\
	 \times r^{N-1}\rho^{N-1} d\rho \, d r \, d\cH^{N-1}(\xi) \, d \cH^{N-1}(\omega).
	 \end{multline*}
	 Observing that $|\omega-\xi|^2=2(1-\omega \cdot \xi)$ we set
	 \begin{align*}
	 K(\omega \cdot \xi)&:=\int_0^{1}\int_0^{1} (|r-\rho|^2+r\rho|\omega-\xi|^2)^\frac{\beta}{2} r^{N-1}\rho^{N-1} d\rho \, d r,\\
	 \gamma&:=\int_{S^{N-1}}K(\omega \cdot \xi)d\cH^{N-1}(\xi)
	 \end{align*}
	 and observe that $\gamma$ is a constant with respect to $\omega$ to achieve
	 \begin{equation*}
	 \cA_1(E_t)=\gamma\int_{S^{N-1}}(1+tu(\omega))^{\beta+2N}d\cH^{N-1}(\omega).
	 \end{equation*}
	 Now we need to evaluate $\gamma$. To do this, let us observe that $\cA_2(B)=0$ and then it holds
	 \begin{equation*}
	 \fG_\beta(B)=\cA_1(B)=N\gamma \omega_N
	 \end{equation*}
	 and then $\gamma=\frac{\fG_\beta(B)}{N \omega_N }$. Thus we finally achieve
	 \begin{equation*}
	 \cA_1(E_t)=\frac{\fG_\beta(B)}{N\omega_N}\int_{S^{N-1}}(1+tu(\omega))^{\beta+2N}d\cH^{N-1}(\omega).
	 \end{equation*}
	 Concerning $\cA_2(E_t)$, we apply the change of variables $1+t\widetilde{r}=\bar{r}$ and $1+t\widetilde{\rho}=\bar{\rho}$ and we obtain
	 \begin{equation*}
	 \cA_2(E_t)=-\frac{t^2}{2}g(t)
	 \end{equation*}
	 where
	 \begin{equation*}
	 g(t)=\int_{S^{N-1}}\int_{S^{N-1}}\int_{u(\omega)}^{u(\xi)}\int_{u(\omega)}^{u(\xi)}f(1+t\widetilde{r},1+t\widetilde{\rho},\omega-\xi)d\widetilde{\rho}\, d\widetilde{r} \, d\cH^{N-1}(\xi) \, d \cH^{N-1}(\omega)
	 \end{equation*}
	 and
	 \begin{equation*}
	 f(r,\rho,p)=(|r-\rho|^2+r\rho|p|^2)^{\frac{\beta}{2}}r^{N-1}\rho^{N-1}.
	 \end{equation*}
	 With some cumbersome calculations, one can show that for $t$ small enough, $\omega, \xi \in S^{N-1}$ and $\beta \ge 1$ it holds
	 \begin{equation*}
	 \der{}{t}f(1+t\widetilde{r},1+t\widetilde{\rho},\omega-\xi)\le C
	 \end{equation*}
	 while, for $\beta<1$, we have
	 \begin{equation*}
	 \der{}{t}f(1+t\widetilde{r},1+t\widetilde{\rho},\omega-\xi)\le C_1+C_2 t^{\beta-1}(1+t\widetilde{r})^{N-1}(1+t\widetilde{\rho})^{N-1}.
	 \end{equation*}
	 In both cases, we can differentiate inside the integral so to get $0 \le g'(t)\le C\Norm{u}{L^\infty(S^{N-1})}^2$. Moreover, by Lagrange theorem, we know that there exists $s \in (0,t)$ such that
	 \begin{equation*}
	 g(t)=g(0)+tg'(s)\le g(0)+Ct\Norm{u}{L^\infty(S^{N-1})}^2.
	 \end{equation*}
	 However, by definition of $g$, we have
	 \begin{equation*}
	 g(0)=\int_{S^{N-1}}\int_{S^{N-1}}|\omega-\xi|^\beta|u(\omega)-u(\xi)|^2d\cH^{N-1}(\omega)\, d\cH^{N-1}(\xi)=[u]_{\beta}^2
	 \end{equation*}
	 and then
	 \begin{equation*}
	 \cA_2(E_t)\ge -\frac{t^2}{2}[u]_\beta^2-Ct^3\Norm{u}{L^\infty(S^{N-1})}^2,
	 \end{equation*}
 for a suitable constant $C>0$ independent of $u$. Finally, we obtain that
	 \begin{equation*}
	 \fG_\beta(E_t)\ge \frac{\fG_\beta(B)}{N\omega_N}\int_{S^{N-1}}(1+tu(\omega))^{\beta+2N}d\cH^{N-1}(\omega)-\frac{t^2}{2}[u]_\beta^2-Ct^3\Norm{u}{L^\infty(S^{N-1})}^2.
	 \end{equation*}
	 On the other hand, it holds
	 \begin{equation*}
	 \fG_\beta(B)=\frac{\fG_\beta(B)}{N\omega_N}\int_{S^{N-1}}d\cH^{N-1}(\omega),
	 \end{equation*}
	 thus
	 \begin{equation*}
	 \fD_\beta(E_t)\ge \frac{\fG_\beta(B)}{N\omega_N}\int_{S^{N-1}}((1+tu(\omega))^{\beta+2N}-1)d\cH^{N-1}(\omega)-\frac{t^2}{2}[u]_\beta^2-Ct^3\Norm{u}{L^\infty(S^{N-1})}^2.
	 \end{equation*}
	 Moreover, we have
	 \begin{equation*}
	 (1+tu(\omega))^{\beta+2N}-1\ge (\beta+2N)tu(\omega)+\frac{(\beta+2N)(\beta+2N-1)}{2}t^2u^2(\omega)-Ct^3\Norm{u}{L^\infty(S^{N+1})}^3
	 \end{equation*}
	 for some constant $C$, hence (being also $\Norm{u}{L^\infty(S^{N-1})}^3 \le \Norm{u}{L^\infty(S^{N-1})}^2$)
	 \begin{multline*}
	 \fD_\beta(E_t)\ge \frac{\fG_\beta(B)(\beta+2N)}{N\omega_N}\left(t\int_{S^{N-1}}u(\omega)d\cH^{N-1}(\omega)+\frac{\beta+2N-1}{2}t^2\Norm{u}{L^2(S^{N-1})}^2\right)\\-\frac{t^2}{2}[u]_\beta^2-Ct^3\Norm{u}{L^\infty(S^{N-1})}^2.
	 \end{multline*}
	 On the other hand, the condition $|E_t|=\omega_N$ implies
	 \begin{equation*}
	 \int_{S^{N-1}}((1+tu(\omega))^N-1)d\cH^{N-1}(\omega)=0
	 \end{equation*}
	 and then there exists a constant $C>0$ independent of $u$ such that
	 \begin{equation*}
	 t\int_{S^{N-1}}u(\omega)d\cH^{N-1}(\omega) \ge -\frac{N-1}{2}t^2\Norm{u}{L^2(S^{N-1})}^2-Ct^3\Norm{u}{L^\infty(S^{N-1})}^3,
	 \end{equation*}
	 hence we get
	 \begin{multline*}
	 \fD_\beta(E_t)\ge \frac{\fG_\beta(B)(\beta+2N)(\beta+N)t^2}{2N\omega_N}\Norm{u}{L^2(S^{N-1})}^2-\frac{t^2}{2}[u]_\beta^2-Ct^3\Norm{u}{L^\infty(S^{N-1})}^2.
	 \end{multline*}
	 Finally, by Equation \eqref{lambda1asGB}, we achieve
	 \begin{equation}\label{Defpass1}
	 \fD_\beta(E_t)\ge \frac{t^2}{2}\left(\lambda_{1,\beta}\Norm{u}{L^2(S^{N-1})}^2-[u]_\beta^2\right)-Ct^3\Norm{u}{L^\infty(S^{N-1})}^2.
	 \end{equation}
	 Now we need to estimate $[u]_\beta^2$ and $\Norm{u}{L^2(S^{N-1})}^2$ in terms of spherical harmonics and $(\lambda_{k,\beta})_{k \ge 0}$. Thus, for each $\cS_k$ let us consider $\cY_k=\{Y_{k,j}\}_{j\le d(k)}$ orthonormal bases of $\cS_k$, where $d(k)=\dim \cS_k$, and $a_{k,j}$ the Fourier coefficients of $u$ to write
	 \begin{equation*}
	 u(\omega)=\sum_{k=0}^{+\infty}\sum_{j=1}^{d(k)}a_{k,j}Y_{k,j}(\omega).
	 \end{equation*}
	 Since $\cY_k$ are orthonormal bases, we get
	 \begin{equation*}
	 \Norm{u}{L^2(S^{N-1})}^2=\sum_{k=0}^{+\infty}\sum_{j=1}^{d(k)}a_{k,j}^2.
	 \end{equation*} 
	 On the other hand, by Equation \eqref{sntoint}, we also have
	 \begin{equation*}
	 [u]_\beta^2=\sum_{k=1}^{+\infty}\sum_{j=1}^{d(k)}\lambda_{k,\beta}a_{k,j}^2.
	 \end{equation*}
	 By using the estimate given in Proposition \ref{prop:lowbound}, we get
	 \begin{align*}
	 \lambda_{1,\beta} \Norm{u}{L^2(S^{N-1})}^2-[u]_\beta^2&\ge \sum_{k=2}^{\infty}\sum_{j=1}^{d(k)}(\lambda_{1,\beta}-\lambda_{k,\beta})a_{k,j}^2\\&\ge  D_\beta \Norm{u}{L^2(S^{N-1})}^2-D_\beta\left(a_0^2+\sum_{j=1}^{n}a_{1,j}^2\right).
	 \end{align*}
	 However, from the volume constraint $|E_t|=\omega_N$ we easily get $a_0^2 \le C t^2\Norm{u}{L^2(S^{N-1})}^2$ and, from the barycenter constraint (the fact that the barycenter of $E_t$ is the origin), we also have $a_{1,j}^2 \le C t^2\Norm{u}{L^2(S^{N-1})}^2$ for some constant $C$. Thus, we can always chose $\varepsilon_0$ small enough to have, for any $t \in (0,\varepsilon_0)$, 
	 \begin{equation*}
	 \lambda_{1,\beta} \Norm{u}{L^2(S^{N-1})}^2-[u]_\beta^2\ge \frac{D_\beta}{2} \Norm{u}{L^2(S^{N-1})}^2.
	 \end{equation*}
	 Using last inequality in Equation \eqref{Defpass1} we achieve
	 \begin{equation*}
	 \fD_\beta(E_t)\ge \frac{t^2D_\beta}{4}\Norm{u}{L^2(S^{N-1})}^2+Ct^3\Norm{u}{L^\infty(S^{N-1})}^2.
	 \end{equation*}
	 Finally, we can consider $\varepsilon_0$ small enough to obtain
	 \begin{equation*}
	 \fD_\beta(E_t)\ge \frac{t^2D_\beta}{8}\Norm{u}{L^2(S^{N-1})}^2,
	 \end{equation*}
	 concluding the proof.
\end{proof}
\begin{rmk}\label{rmk:asympD}
	Since $C(N,\beta)=\frac{D_\beta}{8}$, by Equation \eqref{eq:Dbeta} it holds $\widetilde{C}^\infty_N=8(N+2)(N-1)\omega_{N-1}\Gamma\left(\frac{N-1}{2}\right)2^{\frac{3N+1}{2}}$, implies
	\begin{align*}
		\lim_{\beta \to +\infty}\frac{C(N,\beta)}{\widetilde{C}^\infty_N2^\beta\beta^{-\frac{N+1}{2}}}=1, && \lim_{\beta \to 0}\frac{C(N,\beta)}{\widetilde{C}^0_N \beta}=1,
	\end{align*}
	where
	\begin{align}\label{eq:asympconst}
		\begin{split}
		\widetilde{C}^\infty_N&=8(N+2)(N-1)\omega_{N-1}\Gamma\left(\frac{N-1}{2}\right)2^{\frac{3N+1}{2}},\\ \widetilde{C}^0_N&=(N!)^{-1}(N-1)^2\omega_{N-1}2^{N-2}\Gamma\left(\frac{N-1}{2}\right)
	\end{split}	
	\end{align}
  and we used the asymptotics of the ratio of Gamma functions (see \cite{tricomi1951asymptotic}).
\end{rmk}
\section{Reducing to a nearly-spherical set}\label{Sec5}
In this section we want to prove Theorem \ref{thm:quantrandisop} by using the following strategy:
\begin{itemize}
	\item We show that if the asymmetry is \textit{big} (in the sense that there exists a constant $\mu>0$ for which $\delta(E)>\mu$), then also the deficit is \textit{big} and then Theorem \ref{thm:quantrandisop} follows in such case;
	\item If the asymmetry is \textit{small}, then we show we can construct a nearly-spherical set whose asymmetry and deficit are controlled by means of the asymmetry and the deficit of the original set: for such kind of sets we proved Theorem \ref{nrthm} and then Theorem \ref{thm:quantrandisop} follows.
\end{itemize}
Let us first extend the definition of $\fG_\beta$ to couple of functions in $L^1(\R^N)$. Let $f,g \in L^1(\R^N)$ and define
\begin{equation*}
\fG_\beta(f,g)=\int_{\R^N}\int_{\R^N}f(x)g(y)|x-y|^\beta dx dy.
\end{equation*}
We will actually restrict to the case $f,g \in L^1(\R^N; [0,1])$. Moreover, let us denote $\fG_\beta(f):=\fG_\beta(f,f)$ for any function $f \in L^1(\R^N;[0,1])$, $\fG_\beta(E,H):=\fG_\beta(\chi_{E},\chi_{H})$ for any couple of Borel sets $E,H \subseteq \R^N$ with $\chi_{E}$ and $\chi_{H}$ their respective characteristic functions. It is also obvious that $\fG_\beta(E)=\fG_\beta(\chi_{E})=\fG_\beta(E,E)$ for any Borel set $E \subseteq \R^N$. Let us also denote by $B_r(x)$ the ball centered in $x$ with radius $r>0$, $B_r$ when $x=0$ and $B(x)$ when $r=1$.\\
Before executing our plan, we need some technical estimates on $\fG_\beta$ and a form of weak$^*$ continuity.
\subsection{Estimates on $\fG_\beta$}
The majority of the estimates we are going to show are actually corollary of Riesz inequality (see \cite{riesz1930inegalite}). A first easy corollary of this inequality is given by the case in which one of the involved functions is radially symmetric and decreasing with respect to the modulus. In such case, since its radial decreasing rearrangement coincides with the function itself, the inequality can be restated as follows:
\begin{prop}[\textbf{Riesz inequality for decreasing functions}]\label{prop:Riesz}
	Let $f,g: \R^N \to \R^+$ be two measurable functions and $h:\R^+ \to \R^+$ be a decreasing function. Then, denoting by $f^*$ and $g^*$ the radial decreasing rearrangements of $f$ and $g$, it holds
	\begin{equation}\label{Rineq}
	\int_{\R^N}\int_{\R^N}f(z)g(y)h(|z-y|)dzdy \le \int_{\R^N}\int_{\R^N}f^*(z)g^*(y)h(|z-y|)dzdy.
	\end{equation}
	Moreover, if $h$ is strictly decreasing and $f=g \in L^1(\R^N)$, then equality holds in \eqref{Rineq} if and only if $f=f^*$.
\end{prop}
Let us stress that we considered in the previous Proposition only a particular equality case, while the general one is considered for instance in \cite{burchard1996cases}.\\
This is the main tool adopted in \cite{fusco2019sharp}. However, in our case, $h$ is an increasing function. Thus we need to prove a similar inequality in this setting.
\begin{prop}[\textbf{Riesz inequality for increasing functions}]\label{prop:iRineq}
	Let $f,g: \R^N \to \R^+$ be two measurable functions and $h:\R^+ \to \R^+$ be an increasing function. Then, denoting by $f^*$ and $g^*$ the radial decreasing rearrangements of $f$ and $g$, it holds
	\begin{equation}\label{iRineq}
	\int_{\R^N}\int_{\R^N}f(z)g(y)h(|z-y|)dzdy \ge \int_{\R^N}\int_{\R^N}f^*(z)g^*(y)h(|z-y|)dzdy.
	\end{equation}
	Moreover, if $h$ is strictly increasing and $f=g \in L^1(\R^N)$, then equality holds in \eqref{iRineq} if and only if $f=f^*$.
\end{prop}
\begin{proof}
	Let us fix $m,n \in \N$ and define the following functions
	\begin{align*}
	f_n=(f \wedge n)\chi_{B_n} &&
	g_n=(g \wedge n)\chi_{B_n}\\
	h_m=h \wedge m && k_m=m-h_m.
	\end{align*}
	By definition, we have that $f_n,g_n \in L^1(\R^N; \R^+)$ and $k_m$ is decreasing and non-negative, thus we can use Riesz inequality in the form \eqref{Rineq} to obtain
	\begin{equation*}
	\int_{\R^N}\int_{\R^N}f_n(z)g_n(z)k_m(|z-y|)dzdy\le \int_{\R^N}\int_{\R^N}f^*_n(z)g^*_n(z)k_m(|z-y|)dzdy.
	\end{equation*}
	Now, by using the definition of $k_m$ and the fact that any function is equimeasurable with its radial decreasing rearrangement (and $L^1$ is a rearrangement-invariant Banach space), we easily obtain
	\begin{equation*}
	\int_{\R^N}\int_{\R^N}f_n(z)g_n(y)h_m(|z-y|)dzdy \ge \int_{\R^N}\int_{\R^N}f^*_n(z)g^*_n(y)h_m(|z-y|)dzdy.
	\end{equation*}
	Now let us observe that $f_n \uparrow f$ and so $f^*_n \uparrow f^*$ (see \cite[Proposition $7.1.12$]{kufner1977function}), the same holds for $g$ and also $h_m \uparrow h$, thus we can send $m \to +\infty$ and $n \to +\infty$ to achieve equation \eqref{iRineq} by monotone convergence theorem.\\
	Now let us suppose that $h$ is strictly increasing, $f=g \in L^1(\R^N)$ and equality holds in \eqref{iRineq}. We want to show that $f=f^*$ by using the equality case for Equation \eqref{Rineq}. However, since $h$ could be unbouded, we have to introduce some auxiliary functions. Being $h$ monotone, it is locally of bounded variation and then it admits a distributional derivative $\mu$ that is a Radon measure on $[0,+\infty)$. Let us also consider a Borel function $s$ on $[0,+\infty)$ such that $0<s(x)<1$ $\mu$-a. e. and
	\begin{equation*}
		\int_0^{+\infty}s(x)d\mu(x)=1
	\end{equation*}
	and let us define
	\begin{align*}
		h_1(t)=\int_{[0,t)} s(\tau)d\mu(\tau), && h_2(t)=\int_{[0,t)} (1-s(\tau))d\mu(\tau).
	\end{align*}
	First of all, let us observe that $h_1(t)+h_2(t)=h(t)$ for all but countably many $t$.	Being $h$ strictly increasing, we know that $\mu$ is a positive measure. This, together with $0<s<1$ $\mu$-a.e., leads to the fact that $h_i$ is strictly increasing for $i=1,2$. Thus, by \eqref{iRineq} we get
	\begin{equation}\label{eq:hi}
		\int_{\R^N}\int_{\R^N}f(z)f(y)h_i(|z-y|)dzdy \ge \int_{\R^N}\int_{\R^N}f^*(z)f^*(y)h_i(|z-y|)dzdy, \ i=1,2.
	\end{equation}
	Let us suppose that
	\begin{equation*}
		\int_{\R^N}\int_{\R^N}f(z)f(y)h_1(|z-y|)dzdy > \int_{\R^N}\int_{\R^N}f^*(z)f^*(y)h_1(|z-y|)dzdy.
	\end{equation*} 
	Summing the previous inequality with the one in \eqref{eq:hi} for $i=2$ and using the fact that $h_1+h_2=h$ almost everywhere, we get
	\begin{equation*}
		\int_{\R^N}\int_{\R^N}f(z)f(y)h(|z-y|)dzdy > \int_{\R^N}\int_{\R^N}f^*(z)f^*(y)h(|z-y|)dzdy,
	\end{equation*} 
	which is a contradiction with the equality in \eqref{iRineq}. Hence, it holds
	\begin{equation*}
		\int_{\R^N}\int_{\R^N}f(z)f(y)h_1(|z-y|)dzdy= \int_{\R^N}\int_{\R^N}f^*(z)f^*(y)h_1(|z-y|)dzdy.
	\end{equation*}
	Being $\mu$ a positive measure, we have
	\begin{equation*}
		h_1(t)=\int_{[0,t)}s(\tau)d\mu(\tau)< \int_0^{+\infty}s(\tau)d\mu(\tau)=1,
	\end{equation*}	
		 thus we can consider the strictly decreasing positive function $k(\tau)=1-h_1(\tau)$ to achieve
	\begin{equation*}
		\int_{\R^N}\int_{\R^N}f(z)f(y)k(|z-y|)dzdy= \int_{\R^N}\int_{\R^N}f^*(z)f^*(y)k(|z-y|)dzdy,
	\end{equation*}
	concluding that $f=f^*$ by Proposition \ref{prop:Riesz}. 
\end{proof}
Now that we have this result, we can show some lower bounds on the functional $\fG_\beta$. First of all, we have the following lower bound.
\begin{lem}
	Let $\beta>0$. Then, for any positive measurable function $g:\R^N \to \R^+$ it holds
	\begin{equation}\label{eq:lowboundRi}
	\int_{\R^N}g(y)|y|^\beta dy \ge \int_{\R^N}g^*(y)|y|^\beta dy.
	\end{equation}
	In particular, for any finite measure Borel set $H \subseteq \R^N$, it holds
	\begin{equation}\label{eq:lowboundRiBor}
	\int_{H}|x-y|^\beta dy \ge \frac{N \omega_N}{\beta+N}r^{\beta+N},
	\end{equation}
	where $r=\left(\frac{|H|}{\omega_N}\right)^{\frac{1}{N}}$. Finally, for any finite measure Borel sets $G,H \subseteq \R^N$, it holds
	\begin{equation}\label{eq:lowboundGbeta}
	\fG_\beta(G,H)\ge |G|\tau(|H|),
	\end{equation}
	where
	\begin{equation*}
	\tau(r)=\frac{N \omega_N^{1-\frac{\beta+N}{N}}}{\beta+N}r^{\frac{\beta+N}{N}}
	\end{equation*}
\end{lem}
\begin{proof}
	Let us first observe that arguing as before on the sequence of functions $g_n=(g \wedge n)\chi_{B_n}$, we just have to show inequality \eqref{eq:lowboundRi} on bounded functions of compact support. Let us then suppose there exists a constant $M>0$ such that $g \le M$ and ${\rm supp \,}g \subseteq B_M$. Fix $\varepsilon \in (0,1)$ and define $f(z)=\frac{1}{\omega_N\varepsilon^N}\chi_{B_\varepsilon}(z)$, observing that $f^*\equiv f$. Thus, choosing $h(t)=t^\beta$ for $\beta>0$, we can use inequality \eqref{iRineq} to achieve
	\begin{equation*}
	\int_{B_M}g(y)\fint_{B_\varepsilon}|y-z|^\beta dz dy \ge \int_{B_M}g^*(y)\fint_{B_\varepsilon}|y-z|^\beta dz dy.
	\end{equation*}
	By dominated convergence theorem (that we can use being $g$ and $g^*$ bounded and so $|y-z|^\beta$, since $y \in B_M$ and $z \in B_\varepsilon\subset B_1$), we have, taking $\varepsilon \to 0$,
	\begin{equation*}
	\int_{B_M}g(y)|y|^\beta dy \ge \int_{B_M}g^*(y)|y|^\beta dy,
	\end{equation*}
	which is inequality \eqref{eq:lowboundRi}. Formula \eqref{eq:lowboundRiBor} follows by considering $g(y)=\chi_{H}(x+y)$. Indeed in such case we have $g^*(y)=\chi_{B_r}(y)$ and then
	\begin{equation*}
	\int_{H}|x-y|^\beta dy=\int_{\R^N}g(y)|y|^\beta dy \ge \int_{\R^N}g^*(y)|y|^\beta dy=\int_{B_r}|y|^\beta dy.
	\end{equation*}
	Finally, integrating both sides of Equation \eqref{eq:lowboundRiBor} on $G$ we get inequality \eqref{eq:lowboundGbeta}.
\end{proof}
Now we let us show a similar lower bound on the functional $\fG_\beta$ applied on functions.
\begin{lem}\label{lem:functionminimum}
	Let $\beta>0$. Then, for any function $g \in L^1(\R^N;[0,1])$ it holds
	\begin{equation}\label{eq:lowboundRi2}
	\int_{\R^N}g(y)|y|^\beta dy \ge \frac{N\omega_N}{N+\beta}r^{N+\beta},
	\end{equation}
	where $r=\left(\frac{\Norm{g}{L^1(\R^N)}}{\omega_N}\right)^{\frac{1}{N}}$. Moreover, it holds
	\begin{equation}\label{eq:lowboundGbeta2}
	\fG_\beta(g)\ge \fG_\beta(B_r)
	\end{equation}
	and equality holds if and only if $g$ is the characteristic function of a ball.
\end{lem}
\begin{proof}
	To prove inequality \eqref{eq:lowboundRi2}, we just observe that
	\begin{align*}
	\int_{\R^N}g(y)|y|^\beta dy-\int_{B_r}|y|^\beta dy&=\int_{\R^N\setminus  B_r}g(y)|y|^\beta dy-\int_{B_r}(1-g(y))|y|^\beta dy\\
	&\ge \int_{\R^N\setminus  B_r}g(y)r^\beta dy-\int_{B_r}(1-g(y))r^\beta dy\\
	&=r^\beta(\Norm{g}{L^1}-\omega_N r^N)=0.
	\end{align*}
	Concerning inequality \eqref{eq:lowboundGbeta2}, we need to introduce some other tools. For any function $\theta \in L^1(\R^N;[0,1])$ we define $r(\theta)=\left(\frac{\Norm{g}{L^1(\R^N)}}{\omega_N}\right)^{\frac{1}{N}}$ and $\widehat{\theta}=\chi_{B_{r(\theta)}}$. We claim that:
	\begin{itemize}
		\item For any $f,\theta \in L^1(\R^N;[0,1])$ such that $f=f^*$ and $\theta=\theta^*$, it holds
		\begin{equation*}
		\fG_\beta(f,\theta)\ge \fG_\beta(f,\widehat{\theta})
		\end{equation*}
		and, if $f=\widehat{f}$, equality holds if and only if $\theta=\widehat{\theta}$.
	\end{itemize}
If we show this claim, we have both inequality \eqref{eq:lowboundGbeta2} and the equality condition. Indeed, let us observe that $\widehat{g}=\widehat{g^*}$ since $r(g)=r(g^*)$ and, by using Riesz inequality, we have
\begin{equation*}
\fG_\beta(g)=\fG_\beta(g,g)\ge \fG_\beta(g^*,g^*)\ge\fG_\beta(\widehat{g},g^*)\ge \fG_\beta(\widehat{g},\widehat{g})=\fG_\beta(B_{r}).
\end{equation*}
Moreover, if equality holds, then $\fG_\beta(g,g)=\fG_\beta(g^*,g^*)$, which is true if and only if $g=g^*$ by and the characterization of the equality case in Riesz inequality (see Proposition \ref{prop:iRineq}), and $\fG_\beta(\widehat{g},g^*)=\fG_\beta(\widehat{g},\widehat{g})$, which is true if and only if $g^*=\widehat{g}$; thus $g=\widehat{g}=\chi_{B_r}$.\\
Now let us prove the claim. To do this, let us observe that for any $t>0$ and any $|x_1|=|x_2|=t$, one has $\int_{B}|y-x_1|^\beta dy=\int_{B}|y-x_2|^\beta dy$. Thus the function
\begin{equation}\label{eq:psi}
\psi(t)=\int_{B}|y-x|^\beta dy, \quad |x|=t
\end{equation}
is well defined. Since we can choose any $x \in \R^N$, such that $|x|=t$, let us consider $x=te_1$ to write
\begin{equation*}
\psi(t)=\int_{B}|y-te_1|^\beta dy.
\end{equation*}
From the last formulation, it is easy to observe that $\psi(t)$ is a $C^1$ function and, for any $t>0$,
\begin{equation*}
\psi'(t)=\int_{B}\beta|y-te_1|^{\beta-2}(t-y_1)dy\ge \beta\int_{B}|y_1-t|^{\beta-2}(t-y_1)dy > 0.
\end{equation*}
Now, for any function $f \in L^1(\R^N;[0,1])$ with $f=f^*$, let us define, for $\rho>0$,
\begin{equation*}
\zeta_f(\rho)=\fint_{\partial B_\rho}\int_{\R^N}f(z)|z-y|^\beta dz d\cH^{N-1}(y)
\end{equation*}
and let us show that $\zeta_f$ is increasing. To do this, let us fix $\rho_1<\rho_2$ and define $g=\chi_{B_{\rho_1}}+\chi_{B_{\rho_2+\varepsilon}}-\chi_{B_{\rho_2}}$ for some $\varepsilon>0$. Let $\delta(\varepsilon)>0$ be solution of the equation
\begin{equation*}
(\rho_1+\delta(\varepsilon))^N=\rho_1^N+(\rho_2+\varepsilon)^N-\rho_2^N
\end{equation*}
and observe that $\lim_{\varepsilon \to 0^+}\delta(\varepsilon)=0$. It is also easy to check that
\begin{equation*}
g^*=g+(\chi_{B_{\rho_1+\delta}}-\chi_{B_{\rho_1}})-(\chi_{B_{\rho_2+\varepsilon}}-\chi_{B_{\rho_2}}).
\end{equation*}
By Riesz inequality we have $\fG_\beta(f,g)\ge \fG_\beta(f,g^*)$. By explicitly writing this relation we achieve and denoting by $A_{\rho_1}=B_{\rho_1+\delta}\setminus B_{\rho_1}$ and $A_{\rho_2}=B_{\rho_2+\varepsilon}\setminus B_{\rho_2}$, we have
\begin{equation*}
\int_{A_{\rho_1}}\int_{\R^N}f(z)|z-y|^\beta dz dy \le \int_{A_{\rho_2}}\int_{\R^N}f(z)|z-y|^\beta dz dy.
\end{equation*}
However, by definition of $\delta$, we have that $|A_{\rho_1}|=|A_{\rho_2}|$ and thus we can divide both term of the previous inequality by this measure and send $\varepsilon \to 0$ to get $\zeta_f(\rho_1)\le \zeta_f(\rho_2)$. Thus we have shown that $\zeta_f$ is increasing.\\
The case $f=\widehat{f}$ is more interesting, since we can actually show that $\zeta_f$ is strictly increasing. Indeed, in such case, $f=\chi_{B_R}$ for some $R>0$ and then we have
\begin{align*}
\zeta_f(\rho)&=\fint_{\partial B_\rho}\int_{B_R}|z-y|^{\beta}dzd\cH^{N-1}(y)=R^N\fint_{\partial B_\rho}\int_{B}|R\omega-y|^{\beta}d\omega d\cH^{N-1}(y)\\
&=R^{N+\beta}\fint_{\partial B_\frac{\rho}{R}}\int_{B}\left|\omega-\eta\right|^{\beta}d\omega d\cH^{N-1}(\eta)=R^{N+\beta}\psi\left(\frac{\rho}{R}\right).
\end{align*}
Now let us assume $f=f^*$ and $\theta=\theta^*$ in $L^1(\R^N;[0,1])$ and let us fix $r=r(\theta)$. Since $\theta=\theta^*$, we know that $\theta$ is radially symmetric, thus let us set $\theta(\rho)=\theta(y)$ as $|y|=\rho$. Then, by coarea formula, we get
\begin{align*}
\fG_\beta(f,\theta)-\fG_\beta(f,\hat{\theta})&=\int_0^{+\infty}\int_{\partial B_\rho}\int_{\R^N}f(z)(\theta(\rho)-\hat{\theta}(\rho))|z-y|^{\beta}dzd\cH^{n-1}(y)d\rho\\
&=\int_0^{+\infty}\zeta_f(\rho)(\theta(\rho)-\hat{\theta}(\rho))N\omega_N \rho^{N-1}d\rho\\
&=\int_r^{+\infty}\zeta_f(\rho)\theta(\rho)N\omega_N \rho^{N-1}d\rho-\int_0^{r}\zeta(\rho)(1-\theta(\rho))N\omega_N \rho^{N-1}d\rho\\
&\ge \zeta_f(r)\left(\int_r^{+\infty}\theta(\rho)N\omega_N \rho^{N-1}d\rho-\int_0^{r}(1-\theta(\rho))N\omega_N \rho^{N-1}d\rho\right)
\\&=\zeta_f(r)\left(\Norm{1-\theta}{L^1(B_r)}-\Norm{\theta}{L^1(\R^N \setminus B_r)}\right)=0,
\end{align*}
concluding that $\fG_\beta(f,\theta)\ge \fG_\beta(f,\hat{\theta})$. On the other hand, if $f=\widehat{f}$ and $\fG_\beta(f,\theta)=\fG_\beta(f,\hat{\theta})$, it holds
\begin{multline*}
\zeta_f(r)\left(\int_r^{+\infty}\theta(\rho)N\omega_N \rho^{N-1}d\rho-\int_0^{r}(1-\theta(\rho))N\omega_N \rho^{N-1}d\rho\right)\\=\int_0^{+\infty}\zeta_f(\rho)(\theta(\rho)-\hat{\theta}(\rho))N\omega_N \rho^{N-1}d\rho
\end{multline*}
which is true, being $\zeta_f$ strictly increasing, if and only if $\theta=\widehat{\theta}$, concluding the proof of the claim.
\end{proof}
Now let us observe that, in the case we have an invertible transport map between two sets, the difference of the energy can be controlled in terms of the transport map.
\begin{lem}\label{lem:diff}
	Let $E_1,E_2 \subseteq B_R$ for some $R>0$ be two Borel sets such that $|E_1|=|E_2|$ and $\beta>0$. Suppose $\Phi:E_1 \to E_2$ is an invertible transport map. Then there exists a constant $C$ depending only on $R$, $\beta$ and $N$ such that for any Borel set $E_3 \subseteq B_R$ it holds
	\begin{equation}\label{eq:transcontr}
	|\fG_\beta(E_1,E_3)-\fG_\beta(E_2,E_3)|\le C(R,\beta,N) |E_3|^{\alpha}\int_{E_1}|y-\Phi(y)|dy,
	\end{equation}
where
\begin{equation}\label{eq:alphadef}
	\alpha=\min\left\{1,1+\frac{\beta-1}{N}\right\}.
\end{equation}
\end{lem}
\begin{proof}
	Being $\Phi$ an invertible transport map, we can suppose, without loss of generality, that $\fG_\beta(E_1,E_3)\ge \fG_\beta(E_2,E_3)$.\\
	Let us first assume $\beta \ge 1$. Then, by Lagrange's theorem, we get
	\begin{equation}\label{pass2}
		|y-z|^\beta-|\Phi(y)-z|^\beta \le \beta(2R)^{\beta-1}|y-\Phi(y)|.
	\end{equation}
	Integrating inequality \eqref{pass2} over $E_3$ and $E_1$ we complete the proof in the case $\beta\ge 1$.\\
	Now let us consider $\beta \in (0,1)$. Fix $y \in E_1$ and observe that, if $|y-z|^\beta-|\Phi(y)-z|^\beta \ge 0$, it holds
	\begin{equation*}
		|y-z|^\beta-|\Phi(y)-z|^\beta \le \beta|\Phi(y)-z|^{\beta-1}|y-\Phi(y)|.
	\end{equation*}
	Integrating over $E_3$ to get
	\begin{align*}
		\int_{E_3}|y-z|^\beta dz-\int_{E_3}|\Phi(y)-z|^\beta dz &\le \beta|y-\Phi(y)|\int_{E_3}|\Phi(y)-z|^{\beta-1}dz\\
		&\le \beta|y-\Phi(y)|\int_{B_\rho}|w|^{\beta-1}dw\\
		&=|y-\Phi(y)| \frac{\beta N}{N+\beta-1}\omega_N\rho^{N+\beta-1}\\
		&=|y-\Phi(y)| \frac{\beta N\omega_N^{\frac{1-\beta}{N}}}{N+\beta-1}|E_3|^{1+\frac{\beta-1}{N}},
	\end{align*}
	where we considered $\rho>0$ such that $|E_3|=\rho^N\omega_N$. Integrating both sides of the last inequality in $E_1$ we conclude the proof.
\end{proof}
\subsection{The functional $\fG_\beta$ and the weak$^*$ convergence in $L^\infty$}
As we stated before, we need to obtain some form of weak$^*$ continuity. However, $L^1$ is not the dual of any space. To avoid this problem, let us observe that we fixed the range of the functions: indeed, we have by definition that $L^1(\R^N;[0,1]) \subset L^\infty(\R^N;[0,1])$. Thus we can use as weak$^*$ convergence the one on $L^\infty(\R^N;[0,1])$.\\
However, even in this case, the eventual non-compactness of the support of the involved functions could create some problems. Hence we need to introduce the truncated functional
\begin{equation*}
\fG_\beta^M(g_1,g_2)=\int_{\R^N}\int_{\R^N}g_1(x)g_1(y)(|x-y| \wedge M^{\frac{1}{\beta}})^\beta dxdy,
\end{equation*}
with $\fG_\beta^M(f)=\fG_\beta^M(f,f)$ for any $f \in L^1(\R^N;[0,1])$, $\fG_\beta^M(E,F)=\fG_\beta^M(\chi_E,\chi_F)$ for any measurable sets $E,F \subseteq \R^N$ and $\fG_\beta^M(E)=\fG_\beta^M(E,E)$ for any measurable set $E \subseteq \R^N$. Moreover, let us denote its deficit by $\fD_\beta^M(E)=\fG_\beta^M(E)-\fG_\beta^M(B)$. Since the function $h_M(t)=(t \wedge M^{\frac{1}{\beta}})^\beta$ is increasing, we have by Riesz rearrangement inequality that $\fD_\beta^M(E)\ge 0$ for any Borel set $E$ such that $|E|=\omega_N$.\\
Concerning the truncated energy, let us observe that if $M>2$ then $\fG_\beta^M(B)=\fG_\beta(B)$. Moreover, since if $M_1<M_2$ then $h_{M_1}<h_{M_2}$, it holds $\fG_\beta^{M_1}(E)\le \fG_\beta^{M_2}(E)$.\\
By monotone convergence theorem, one also obtains $\lim_{M \to +\infty}\fG_\beta^M(f)=\fG_\beta(f)$, thus we can also conclude that for any $M>0$ it holds $\fG_\beta^M(f)\le \fG_\beta(f)$. Now we are ready to prove the following \textit{weak$^*$ continuity} Lemma.
\begin{lem}\label{lem:weakstarcont}
	Let $\{f_n\}_{n \in \N} \subset L^1(\R^N;[0,1])$ be a sequence of functions such that $f_n \overset{*}{\rightharpoonup} f$ in $L^\infty(\R^N;[0,1])$ and $\Norm{f_n}{L^1(\R^N)}\to\Norm{f}{L^1(\R^N)}$. Then
	\begin{equation*}
	\fG_\beta(f)=\lim_{M \to +\infty}\lim_{n \to +\infty}\fG_\beta^M(f_n).
	\end{equation*} 
\end{lem}
\begin{proof}
	Fix $\varepsilon>0$ and observe that, being $f \in L^1(\R^N;[0,1])$, there exists a radius $R>0$ such that $\int_{\R^N \setminus B_R}f(x)dx<\varepsilon$. By using both the convergence $\Norm{f_n}{L^1(\R^N)}\to \Norm{f}{L^1(\R^N)}$ and $f_n \overset{*}{\rightharpoonup} f$ in $L^\infty(\R^N)$, we get
	\begin{align*}
	\lim_{n \to +\infty}\int_{\R^N\setminus B_R}f_n(x)dx&=\lim_{n \to +\infty}\int_{\R^N}f_n(x)dx-\lim_{n \to +\infty}\int_{B_R}f_n(x)dx\\
	&=\int_{\R^N}f(x)dx-\int_{B_R}f(x)dx=\int_{\R^N\setminus B_R}f(x)dx<\varepsilon,
	\end{align*}
	hence, for $n$ big enough, we achieve $\int_{\R^N \setminus B_R}f_n(x)dx<\varepsilon$.\\
	Now, for any function $g \in L^\infty(\R^N)$ let us define $\widehat{g}(x,y)=g(x)g(y)$ that is a function in $L^\infty(\R^{2N})$. In particular we have that $\widehat{f}_n \overset{*}{\rightharpoonup}\widehat{f}$ and in particular $\fG^M_\beta(f_n\chi_{B_R}) \to \fG^M_\beta(f_n\chi_{B_R})$. \\
	On the other hand, since $\int_{\R^N \setminus B_R}f(x)dx<\varepsilon$, it is easy to check that
	\begin{equation*}
	\fG_\beta^M(f)-\fG_\beta^M(f\chi_{B_R})\le 2\varepsilon \Norm{f}{L^1}M
	\end{equation*}
	and the same holds for $f_n$.\\
	We achieve, for $n$ big enough,
	\begin{equation*}
	|\fG_\beta^M(f)-\fG_\beta^M(f_n)|\le 2\varepsilon M(\Norm{f}{L^1(\R^N)}+\Norm{f_n}{L^1(\R^N)})+|\fG_\beta^M(f\chi_{B_R})-\fG_\beta^M(f_n\chi_{B_R})|
	\end{equation*}
	and then, sending $n \to +\infty$,
	\begin{equation*}
	\limsup_{n \to +\infty}|\fG_\beta^M(f)-\fG_\beta^M(f_n)|\le 4\varepsilon M\Norm{f}{L^1(\R^N)}.
	\end{equation*}
	Since $\varepsilon>0$ is arbitrary, we can send it to $0^+$ to obtain
	\begin{equation*}
	\lim_{n \to +\infty}\fG_\beta^M(f_n)=\fG_\beta^M(f).
	\end{equation*}
	Finally, taking the limit as $M \to +\infty$, we conclude the proof.
\end{proof}
\subsection{The \textit{big asymmetry} case}
Now that we have some estimates and continuity properties for $\fG_\beta$, we can work on the first part of our plan. As first step, we need to show a rough deficit estimate as the asymmetry is big enough.
\begin{lem}\label{lem:rest}
	There exists a constant $\xi:=\xi(N,\beta)>0$ such that for any Borel set $E \subset \R^N$ with $|E|=\omega_N$ and $\delta(E)\ge 2(\omega_N-\xi)$ it holds
	\begin{equation*}
	\fD_\beta(E)\ge \frac{(3^\beta-2^\beta)}{2}\omega_N^2.
	\end{equation*}
\end{lem}
\begin{proof}
	Note that 
	\begin{equation*}
		\fG_\beta(B)=\int_B\int_B|x-y|^\beta dxdy \le 2^\beta\omega_N^2.
	\end{equation*} 
	On the other hand, given a ball $B(x)$,
	\begin{equation*}
		|B(x)\setminus E|=\omega_N-|B(x)\cap E|=|E \setminus B(x)|=\frac{|B(x) \Delta E|}{2}.
	\end{equation*}	
	Let $k \in \N$ the minimum number of balls of radius $1$ covering $B_3$. Then, given $x \in E$, we have from the above estimate $|B_3(x) \cap E|<k\xi$. Hence
	\begin{equation*}
		\fG_\beta(E)\ge \int_E dx \int_{E\setminus B_3(x)}|x-y|^\beta dy\ge 3^\beta|E|(|E|-k\xi).
	\end{equation*}  
	Thus if $\xi$ is sufficiently small, depending on $N$ and $\beta$, we have
	\begin{equation*} \fG_\beta(E)-\fG_\beta(B)\ge 3^\beta\omega_N(\omega_N-k\xi)-2^\beta \omega_N^2=(3^\beta-2^\beta)\omega_N^2-3^\beta \omega_N k\xi>0
	\end{equation*}
\end{proof}
Now that we have a rough deficit estimate, we can refine such estimate to show that we can reduce to the \textit{small asymmetry} case.
\begin{lem}\label{lemcompact}
	For any $\mu>0$ there exists $\eta=\eta(\mu,\alpha,N)>0$ such that for any set $E \subseteq \R^N$ such that $|E|=\omega_N$ and $\delta(E)\ge \mu$ it holds $\fD_\beta(E)\ge \eta$.
\end{lem}
\begin{proof}
	Let $E_n$ be such that $|E_n|=\omega_N$ and $\fD_\beta(E_n)\to 0$. We want to show that $\delta(E_n)\to 0$.\\
	By the concentration compactness Lemma (see \cite{lions1984concentration}), there exists a non-relabeled subsequence $E_n$ such that one of the following properties hold:
	\begin{itemize}
		\item \textit{vanishing}: For any $R>0$ it holds $\lim_{n}\sup_{x \in \R^N}|E_n \cap B_R(x)|=0$;
		\item \textit{dichotomy}: There exists $\lambda \in (0,\omega_N)$ such that for every $\varepsilon>0$ there exist $R(\varepsilon)$ and two sequences of sets $E^1_n,E^2_n \subseteq E_n$ with $E^1_n \subseteq B_{R(\varepsilon)}$ such that
		\begin{align}\label{dic}
		|E_n \Delta (E_n^1 \cup E_n^2)|\le \varepsilon && {\rm dist}(E_n^1,E_n^2) \to +\infty \\
		\nonumber||E_n^1|-\lambda| \le \varepsilon &&
		||E_n^2|-\omega_N+\lambda| \le \varepsilon.
		\end{align}
		\item \textit{tightness}: For any $\varepsilon>0$ there exists $R=R(\varepsilon)>0$ such that $\limsup_{n \to +\infty}|E_n\setminus B_R(0)|<\varepsilon$.
	\end{itemize}
	Let us exclude the case of a vanishing sequence. Indeed, if $R=1$ we have for any $n$
	\begin{equation*}
	\delta(E_n)=\inf_{x \in \R^N}|E_n \Delta B(x)|=2\inf_{x \in \R^N}|B(x) \setminus E_n|=2\omega_N-2\sup_{x \in \R^N}|B(x)\setminus E_n|
	\end{equation*}
	hence $\delta(E_n)\to 2\omega_n$, which is a contradiction to $\fD(E_n)\to 0$ by Lemma \ref{lem:rest}.\\
	Let us now exclude dichotomy case. Let $\lambda \in (0,\omega_N)$ be such that \eqref{dic} holds. Consider two balls with volume $|E_n^1|$ and $|E_n^2|$, hence with radii
	\begin{equation*}
	R_{1,n}=\left(\frac{|E_n^1|}{\omega_N}\right)^{1/N}, \qquad R_{2,n}=\left(\frac{|E_n^2|}{\omega_N}\right)^{1/N}.
	\end{equation*}
	In particular, since the energy is minimized on balls, we have
	\begin{equation*}
	\fG_\beta(E_n^i)\ge R_{i,n}^{2N+\beta}\fG_\beta(B)
	\end{equation*}
	and then
	\begin{equation*}
	\fG_\beta(E_n^1)+\fG_\beta(E_n^2)\ge (R_{1,n}^{2N+\beta}+R_{2,n}^{2N+\beta})\fG_\beta(B).
	\end{equation*}
	Now let us estimate $\fG_\beta(E_n^1,E_n^2)$. Let $d_n={\rm dist}(E_n^1,E_n^2)$ and observe that
	\begin{equation*}
	\fG_\beta(E_n^1,E_n^2)\ge d_n^\beta |E_n^1||E_n^2|.
	\end{equation*}
	Thus we have
	\begin{equation*}
	\fG_\beta(E_n) \ge (R_{1,n}^{2N+\beta}+R_{2,n}^{2N+\beta})\fG_\beta(B)+d_n^\beta |E_n^1||E_n^2|,
	\end{equation*}
hence $\fG_\beta(E_n) \to +\infty$ which is a contradiction.\\
	This proves that we are in the tightness case. By Prohorov's Theorem (see \cite{billingsley2013convergence}), up to a subsequence, $\chi_{E_n} \overset{*}{\rightharpoonup} f$ in $L^\infty$. Moreover, tightness implies also that $\Norm{f}{L^1}=\lim_{n \to +\infty}\Norm{\chi_{E_n}}{L^1}=\omega_N$. By Lemma \ref{lem:weakstarcont} we have
	\begin{equation*}
	\fG_\beta(f)=\lim_{M \to +\infty}\lim_{n \to +\infty}\fG^M_\beta(E_n).
	\end{equation*}
	Let us fix $M>2$: we have
	\begin{equation*}
	\fD^M_\beta(E_n):=\fG^M_\beta(E_n)-\fG_\beta(B)\le \fG_\beta(E_n)-\fG_\beta(B)=\fD(E_n)
	\end{equation*}
	hence $\fD^M(E_n)\to 0$. Thus, in particular,
	\begin{equation*}
	\lim_{n \to +\infty}\fG^M_\beta(E_n)=\fG^M_\beta(B)
	\end{equation*}
	and then
	\begin{equation*}
	\fG_\beta(f)=\fG_\beta(B),
	\end{equation*}
	that implies $f=\chi_{B(z)}$ for some $z \in \R^N$.\\
	Finally, we have
	\begin{equation*}
	\limsup_{n \to +\infty}\delta(E_n)\le \limsup_{n \to +\infty}|E_n \Delta B(z)|=2\limsup_{n \to +\infty}|B(z) \setminus E_n|=0
	\end{equation*}
	by the $L^\infty$ weak$^*$ convergence of $\chi_{E_n}$ to $\chi_{B(z)}$.
\end{proof}
From last Lemma, we conclude that we can always reduce to the case in which asymmetry is small. Now that we have reduced to this case, we can study the construction of the \textit{better} nearly-spherical set.
\subsection{Construction of the nearly-spherical set}
The second part of the plan follows the same ideas of \cite{fusco2019sharp}. However, since Lemma \ref{lem:diff} is proved only for bounded sets, we need a preliminary step. Summarizing the plan:
\begin{enumerate}
	\item First we prove that we can trap our set in a ball of a certain radius (independent of the set itself) without changing the asymmetry and controlling in a suitable way the deficit;
	\item As second step we \textit{fill the holes} of our set in order to obtain a new set that is uniformly close to a ball of radius $1$, without changing the asymmetry and reducing the deficit;
	\item As third step, we show that if Equation \eqref{quantisop} still does not hold, then we can move some other mass of the set to construct a nearly-spherical set around a certain ball such that the deficit at most duplicates and the symmetric difference with respect to the chosen ball controls the asymmetry;
	\item Finally, observing that the barycentre of the previously constructed set depends continuously on the choice of the centre of the ball, we can construct it in such a way that the barycentre of the new set coincides with the centre of the chosen ball.
\end{enumerate}
Let us formalize the first step of this procedure.
\begin{lem}\label{lembound}
	There exist three positive constants $R$, $K$ and $\delta_0$ depending only on $N$ and $\beta$ such that for any Borel set $E \subseteq \R^N$ with $|E|=\omega_N$ and $\delta(E)<\delta_0$ there exists a set $\widetilde{E} \subseteq B_R$ such that $|\widetilde{E}|=\omega_N$, $\fD_\beta(\widetilde{E})\le K\fD_\beta(E)$ and $\delta(\widetilde{E})=\delta(E)$. Moreover, $\delta(\widetilde{E})=|\widetilde{E} \Delta B|$.
\end{lem}
\begin{proof}
	Let us observe that since $\delta(E)<\delta_0$, arguing as in the proof of Lemma \ref{lem:rest}, we have $|E \setminus B(z)|<\frac{\delta_0}{2}$, where $B(z)$ is the optimal ball for the asymmetry, i.e. $\delta(E)=|E \Delta B(z)|$. Moreover, the operator $\fG_\beta$ is translation-invariant, hence we may assume $z \equiv 0$. We also have $|E \cap B| \ge \omega_N-\frac{\delta_0}{2}$, thus we can choose $\delta_0$ to be so small such that $|E \cap B_1(z)| \ge \frac{\omega_N}{k_R}$ where $k_R>1$ is to be defined in what follows.\\
	Now let us consider the ball $B_{R}$ where $R>0$ is to be specified later and the annulus $A:=B_R \setminus B$ and let us split $E$ in three parts:
	\begin{equation*}
	E_1=E \cap B, \qquad E_2=E \setminus B_R, \qquad E_3=E \cap A.
	\end{equation*}
	Now let us consider another annulus $\widetilde{A}=B_{1+\frac{R}{2}}\setminus B_{1+\frac{R}{3}}$.
	As 
	\begin{equation*}
	|\widetilde{A} \setminus E_3|\ge \omega_N\left(\left(1+\frac{R}{2}\right)^N-\left(1+\frac{R}{3}\right)^N\right)-\frac{\delta_0}{2},
	\end{equation*}
	we can choose $\delta_0$ so small that
	\begin{equation*}
	\omega_N \left(\left(1+\frac{R}{2}\right)^N-\left(1+\frac{R}{3}\right)^N\right)-\frac{\delta_0}{2}>\frac{\delta_0}{2}.
	\end{equation*}
	Thus, since $|E_2|<\frac{\delta_0}{2}$, there exists $\widetilde{E}_2 \subseteq \widetilde{A} \setminus E_3$ such that $|\widetilde{E}_2|=|E_2|$. Then we define $\widetilde{E}=(E \cup \widetilde{E}_2)\setminus E_2$. Before estimating $\fD_\beta(\widetilde{E})$, let us show that we can choose $k_R$, hence $\delta_0$, so small that $|E_2|\le C\fD_\beta(E)$ where $C$ is a constant depending only on $N, \beta$.\\
	To do this, let us recall that $\fG_\beta(E)=\fD_\beta(E)+\fG_\beta(B)$ and then, using the fact that $E=E_1 \cup E_2 \cup E_3$, we have
	\begin{align*}
	\fG_\beta(E)\ge \fG_\beta(E_1 \cup E_3)+2\fG_\beta(E_2,E_1)
	\end{align*}
	from which we obtain
	\begin{equation*}
	\fG_\beta(E_2,E_1)\le \fD_\beta(E)+\fG_\beta(B)-\fG_\beta(E_1 \cup E_3)-\fG_\beta(E_1,E_2).
	\end{equation*}
	Now let us observe that, denoting by $\widetilde{B}$ a ball with measure $|\widetilde{B}|=|E_1 \cup E_3|=\omega_N-|E_2|$, we have, by minimality of the ball,
	\begin{equation*}
	\fG_\beta(E_1 \cup E_3)\ge \left(\frac{\omega_N-|E_2|}{\omega_N}\right)^{2+\frac{\beta}{N}}\fG_\beta(B)
	\end{equation*}
	and then, since $|E_2|/\omega_N<1$,
	\begin{align*}
	\fG_\beta(E_2,E_1)&\le \fD_\beta(E)+\left(1-\left(1-\frac{|E_2|}{\omega_N}\right)^{2+\frac{\beta}{N}}\right)\fG_\beta(B)-\fG_\beta(E_1,E_2)\\
	&\le \fD_\beta(E)+\left(1-\left(1-\frac{|E_2|}{\omega_N}\right)^{2}\right)\fG_\beta(B)-\fG_\beta(E_1,E_2)\\
	&\le \fD_\beta(E)+|E_2|\left(\frac{2}{\omega_N}-\frac{|E_2|}{\omega_N^2}\right)\fG_\beta(B)-\fG_\beta(E_1,E_2) .
	\end{align*}
	Now let us observe that, since $|E_1|\ge \frac{\omega_N}{k_R}$ and $d(E_1,E_2)\ge R$,
	\begin{equation}\label{est2}
	\fG_\beta(E_1,E_2)\ge |E_2||E_1|R^{\beta}\ge \frac{R^{\beta} \omega_N}{k_R}|E_2|,
	\end{equation} 
	thus it holds
	\begin{align*}
	\fG_\beta(E_2,E_1)&\le \fD_\beta(E)+|E_2|\left(\frac{2}{\omega_N}\fG_\beta(B)-\frac{R^{\beta} \omega_N}{k_R}-\frac{|E_2|}{\omega_N^2}\fG_\beta(B)\right).
	\end{align*}
	Now we want
	\begin{equation}\label{eq:estnec}
	\frac{2}{\omega_N}\fG_\beta(B)-\frac{R^{\beta} \omega_N}{k_R}<-1
	\end{equation}
	that is equivalent to say
	\begin{equation*}
	k_R<\frac{R^{\beta} \omega^2_N}{2\fG_\beta(B)+\omega_N}.
	\end{equation*}
	For $k_R$ to exists, we have to ask $\frac{R^{\beta} \omega^2_N}{2\fG_\beta(B)+\omega_N}>1$, that is to say $R>\left(\frac{2\fG_\beta(B)+\omega_N}{\omega^2_N}\right)^{\frac{1}{\beta}}$.
	Now, we $R:=2\left(\frac{2\fG_\beta(B)+\omega_N}{\omega^2_N}\right)^{\frac{1}{\beta}}+2$ and choose $k_R>1$ such that \eqref{eq:estnec} holds. We have
	\begin{align*}
	\fG_\beta(E_2,E_1)&\le \fD_\beta(E)+|E_2|\left(-1-\frac{\delta_R}{\omega_N^2}\fG_\beta(B)\right)\le \fD_\beta(E).
	\end{align*}
	Now, by using again Equation \eqref{est2}, we get
	\begin{equation*}
	|E_2|\le \frac{k_R}{R^\beta \omega_N}\fD(E).
	\end{equation*}
	Now we can estimate $\fD(\widetilde{E})$. To do this, let us observe that
	\begin{align*}
	\fG_\beta(\widetilde{E})&\le \fG_\beta(E)+\fG_\beta(\widetilde{E}_2)+2\fG_\beta(E_1,\widetilde{E_2})+2\fG_\beta(E_3,\widetilde{E_2})\\
	&\le \fD_\beta(E)+\fG_\beta(B)+2\fG_\beta(\widetilde{E},\widetilde{E_2}).
	\end{align*}
	Now let us recall that, by construction, ${\rm diam}(\widetilde{E})\le 2R$, hence
	\begin{equation*}
	\fG_\beta(\widetilde{E},\widetilde{E_2})\le 2^\beta R^\beta \omega_N |E_2|\le 2^\beta k_R \fD_\beta(E).
	\end{equation*} 
	Thus we finally get
	\begin{equation*}
	\fG_\beta(\widetilde{E})\le (1+2^{\beta+1} k_R)\fD_\beta(E)+\fG_\beta(B) 
	\end{equation*}
	and then $\fD_\beta(\widetilde{E})\le (1+2^{\beta+1} k_R)\fD_\beta(E)$. Setting $K=(1+2^{\beta+1} k_R)$ we have the estimate on the deficit.\\
	Finally, let us observe that $E \Delta \widetilde{E} \subseteq \R^N \setminus B_{1+R/3}$ and, in particular, we have $|\widetilde{E} \Delta B|=|E \Delta B|=\delta(E)$. On the other hand, for any $x \in \R^N$ such that $|x|\le R/3$ it holds
	\begin{equation*}
	|\widetilde{E}\Delta B(x)|=|E\Delta B(x)|\ge \delta(E).
	\end{equation*}
	Finally, if $|x|>R/3$, then, chosen any $y \in \R^N$ such that $|y|=\frac{R}{3}$, $|B \Delta B(x)|\ge |B \setminus B(y)|:=C_1$, where $C_1$ does only depend on $R$ and $N$. Thus we have
	\begin{equation*} 
	|\widetilde{E}\Delta B(x)|\ge |B\Delta B(x)|-|\widetilde{E}\Delta E|-|E\Delta B(x)| \ge C_1-3\delta_0.
	\end{equation*}
	Now, since we have fixed $R$, we can chose $\delta_0$ small enough to have $C_1-3\delta_0> \delta_0$, to obtain $|\widetilde{E}\Delta B(x)|> \delta(E)$. Then, taking the infimum on $x \in \R^N$ we obtain $\delta(\widetilde{E})=\delta(E)$.
\end{proof}
Now that we have shown we can reduce to the case in which $E \subseteq B_R$ for some universal radius $R>0$ (if $\delta(E)$ is small enough), let us proceed with step $2$ of our plan. So, let us show the following Lemma.
\begin{lem}\label{nearaball}
	Fix $\varepsilon \in (0,1)$. Then there exists a positive constant $\delta_\varepsilon>0$ such that the following property holds: if $E \subseteq B_R$, where $R$ is defined in Lemma \ref{lembound}, is a Borel set such that $|E|=\omega_N$ and $\delta(E)=|E \Delta B|<\delta_0$, then there exists a Borel set $\widetilde{E}$ such that $|\widetilde{E}|=\omega_N$, $B_{1-\varepsilon^2}\subseteq \widetilde{E} \subseteq B_{1+\varepsilon^2}$, $\fD_\beta(\widetilde{E})\le \fD_\beta(E)$ and $\delta(\widetilde{E})=\delta(E)=|\widetilde{E}\Delta B|$.
\end{lem}
\begin{proof}
	Fix $\varepsilon \in (0,1)$ and observe, as before, that if $\delta(E)<\delta_\varepsilon$, we have $|E \setminus B|<\frac{\delta_\varepsilon}{2}$. We divide the following proof in two parts: first we move the part of $E$ that is outside $B_{1+\varepsilon^2}$ inside this ball, defining a new set $E_1$; then we use the mass in $E_1\setminus B_{1-\varepsilon^2}$ to fill the holes in $E_1 \cap B_{1-\varepsilon^2}$, defining the set $\widetilde{E}$.\\
	Arguing as in the proof of Lemma \ref{lembound}, if we consider $G=E \setminus B_{1+\varepsilon^2}$, it holds $|G|<\delta_\varepsilon$, and, for $\delta_\varepsilon$ small enough, there is enough room to construct a set $\widetilde{G} \subseteq A_1\setminus E$, where $A_1=B_{1+\frac{\varepsilon^2}{2}}\setminus B_{1+\frac{\varepsilon^2}{3}}$, such that $|\widetilde{G}|=|G|$.\\
	 Define the set $E_1=(E \cup \widetilde{G})\setminus G$ and observe that $|E_1|=|E|$. Now let us show that $\fG_\beta(E_1)\le \fG_\beta(E)$. To do this let us observe that
	\begin{multline*}
	\fG_\beta(E_1)\le \fG_\beta(E)+\fG_\beta(\widetilde{G})+2\fG_\beta(B,\widetilde{G})\\+2\fG_\beta(E \setminus B,\widetilde{G})-2\fG_\beta(B,G)+2\fG_\beta(B \setminus E,G).
	\end{multline*}
	Now let us observe that, being $E \subset B_R$,
	\begin{equation*}
	\fG_\beta(\widetilde{G})+2\fG_\beta(E \setminus B,\widetilde{G})+2\fG_\beta(B \setminus E,G)\le |G|\delta_\varepsilon\left(4(2R)^\beta+(2+\varepsilon^2)^\beta\right).
	\end{equation*}
	Moreover, we have, recalling the definition of the function $\psi$ given in Equation \eqref{eq:psi} and the fact that it is a strictly increasing function,
	\begin{align*}
	\fG_\beta(B,\widetilde{G})&=\int_{\widetilde{G}}\psi(|x|)dx\le |G|\psi\left(1+\frac{\varepsilon^2}{2}\right),\\ \fG_\beta(B,G)&=\int_{G}\psi(|x|)dx\ge |G|\psi\left(1+\varepsilon^2\right),
	\end{align*}
	obtaining
	\begin{align*}
	\fG_\beta(E_1)&\le \fG_\beta(E)+|G|\left(\psi\left(1+\frac{\varepsilon^2}{2}\right)-\psi(1+\varepsilon^2)+\delta_\varepsilon\left(4(2R)^\beta+(2+\varepsilon^2)^\beta\right)\right).
	\end{align*}
	Finally, we can chose $\delta_\varepsilon$ small enough to have
	\begin{equation*}
	\psi\left(1+\frac{\varepsilon^2}{2}\right)-\psi(1+\varepsilon^2)+\delta_\varepsilon\left(4(2R)^\beta+(2+\varepsilon^2)^\beta\right)<0
	\end{equation*}
	and then $\fG_\beta(E_1)\le \fG_\beta(E)$.\\
	Now we want to modify again $E_1$ in such a way to fill with some mass the holes in $B_{1-\varepsilon^2}$. To do this, let us consider the set $H=B_{1-\varepsilon^2}\setminus E_1$ with measure $|H|<\frac{\delta_\varepsilon}{2}$. Now let us consider the mass of $E_1$ contained in $A_2=B_{1-\varepsilon^2/3}\setminus B_{1-\varepsilon^2/2}$. Since $|A_2 \setminus E_1|<\frac{\delta_\varepsilon}{2}$, we have $|A_2 \cap E_1|>\omega_N-\frac{\delta_\varepsilon}{2}$. Thus we can choose $\delta_\varepsilon$ small enough such that we can define $\widetilde{H}\subseteq A_2 \cap E$ with $|\widetilde{H}|=|H|$.\\
	Now let us define $\widetilde{E}=(E_1 \cup H)\setminus \widetilde{H}$ and observe that $|\widetilde{E}|=\omega_N$. Arguing exactly as before we get $\fG_\beta(\widetilde{E})\le \fG_\beta(E_1) \le \fG_\beta(E)$.\\
	Now let us observe that, by construction, $\widetilde{E}\Delta E \subseteq B_{1-\frac{\varepsilon^2}{3}}\cup \left(\R^N\setminus B_{1+\frac{\varepsilon^2}{3}}\right)$. By the way we modified the set, we got $|\widetilde{E} \Delta B|=|E \Delta B|=\delta(E)$.\\
	In general, for any $x \in \R^N$ with $|x|\le \frac{\varepsilon^2}{3}$ we have $|\widetilde{E} \Delta B(x)|=|E \Delta B(x)|\ge \delta(E)$ while for $|x|>\frac{\varepsilon^2}{3}$ we have
	\begin{equation*}
	|\widetilde{E} \Delta B(x)|\ge |B \Delta B(x)|-|B\Delta E|-|\widetilde{E}\Delta E|\ge |B \Delta B(x)|-3\delta_\varepsilon
	\end{equation*}
	hence, since $|B \Delta B(x)|\ge C$ for some constant $C$ depending only on $\varepsilon$, one can choose $\delta_\varepsilon<\frac{C}{4}$ to have $|\widetilde{E} \Delta B(x)|>\delta_\varepsilon$, concluding that $\delta(\widetilde{E})=\delta(E)$ and the optimal ball is still $B$.
\end{proof}
Now that we can construct a set that is uniformly close to a ball, let us show that if \eqref{quantisop} is not verified with a sufficiently large constant, we can reduce to a nearly-spherical set. Before doing this, we need to show a preliminary result that will ultimately lead to the nearly-spherical set.
\begin{lem}\label{lem:prens}
	There exist two constants $\varepsilon_1 \in (0,1)$ and $\widetilde{C}>0$ depending only on $\beta>0$ and $N$ such that for any $\varepsilon \in (0,\varepsilon_1)$, any $E \subseteq \R^N$ with $|E|=\omega_N$ and any $z \in \R^N$ such that $B_{1-\varepsilon}(z)\subseteq E \subseteq B_{1+\varepsilon}(z)$ one of the following properties holds:
	\begin{itemize}
		\item $E$ satisfies estimate \eqref{quantisop} with the constant $\widetilde{C}$;
		\item There exist two functions $u^{\pm}_z:\mathbb{S}^{N-1}\to [0,\varepsilon)$ such that the set
		\begin{equation*}
		E'_z=\left\{z+tx: \ t \in [0,1-u_z^-(x)) \cup (1,1+u_z^+(x)), \ x \in \mathbb{S}^{N-1}\right\}
		\end{equation*}
		has volume $|E'_z|=\omega_N$, satisfies $\fD_\beta(E'_z)\le \fD_\beta(E)$,  $\delta(E'_z)\ge \frac{\delta(E)}{2}$ and the functions $u^{\pm}_z$ depend continuously on the parameter $z$.
	\end{itemize}
\end{lem}
\begin{proof}
	Let us consider the quantities, for any $x \in \mathbb{S}^{N-1}$
	\begin{equation*}
	M^+_z(x)=\int_1^{+\infty}t^{N-1}\chi_{E}(z+tx)dt \mbox{ and } M^-_z(x)=\int_0^{1}t^{N-1}\chi_{\R^N \setminus E}(z+tx)dt.
	\end{equation*}
	These two quantities are continuous functions of $z \in \R^N$ and then, defining $u_z^{\pm}$ by
	\begin{equation*}
	u_z^+(x)=(NM^+_z(x)+1)^{\frac{1}{N}}-1 \mbox{ and } 1-(1-u_z^-(x))^N=1-(1-NM^-_z(x))^{\frac{1}{N}},
	\end{equation*}
	we know that the latter depend continuously on $z \in \R^N$ and
	\begin{equation*}
	\int_1^{1+u^+_z(x)}t^{N-1}dt=M^+_z(x) \mbox{ and } \int^1_{1-u^+_z(x)}t^{N-1}dt=M^-_z(x).
	\end{equation*}
	Now let us construct $E'_z$ as declared and let us consider the following sets
	\begin{align*}
	G^+=(E \setminus E'_z)\setminus B(z) && \widetilde{G}^+=(E'_z \setminus E)\setminus B(z)\\
	G^-=B(z)\cap (E'_z \setminus E) && \widetilde{G}^-=B(z)\cap (E \setminus E'_z)
	\end{align*}
	observing that $\widetilde{G}^+\subseteq B_{1+\varepsilon}(z)\setminus B(z)$ and $\widetilde{G}^- \subseteq B(z)\setminus B_{1-\varepsilon}(z)$ by definition. It is not difficult to check that $|E'_z|=\omega_N$. Indeed, we have
	\begin{align*}
	|E'_z|&=\int_{\mathbb{S}^{N-1}}\left(\int_0^{1-u^-_z(x)}t^{N-1}dt+\int_1^{1+u^-_z(x)}t^{N-1}dt\right)d\cH^{N-1}(x)\\
	&=\int_{\mathbb{S}^{N-1}}\left(1-\int_{1-u^-_z(x)}^1t^{N-1}dt+\int_1^{1+u^-_z(x)}t^{N-1}dt\right)d\cH^{N-1}\\
	&=\int_{\mathbb{S}^{N-1}}\left(\int_{0}^1t^{N-1}(1-\chi_{\R^N \setminus E}(z+tx))dt\right.\\&\qquad\left.+\int_1^{+\infty}t^{N-1}\chi_{\R^N \setminus E}(z+tx)dt\right)d\cH^{N-1}(x)\\
	&=\int_{\mathbb{S}^{N-1}}\int_0^{+\infty}t^{N-1}\chi_{E}(z+tx)dtd\cH^{N-1}(x)=|E|.
	\end{align*}
	Being $|E'_z|=|E|$, then we have $|E'_z\setminus E|=|E \setminus E'_z|$ and, in particular $|\widetilde{G}^\pm|=|G^\pm|$.\\
	Let us define $H_z=\widetilde{G}^+ \cup G^-$ and $K_z=G^+ \cup \widetilde{G}^-$. By definition $|H_z|=|K_z|$. Now let us define the following transport map from $H_z$ to $K_z$. For $y \in H_z$, we define $\Phi_z(y)=\varphi(y)\frac{y-z}{|y-z|}+z$ where, if $|y-z| \ge 1$, we have
	\begin{equation*}
	\int_1^{|y-z|}\chi_{E'_z \setminus E}\left(z+t\frac{y-z}{|y-z|}\right)t^{N-1}dt=\int_1^{\varphi(y)}\chi_{E \setminus E'_z}\left(z+t\frac{y-z}{|y-z|}\right)t^{N-1}dt
	\end{equation*}
	and if $|y-z|<1$
	\begin{equation*}
	\int_{|y-z|}^1\chi_{E'_z \setminus E}\left(z+t\frac{y-z}{|y-z|}\right)t^{N-1}dt=\int_{\varphi(y)}^1\chi_{E \setminus E'_z}\left(z+t\frac{y-z}{|y-z|}\right)t^{N-1}dt.
	\end{equation*}
	This map is actually simple to describe by words. Both $H_z$ and $K_z$ are constituted by a part that is inside the ball $B$ and a part that is outside the ball $B$. The map $\Phi_z$ actually considers the angular coordinate of $y$ with respect to the ball $B(z)$ (i. e. its projection on $\partial B(z)$) and modify its radius in such a way to send the part outside $B$ in $H_z$ to the one outside $B$ in $K_z$ and the same for the inside, while preserving the volume of the part that is moving. This is actually a generalization of a Knothe-Rosenblatt rearrangement as described in Section \ref{Sec2} by means of radial and angular components of the sets (hence with respect to the $1$-dimensional Lebesgue measure on $\R$ and the $N-1$-dimensional Hausdorff measure on $S^{N-1}$). In particular, $\Phi$ is an invertible transport map between $H_z$ and $K_z$.\\
	Now let us work with the energies of $E'_z$ and $E$. We have
	\begin{equation*}
	\fG_\beta(E'_z)=\fG_\beta(E)+\fG_\beta(H_z,E)+\fG_\beta(H_z,E'_z)-\fG_\beta(K_z,E)-\fG_\beta(K_z,E'_z).
	\end{equation*}
	Let us split everything with respect to the ball $B(z)$ to achieve
	\begin{align*}
	\fG_\beta(E'_z)-\fG_\beta(E)&=2\fG_\beta(H_z,B(z))-2\fG_\beta(K_z,B(z))\\&+\fG_\beta(H_z,E \setminus B(z))-\fG_\beta(K_z,E \setminus B(z))\\&+\fG_\beta(H_z,E'_z \setminus B(z))-\fG_\beta(K_z,E'_z \setminus B(z))\\
	&+\fG_\beta(K_z,B(z)\setminus E'_z)-\fG_\beta(H_z,B(z)\setminus E'_z)\\
	&+\fG_\beta(K_z,B(z)\setminus E)-\fG_\beta(H_z,B(z)\setminus E)
	\end{align*}
	By using Equation \eqref{eq:transcontr} we obtain
	\begin{align*}
	\fG_\beta(E)-\fG_\beta(E'_z)&\ge 2\fG_\beta(K_z,B(z))-2\fG_\beta(H_z,B(z))-2C\delta^\alpha_\varepsilon\int_{H_z}|y-\Phi(y)|dy,
	\end{align*}
where $\alpha$ is defined in equation \eqref{eq:alphadef} and $C_1>0$ depends only on $\beta$ and $N$.\\
	Now let us also observe that for $y \in H_z$ it holds $\frac{y-z}{|y-z|}=\frac{\Phi_z(y)-z}{|\Phi_z(y)-z|}$. However $|\Phi_z(y)-z|=\varphi(y)\ge |y-z|$. Being $\psi$ increasing and $C^1$ we have
	\begin{equation*}
	\psi(|\Phi_z(y)-z|)-\psi(|y-z|)\ge c_\varepsilon|y-\Phi_z(y)|
	\end{equation*}
	where $c_\varepsilon=\min_{t \in [1-\varepsilon,1+\varepsilon]}\psi'(t)$. Integrating this relation over $H_z$ we get
	\begin{equation*}
	\fG_\beta(K_z,B(z))-\fG_\beta(H_z,B(z))\ge c_\varepsilon \int_{H_z}|y-\Phi_z(y)|dy
	\end{equation*}
	and then
	\begin{equation}\label{eq:est3}
	\fG_\beta(E)-\fG_\beta(E'_z)\ge  2\left(c_\varepsilon-C_1\delta_\varepsilon^\alpha\right)\int_{H_z}|y-\Phi_z(y)|dy.
	\end{equation}
	Being $\psi'$ continuous near $1$, as $\varepsilon \to 0$ we have that $c_\varepsilon \to \psi'(1)>0$, hence we can consider $\varepsilon_1$ small enough to have $c_\varepsilon>0$ and then $\delta_\varepsilon$ small enough to have $c_\varepsilon-C_1\delta^\alpha_\varepsilon>0$ and finally $\fG_\beta(E)\ge \fG_\beta(E'_z)$.\\
	Let us remark that up to this point we have not used the fact that $E$ does not satisfy Equation \eqref{quantisop}. So now we have to show that one of the properties hold. In particular, let us suppose that $\delta(E'_z) \le \delta(E)/2$. Thus there exists a ball $B'$ such that $\delta(E'_z)=|B'\Delta E'_z|\le \frac{\delta(E)}{2}$. We claim that $E$ satisfies Equation \eqref{quantisop}.\\
	To do this let us first observe that
	\begin{equation*}
	\delta(E)\le |E \Delta B'|\le |E \Delta E'_z|+|E'_z \Delta B'|\le |E \Delta E'_z|+\delta(E)/2
	\end{equation*}
	and that
	\begin{equation*}
	|H_z|=\frac{|E \Delta E'_z|}{2}\le \frac{\delta(E)}{4}.
	\end{equation*}
	Now let us argue by slicing $H_z$. Fix $\nu \in \mathbb{S}^{N-1}$ and let $G_\nu=H \cap \nu \R$ be the section of $H_z$ in direction $\nu$. Let us split this set in the part interior and exterior to the ball $B(z)$, i. e. $G_\nu=G_\nu^+\cup G_\nu^-$ where $G_\nu^+=G_\nu \setminus B(z)$ and $G_\nu^-=G_\nu \cap B(z)$. Now, let us observe that, by construction, $(E'_z \setminus B) \cap \nu\R$ is the segment $(1,1+u_z^+(\nu))\nu$. Thus, by construction, $\widetilde{G}^+ \cap \nu\R$ is in this segment and $G^+ \cap \nu \R$ is outside the segment. A similar argument holds for $G^-$ and $\widetilde{G}^-$.
	Thus, if we set $L_\nu^\pm=\cH^1(G_\nu^\pm)$, since the subset of $G_\nu^{+}$ made by those points for which $|\Phi_z(y)-y|\ge \frac{L^+_\nu}{2}$ has length at least $\frac{L^+_\nu}{2}$, we have
	\begin{equation*}
	\int_{G_\nu^+}|\Phi_z(y)-y|d\cH^1(y)\ge \left(\frac{L^+_\nu}{2}\right)^2.
	\end{equation*}
	Arguing in the same way, we have
	\begin{equation*}
	\int_{G_\nu^-}|\Phi_z(y)-y|d\cH^1(y)\ge \left(\frac{L^-_\nu}{2}\right)^2,
	\end{equation*}
	and then summing
	\begin{equation*}
	\int_{G_\nu}|\Phi_z(y)-y|d\cH^1(y)\ge \left(\frac{L^-_\nu}{2}\right)^2+\left(\frac{L^+_\nu}{2}\right)^2\ge\frac{L_\nu^2}{8}
	\end{equation*}
	where $L(\nu)=\cH^1(G_\nu)$. Now let us reconstruct the measure of $H_z$ in terms of the sections. We have, by Coarea formula,
	\begin{align*}
	|H_z|&\le (1+\varepsilon)^{N-1}\int_{\mathbb{S}^{N-1}}L_\nu d\cH^{N-1}(\nu)\\
	&\le (1+\varepsilon)^{N-1}\sqrt{N\omega_N}\sqrt{\int_{\mathbb{S}^{N-1}}L^2_\nu d\cH^{N-1}(\nu)}\\
	&\le (1+\varepsilon)^{N-1}\sqrt{8N\omega_N}\sqrt{\int_{\mathbb{S}^{N-1}}\int_{G_\nu}|\Phi_z(y)-y|d\cH^1(y)\cH^{N-1}(\nu)}\\
	&\le \frac{(1+\varepsilon)^{N-1}}{(1-\varepsilon)^{\frac{N-1}{2}}}\sqrt{8N\omega_N}\sqrt{\int_{H_z}|\Phi_z(y)-y|dy}.
	\end{align*}
	However, we also have $|H_z|\ge \frac{\delta(E)}{4}$, hence
	\begin{equation*}
	\delta(E)\le 4\frac{(1+\varepsilon)^{N-1}}{(1-\varepsilon)^{\frac{N-1}{2}}}\sqrt{8N\omega_N}\sqrt{\int_{H_z}|\Phi_z(y)-y|dy}.
	\end{equation*}
	On the other hand, we have shown in Equation \eqref{eq:est3} that
	\begin{equation*}
	\int_{H}|\Phi_z(y)-y|dy\le C(\fG_\beta(E)-\fG_\beta(E'_z))\le C \fD_\beta(E) 
	\end{equation*}
	where, for $\varepsilon$ and $\delta_\varepsilon$ small enough, $C$ is a universal constant. Thus we have, for $\varepsilon$ small enough,
	\begin{equation*}
	\delta(E)\le \widetilde{C} \sqrt{\fD_\beta(E)}
	\end{equation*}
	where $\widetilde{C}$ is some universal constant. This completes the proof. 
\end{proof}
Now that we have proved this intermediate step, we can use an approximation of the functions $u^{\pm}_z$ by a locally constant function to construct the nearly-spherical set we are searching for.
\begin{prop}\label{propns}
	There exists a positive constant $\varepsilon_1$, depending only on $N$ and $\beta$, such that for any $\varepsilon \in (0,\varepsilon_1)$, for any $E \subseteq \R^N$ with $|E|=\omega_N$, $E \subseteq B_R$, where $R$ is the radius defined in Lemma \ref{lembound}, and for any $z \in \R^N$ such that $B_{1-\varepsilon}(z)\subseteq E \subseteq B_{1+\varepsilon}(z)$, one of the two following properties hold:
	\begin{itemize}
		\item $E$ satisfies estimate \eqref{quantisop} with the constant $\widetilde{C}$ defined in Lemma \ref{lem:prens};
		\item There exists a set $E_z$ that is nearly-spherical around $B(z)$, $E_z=z+E_{1,u_z}$ where $E_{1,u_z}$ is defined in \eqref{eq:nearlyspherical}, $u_z$ depends continuously on z, $\Norm{u_z}{L^\infty(\mathbb{S}^{N-1})}\le \varepsilon$, $\fD_\beta(E_z)\le 2\fD_\beta(E)$ and $|E_z \Delta B(z)|\ge \frac{\delta(E)}{6}$. Moreover, the baricenter ${\rm Bar}(z)$ of $E_z$ depends continuously on $z$.
	\end{itemize}  
\end{prop}
\begin{proof}
	Let us suppose $E$ does not satisfy inequality \eqref{quantisop} with the constant $\widetilde{C}$ defined in Lemma \ref{lem:prens}. Let us consider the set $E'_z$ defined in Lemma \ref{lem:prens} and $u^\pm_z:\mathbb{S}^{N-1} \to [0,\varepsilon)$ as before.\\
	Being $u^{\pm}_z$ non-negative, there exist two locally constant functions $\widetilde{u}^{\pm}_z$, close in $L^\infty(S^{N-1})$ to $u^{\pm}_z$ as much as we want, with values $u^\pm_{z,i}$ on a family of finitely many measurable sets $U_i \subseteq S^{N-1}$ such that ${\rm diam}(U_i)\le \min\{u_{z,i}^+,u_{z,i}^-\}$ whenever $\min\{u_{z,i}^+,u_{z,i}^-\}>0$. Now we define $E''_z$ as
	\begin{equation*}
	E''_z=\left\{z+tx: \ t \in [0,1-\widetilde{u}_z^-(x)) \cup (1,1+\widetilde{u}_z^+(x)), \ x \in \mathbb{S}^{N-1}\right\}.
	\end{equation*}
	Note that we can choose $u_{z,i}^\pm$ in such a way that $|E''_z|=\omega_N$ and so that $\widetilde{u}^{\pm}_z$ depends continuously on $z$.\\
	By Lemma \ref{quantisop}, we have $\fD_\beta(E'_z)\le \fD_\beta(E)$, hence, since $\widetilde{u}_z^{\pm}$ can be chosen uniformly close to $u_z^{\pm}$ as we need, we may assume that $\fD_\beta(E''_z)\le 2\fD_\beta(E)$. For the same reason we may also assume that $\delta(E''_z)\ge \frac{\delta(E_z')}{3}$.\\
	Now let us construct the nearly spherical set $E_z$. To do this, let us work locally on each $U_i$. If $\min\{u^+_{z,i},u^-_{z,i}\}=0$, then we define, for any $\omega \in U_i$, $u_z(\omega)=\pm u_{z,i}^\pm$ if $u_{z,i}^\mp=0$. On the other hand, if $\min\{u^+_{z,i},u^-_{z,i}\}>0$ we can subdivide $U_i$ in two sets $L_i$ and $R_i$ such that
	\begin{equation*}
	\cH^{N-1}(L_i)(1-(1-u_{z,i}^-)^N)=\cH^{N-1}(R_i)((1+u_{z,i}^+)^N-1)
	\end{equation*}
	and define $u_z$ as $u_z(\omega)=\chi_{L_i}u_{z,i}^+-\chi_{R_i}u_{z,i}^-$ for $\omega \in U_i$. Moreover, since we have chosen $u_{z,i}^\pm$ in such a way that $\widetilde{u}_z^\pm$ depend with continuity on $z$, also $u_z$ is continuous with respect to $z$.\\
	Finally, let us define $E_z$ as
	\begin{equation*}
	E_z=\{z+(1+\rho)x: x \in \mathbb{S}^{N-1}, -1 \le \rho \le u_z(x)\}.
	\end{equation*}
	Note that $E_z$ is nearly spherical. Moreover, observe that since $u_z$ depends continuously on $z$, then also ${\rm Bar}(z)$ is a continuous function of $z$.\\
	By construction we have $|E_z|=\omega_N$. Let us first work with the symmetric difference of $E_z$ with respect to the ball $B(z)$. To do this, let us consider the different contributes of $E_z$ and $E''_z$ outside and inside the ball, as
	\begin{equation*}
	F=E_z \setminus B(z), \quad D=B(z) \setminus E_z, \quad \widetilde{F}=E''_z\setminus B(z), \quad \widetilde{D}=B(z)\setminus E''_z,
	\end{equation*}
	since the behaviour of $u_z$ is different depending on the sets $U_i$ in which it is defined, it can be useful to consider the cones $K_i$ with vertices in $z$ and such that $K_i \cap \partial B_z=z+U_i$, and define
	\begin{equation*}
	F_i=F \cap K_i, \quad D_i=D \cap K_i, \quad \widetilde{F}_i=\widetilde{F}\cap K_i, \quad \widetilde{D}_i=\widetilde{D} \cap K_i.
	\end{equation*}
	By construction of $u_z$, we have $-\widetilde{u}_z^- \le u_z \le \widetilde{u}_z^+$ on the whole sphere, so in particular $F \subseteq \widetilde{F}$ and $D \subseteq \widetilde{D}$.\\
	In particular, the same inclusions hold for any $F_i,\widetilde{F}_i, D_i, \widetilde{D}_i$. However, $F_i$ is empty if and only if $u^+_{z,i}=0$, which implies also that $\widetilde{F}_i$ is empty. If it is not, then let us distinguish two cases. If $u^-_{z,i}=0$, then $\widetilde{F}_i=F_i$. Otherwise, it holds $\min\{u^-_{z,i},u^+_{z,i}\}>0$, and then we have
	\begin{equation*}
	\widetilde{F}_i=(L_i \cup R_i)\times (1,1+u^+_{z,i}).
	\end{equation*}
	On the other hand, we have $u_z=-u^-_{z,i}$ for $x \in R_i$ and $u_z=u^+_{z,i}$ for $x \in L_i$, hence
	\begin{equation*}
	F_i=L_i \times (1,1+u^+_{z,i}).
	\end{equation*}
	With the same reasoning on $D_i$ we obtain:
	\begin{equation*}
	D_i=\begin{cases} \widetilde{D}_i & \min\{u^+_{z,i},u^-_{z,i}\}=0 \\
	R_i \times (1-u^-_{z,i},1) & \min\{u^+_{z,i},u^-_{z,i}\}>0
	\end{cases}
	\end{equation*}
	while, if $\min\{u^+_{z,i},u^-_{z,i}\}>0$, $\widetilde{D}_i=(L_i \cup R_i)\times (1-u^-_{z,i},1)$. Hence, if $\min\{u^+_{z,i},u^-_{z,i}\}=0$, then we easily obtain $|D_i|+|F_i|=|\widetilde{D}_i|+|\widetilde{F}_i|$.\\
	Now let us consider the case in which $\min\{u^+_{z,i},u^-_{z,i}\}>0$. We have
	\begin{equation*}
	|F_i|+|D_i|=\cH^{N-1}(L_i)((1+u_{z,i}^+)^N-1)+\cH^{N-1}(R_i)(1-(1-u_{z,i}^-)^N)
	\end{equation*}
	while
	\begin{equation*}
	|\widetilde{F}_i|+|\widetilde{D}_i|=(\cH^{N-1}(L_i)+\cH^{N-1}(R_i))((1+u_{z,i}^+)^N-1)+(1-(1-u_{z,i}^-)^N)).
	\end{equation*}
	From these two relations it is easy to check that
	\begin{equation*}
	|F_i|+|D_i|\ge \frac{|\widetilde{F}_i|+|\widetilde{D}_i|}{2}
	\end{equation*}
	for any $i$. Finally, summing over $i$, we obtain
	\begin{equation*}
	|E_z \Delta B(z)|\ge \frac{|E''_z \Delta B(z)|}{2}\ge\frac{\delta(E''_z)}{2}\ge \frac{\delta(E)}{6}.
	\end{equation*}
	Now let us work with $\fG_\beta(E_z)$. To do this, we need to construct an invertible transport map between $\widetilde{F}_i \setminus F_i$ and $\widetilde{D}_i \setminus D_i$ only for $i$ such that $\min\{u^+_{z,i},u^-_{z,i}\}>0$, since in the other case we have the equalities $\widetilde{F}_i=F_i$ and $\widetilde{D}_i=D_i$. First let us remark that, by construction, $|\widetilde{F}_i \setminus F_i|=|\widetilde{D}_i \setminus D_i|$. By definition of $R_i$ and $L_i$, we have that $\chi_{R_i}$ and $\frac{1-(1-u_{z,i}^-)^N}{(1+u_{z,i}^+)^N-1}\chi_{L_i}$ admit the same $L^1$ norm with respect to the Hausdorff measure $\cH^{N-1}$, thus we can construct an invertible transport map between them. In particular, let us construct a Knothe-Rosenblatt transport map $\tau_i$ between $R_i$ and $L_i$ that preserve such norms (thus by disintegration and conditioning). More precisely, defining
	\begin{equation*}
	g_i(t)=\left(\frac{\cH^{N-1}(R_i)(t^N-1)+\cH^{N-1}(L_i)(1-u_{z,i}^-)^N}{\cH^{N-1}(L_i)}\right)^{\frac{1}{N}}
	\end{equation*}
	we can construct the transport map $\Phi_i: \widetilde{F_i}\setminus F_i \to \widetilde{D_i}\setminus D_i$ as
	\begin{equation*}
	\Phi_i(t\nu)=g_i(t)\tau_i(\nu) \quad \ \forall \nu \in R_i, \ t \in(1,1+u_{z,i}^+).
	\end{equation*}
	In particular, this is an invertible transport map and the volume distortion caused by the Knothe-Rosenblatt rearrangement $\tau_i$ (which preserves instead the $L^1(\cH^{N-1})$ norms of $\chi_{R_i}$ and $\frac{1-(1-u_{z,i}^-)^N}{(1+u_{z,i}^+)^N-1}\chi_{L_i}$) is balanced by the distortion on the interval given by $g_i$ (to preserve the volume of the whole set $|\widetilde{F}_i\setminus F_i|$ onto $|\widetilde{D}_i\setminus D_i|$).\\
	Moreover, $\tau_i(\nu)$ is a transport map on the sphere, so $|\tau_i(\nu)|=1$. Thus we have
	\begin{align*}
	|y|-|\Phi_i(y)|&=t-g_i(t)\ge \min\{u_{i,z}^+,u_{i,z}^-\}\ge {\rm diam}\,  U_i.
	\end{align*}
	On the other hand we have, since $g_i(t)\le 1$
	\begin{align*}
	|y-\Phi_i(y)|&\le |t-g_i(t)|+g_i(t)|\nu-\tau_i(\nu)|\\
	&\le (|y|-|\Phi_i(y)|)+g_i(t)|\nu-\tau_i(\nu)|\\
	&\le (|y|-|\Phi_i(y)|)+{\rm diam} U_i \\
	&\le 2(|y|-|\Phi_i(y)|).
	\end{align*}
	We can \textit{glue} all the $\Phi_i$ to construct a transport map $\Phi:\widetilde{F}\setminus F \to \widetilde{D}\setminus F$.\\
	Now we are ready to evaluate the energy of $E_z$. To do this, let us observe that
	\begin{align*}
	\fG_\beta(E_z'')-\fG_\beta(E_z)&=2\fG_\beta(B,\widetilde{F}\setminus F)-2\fG_\beta(B,\widetilde{D}\setminus D)\\
	&+\fG_\beta(\widetilde{D},\widetilde{D}\setminus D)-\fG_\beta(\widetilde{D},\widetilde{F}\setminus F)\\
	&+\fG_\beta(D,\widetilde{D}\setminus D)-\fG_\beta(D,\widetilde{F}\setminus F)\\
	&+\fG_\beta(\widetilde{F},\widetilde{F}\setminus F)-\fG_\beta(\widetilde{F},\widetilde{D}\setminus D)\\
	&+\fG_\beta(F,\widetilde{F}\setminus F)-\fG_\beta(F,\widetilde{D}\setminus D).
	\end{align*}
	By Equation \eqref{eq:transcontr} we know that
	\begin{align*}
	\fG_\beta(\widetilde{D},\widetilde{D}\setminus D)-\fG_\beta(\widetilde{D},\widetilde{F}\setminus F)&+\fG_\beta(D,\widetilde{D}\setminus D)-\fG_\beta(D,\widetilde{F}\setminus F)\\
	&+\fG_\beta(\widetilde{F},\widetilde{F}\setminus F)-\fG_\beta(\widetilde{F},\widetilde{D}\setminus D)\\
	&+\fG_\beta(F,\widetilde{F}\setminus F)-\fG_\beta(F,\widetilde{D}\setminus D)\\
	&\ge -C_1((1+\varepsilon)^N-(1-\varepsilon)^N)^\alpha\int_{\widetilde{F}\setminus F}|y-\Phi(y)| dy,
	\end{align*}
	for some constant $C_1$ depending only on $\beta$ and $N$ and $\alpha$ defined in \eqref{eq:alphadef}.
	On the other hand, denoting by $c_\varepsilon=\min_{t \in (1-\varepsilon,1+\varepsilon)}\psi'(t)$, we have, arguing as in Lemma \ref{lem:prens},
	\begin{equation*}
	\fG_\beta(B,\widetilde{F}\setminus F)-\fG_\beta(B,\widetilde{D}\setminus D)\ge c_\varepsilon\int_{\widetilde{F}\setminus F}|y-\Phi(y)|dy
	\end{equation*}
	hence
	\begin{align*}
	\fG_\beta(E_z'')-\fG_\beta(E_z)\ge\left(c_\varepsilon-C_1((1+\varepsilon)^N-(1-\varepsilon)^N)^\alpha\right)\int_{\widetilde{F}\setminus F}|y-\Phi(y)|dy.
	\end{align*}
	As $\varepsilon \to 0$ we have that $c_\varepsilon \to \psi'(1)>0$, hence we can consider $\varepsilon$ small enough to have $c_\varepsilon-C_1((1+\varepsilon)^N-(1-\varepsilon)^N)^\alpha>0$ and finally
	\begin{equation*}
	\fG_\beta(E_z'')-\fG_\beta(E_z)\ge 0,
	\end{equation*}
	that implies
	\begin{equation*}
	\fD_\beta(E_z)\le \fD_\beta(\widetilde{E}'')\le 2 \fD_\beta(E).
	\end{equation*}
\end{proof}
Now we turn to the fourth step of our plan. To do this, we have to show that the function ${\rm Bar}(z)$ admits a fixed point. Since this can be done by means of \cite[Lemmas $2.14$ and $2.15$]{fusco2019sharp}, we omit the proof.
\begin{lem}\label{baradj}
	Under the hypotheses of Proposition \ref{propns}, we can construct the set $E_z$ in such a way that ${\rm Bar}(z)=z$.
\end{lem}
\subsection{Proof of Theorem \ref{thm:quantrandisop}}
Now we have all the tools we need to prove Theorem \ref{thm:quantrandisop}.
\begin{proof}[Proof of Theorem \ref{thm:quantrandisop}]
	Let us consider $\varepsilon_1$ as defined in Proposition \ref{propns} and $\varepsilon_0$ as defined in Theorem \ref{nrthm}. Fix $\varepsilon \in \left(0,\min\left\{\varepsilon_1,\frac{\varepsilon_0}{2},1\right\}\right)$ and define $\delta_\varepsilon>0$ as in Lemma \ref{nearaball}. Now consider $\delta_0$ as in Lemma \ref{lembound} and fix $\mu \in (0,\min\{\delta_\varepsilon,\delta_0\})$. Fix $R$ as in Lemma \ref{lembound} and consider $\eta>0$ as in Lemma \ref{lemcompact} associated to this $\mu$.\\
	If $\delta(E)\ge \mu$, then $\fD_\beta(E)\ge \eta$ and we have
	\begin{equation*}
	\delta(E)\le 2\omega_N\le\frac{2\omega_N}{\sqrt{\eta}}\sqrt{\fD_\beta(E)}.
	\end{equation*} 
	Now suppose $\delta(E)<\mu$. Then, since $\delta(E)<\delta_0$, by Lemma \ref{lembound}, we can construct a set $\widetilde{E} \subseteq B_R$ such that $\delta(\widetilde{E})=\delta(E)$ and $\fD_\beta(\widetilde{E})\le K \fD_\beta(E)$ where $K$ depends only on $\beta$ and $N$. Now, by Lemma \ref{nearaball}, since $\delta<\delta_\varepsilon$, we can construct a set $E'$ with $B_{1-\varepsilon}\subseteq E' \subseteq B_{1+\varepsilon}$, $\fD_\beta(E')\le \fD(\widetilde{E})\le K \fD_\beta(E)$ and $\delta(E')=\delta(\widetilde{E})=\delta(E)$.\\
	If this set satisfies \eqref{quantisop} with the constant $\widetilde{C}$ defined in Lemma \ref{lem:prens}, we conclude the proof. Otherwise, we can use Proposition \ref{propns} (since $\varepsilon<\varepsilon_1$) to construct a nearly-spherical set $E''$ with volume $|E''|=\omega_N$ and barycentre in the origin (by also using Lemma \ref{baradj} and then translating the set in such a way that $z=0$). In particular, $|E'' \Delta B| \ge \frac{\delta(E)}{6}$ and $\fD_\beta(E'')\le 2K \fD_\beta(E)$. Moreover, since $\varepsilon<\frac{\varepsilon_0}{2}$, we can write
	\begin{equation*}
	E''=\{x \in \R^N: \ x=\rho z, \ z \in S^{N-1}, \ \rho \in (0,1+tu(z)]\}
	\end{equation*}
	with $\Norm{u}{L^\infty(S^{N-1})}\le 1/2$ and $t \in (0,\varepsilon_0)$. Thus, by Theorem \ref{nrthm}, we know that there exists a constant $C$ such that
	\begin{equation*}
	|E'' \Delta B|\le C \sqrt{\fD_\beta(E)}
	\end{equation*}
	concluding the proof.
\end{proof}
\begin{rmk}
	Comparing the constant $C(N,\beta)$ in Theorem \ref{thm:quantrandisop} with the one in Theorem \ref{nrthm} and by Remark \ref{rmk:asympD} we obtain $C(N,\beta)\ge \sqrt{\frac{8N\omega_N}{D_\beta}}$,
where $D_\beta$ is defined in Equation \eqref{eq:Dbeta}. The latter inequality, together with Remark \ref{rmk:asympD}, implies
\begin{align*}
	\liminf_{\beta \to \infty}C(N,\beta)\sqrt{\frac{\widetilde{C}_N^\infty 2^\beta \beta^{-\frac{N+1}{2}}}{N\omega_N}}\ge 1, && 	\liminf_{\beta \to 0^+}C(N,\beta)\sqrt{\frac{\widetilde{C}_N^0 \beta}{N\omega_N}}\ge 1,
\end{align*} 
where $\widetilde{C}_N^\infty$ and $\widetilde{C}_N^0$ are defined in Formula \eqref{eq:asympconst}.
\end{rmk}
\section{The minimizer of a mixed energy with a perimeter penalization}\label{Sec6}
Now let us consider the mixed energy functional
\begin{equation}\label{mixfunc}
\fE(E)=\fG_\beta(E)+\varepsilon P_s(E)+V_\alpha(E)
\end{equation}
for measurable sets $E \subseteq \R^N$, where $\beta>0$, $\varepsilon>0$, $\alpha \in (0,N)$, $s \in (0,1]$, $P_s$ is the fractional perimeter defined in Equation \eqref{fracper} for $s \in (0,1)$, $P_1:=P$ is the classical perimeter and $V_\alpha$ is the Riesz potential defined in Equation \eqref{Rieszpot}. We want to find a minimizer of $\fE$ under the volume constraint $|E|=m$.
Let us first recall that $P_s$ (for $s \in (0,1]$) satisfies the following isoperimetric inequality (see \cite{frank2008hardy}), setting $|E|=m$ and for fixed $s \in (0,1]$,
\begin{equation*}
P_s(E)\ge P_s(B[m]),
\end{equation*}
while the Riesz potential is maximized by the ball, i.e. for any $\alpha \in (0,N)$, by Riesz rearrangement inequality, it holds
\begin{equation*}
V_\alpha(E)\le V_\alpha(B[m]),
\end{equation*}
where $B[m]$ is the ball of volume $m$ and the equality holds in both inequality if and only if $E$ is a ball.
We want to show that there exists a critical mass $m_0$ such that if $m>m_0$, the ball is a minimizer of \eqref{mixfunc} for $\varepsilon>\varepsilon_0$, where $\varepsilon_0$ may depend on $m$. Note that if $\alpha \in (1,N)$ this result follows from \cite{frank2019proof} where it is proved that there exists a critical mass $m_0$ such that, if $m>m_0$, then the ball of mass $m$ is a minimizer for the mixed energy $\fG_\beta+V_\alpha$. On the other hand, by the result proved in \cite{figalli2015isoperimetry} we the note that the ball is a minimizer for $s \in (0,1]$, $\alpha \in (0,N)$, $\beta>0$ and any $\varepsilon>0$ if the mass is sufficiently small.\\
However, still in \cite{frank2019proof}, it has been shown that the characteristic function of a ball is not a critical point for $\fG_\beta+V_\alpha$ when $\alpha \in (0,1)$ for the problem when relaxed on $L^1$ functions. Thus, from this observation, by \cite[Theorem $4.4$]{burchard2015nonlocal}, we can conclude that the ball cannot be a minimizer for $\fG_\beta+V_\alpha$ for any mass constraint $m$.\\
Hence, what we aim to show is that if we add a penalization to the functional $\fG_\beta+V_\alpha$ with the (possibly fractional) perimeter, the new penalized functional admits a minimum when the volume constraint $m$ is above a critical mass $m_0$, and that the ball is actually a minimum over a second critical mass $m_1$.\\
The first thing we have to show is the actual existence of the minimizer. To do this, let us consider the shape functional
\begin{equation*}
\fE_m(E)=\fG_\beta(E)+V_\alpha(E)+\varepsilon(m) P_s(E)
\end{equation*}
where $\varepsilon(m)=\left(\frac{m}{\omega_N}\right)^{1+\frac{\beta+s}{N}}$ and let us show that there exists a minimizer for it as $m>m_0$. The same exact proof will show that there exists a minimizer for $\fE(E)$ as $m>m_0$ if $\varepsilon>\varepsilon(m)$.

\begin{lem}\label{constr}
	Let $N \ge 2$, $\alpha \in (0,N)$, $\beta>0$, $s \in (0,1]$ and set $\varepsilon(m)=\left(\frac{m}{\omega_N}\right)^{1+\frac{\beta+s}{N}}$. Then there exist two positive constants $m_0$ and $R_0$ depending on $\alpha,\beta,s,N$ such that for any $m > m_0$ it holds
	\begin{equation}\label{eq:constrprob}
		\inf\{\fE_m(E): \ |E|=m\}=\inf\left\{\fE_m(E): \ |E|=m, \ E \subseteq B_{\left(\frac{m}{\omega_N}\right)^{\frac{1}{N}}R_0}\right\}.
	\end{equation}
	In particular problem \eqref{eq:constrprob} admits a  minimizer.
\end{lem}
\begin{proof}
Let us set
\begin{equation*}
	\gamma=\inf\{\fE_m(E): \ |E|=m\}
\end{equation*}
and $m_0>\max\{\omega_N,1\}$. Consider $m>m_0$ and let us recall that $\beta+N>\alpha$. Moreover, let us observe that
\begin{equation*}
	V_\alpha(B[m])=\left(\frac{m}{\omega_N}\right)^{1+\frac{\alpha}{N}}V_\alpha(B), \qquad P_s(B[m])=\left(\frac{m}{\omega_N}\right)^{1-\frac{s}{N}}P_s(B).
\end{equation*} 
We have
\begin{align}\label{gammaest}
	\begin{split}
	\gamma\le \fE_m(B[m])&=\left(\frac{m}{\omega_N}\right)^{2+\frac{\beta}{N}}\fG_\beta(B)+\left(\frac{m}{\omega_N}\right)^{1+\frac{\alpha}{N}}V_\alpha(B)+ \left(\frac{m}{\omega_N}\right)^{2+\frac{\beta}{N}}P_s(B)\\
	&\le \left(\frac{m}{\omega_N}\right)^{2+\frac{\beta}{N}}\fE_{\omega_N}(B).
\end{split}
\end{align}
Fix any Borel set $E$ such that $|E|=m$ and $\fE_m(E)\le \gamma+V_\alpha(B[m])$. Thus we have, by using the isoperimetric inequalities involving $V_\alpha$ and $P_s$,
\begin{align}\label{defest}
	\begin{split}
		\frac{\fD_\beta(E)}{\fG_\beta(B[m])}&\le \frac{2\left(\frac{m}{\omega_N}\right)^{1+\frac{\alpha}{N}}V_\alpha(B)-V_\alpha(E)+\varepsilon(m)(P_s(B[m])- P_s(E))}{\left(\frac{m}{\omega_N}\right)^{2+\frac{\beta}{N}}\fG_\beta(B)}\\
		&\le \frac{2 V_\alpha(B)}{\fG_\beta(B)}\left(\frac{\omega_N}{m}\right)^{\frac{\beta+N-\alpha}{N}}.
	\end{split}
\end{align}
Now let us consider $\lambda=\left(\frac{\omega_N}{m}\right)^{\frac{1}{N}}$ in such a way that $E_*=\lambda E$ satisfies $|E_*|=\omega_N$. Moreover, we can translate $E_*$ so that $\delta(E_*)=|E_* \Delta B|$. We have, by Equations \eqref{quantisop}, \eqref{defest} and the fact that $\frac{\fD_\beta(E_*)}{\fG_\beta(B)}=\frac{\fD_\beta(E)}{\fG_\beta(B[m])}$,
\begin{equation*}
	|E_* \setminus B|\le |E_* \Delta B|\le C_1\left(2 V_\alpha(B)\left(\frac{\omega_N}{m}\right)^{\frac{\beta+N-\alpha}{N}}\right)^{\frac{1}{2}},
\end{equation*}
for some constant $C_1$ depending only on $\beta$ and $N$. We can set 
\begin{equation}\label{eq:etadef}
	\eta=C_1\left(2 V_\alpha(B)\left(\frac{\omega_N}{m}\right)^{\frac{\beta+N-\alpha}{N}}\right)^{\frac{1}{2}}=:C_2 m^{-\frac{\beta+N-\alpha}{2N}},
\end{equation}
where $C_2$ depends only on $\beta,\alpha,N$ and consider $m_0>\left(\frac{2C_2}{\omega_N}\right)^{\frac{2N}{\beta+N-\alpha}}$ in such a way that for any $m>m_0$ it holds $\eta<\frac{\omega_N}{2}$. By the Truncation Lemma for the fractional perimeter \cite[Lemma $4.5$]{figalli2015isoperimetry} if $s \in (0,1)$ or for the classical perimeter \cite[Lemma $5.1$]{figalli2015isoperimetry} if $s=1$, we know there exist two constants $C_3$ and $C_4$ depending on $N$ and $s$ and a radius $r_* \in [1,1+C_3\eta^{\frac{1}{N}}]$ such that
\begin{equation*}
	P_s(E_* \cap B_{r_*})\le P_s(E_*)-\frac{|E_*\setminus B_{r_*}|}{C_4N\omega_N \eta^{\frac{s}{N}}}.
\end{equation*}
If we define $r_E=\frac{r_*}{\lambda}$ we have
\begin{equation*}
	\lambda^{N-s}P_s(E \cap B_{r_E})\le \lambda^{N-s}P_s(E)-\frac{\lambda^N|E\setminus B_{r_E}|}{C_4N\omega_N \eta^{\frac{s}{N}}},
\end{equation*}
that is to say
\begin{equation}\label{eq:estper1}
	P_s(E \cap B_{r_E})\le P_s(E)-\frac{\lambda^s|E\setminus B_{r_E}|}{C_4N\omega_N \eta^{\frac{s}{N}}}.
\end{equation}
Now let us set $u=\frac{|E\setminus B_{r_E}|}{m}$, $\mu=(1-u)^{-\frac{1}{N}}$ and $F=\mu(E \cap B_{r_E})$. It holds, by definition, $|F|=m$. Let us observe, in particular, that
\begin{equation*}
	u=\frac{|E\setminus B_{r_E}|}{m}=\frac{|E_*\setminus B_{r_*}|}{\omega_N}\le \frac{\eta}{\omega_N}\le \frac{1}{2}.
\end{equation*}
Let us consider any exponent $k>0$. Then, by Lagrange's theorem and the fact that $u \le \frac{1}{2}$ we obtain
\begin{equation}\label{eq:estmu}
	\mu^k\le 1+\frac{k}{N}2^{\frac{k}{N}+1}u.
\end{equation}
By using \eqref{eq:estmu} with $k=N-s$ and \eqref{eq:estper1} we have
\begin{equation}\label{eq:estper2}
	P_s(F)\le \left(1+\frac{N-s}{N}2^{\frac{N-s}{N}+1}u\right)P_s(E)-\left(1+\frac{N-s}{N}2^{\frac{N-s}{N}+1}u\right)\frac{\lambda^s|E\setminus B_{r_E}|}{C_4N\omega_N \eta^{\frac{s}{N}}}.
\end{equation}
Concerning $\fG_\beta(F)$, we have, by using \eqref{eq:estmu} with $k=2N+\beta$,
\begin{equation}\label{eq:estg}
	\fG_\beta(F)\le \left(1+\frac{2N+\beta}{N}2^{\frac{2N+\beta}{N}+1}u\right)\fG_\beta(E \cap B_{r_E})\le \left(1+\frac{2N+\beta}{N}2^{\frac{2N+\beta}{N}+1}u\right)\fG_\beta(E).
\end{equation}
Finally, for $V_\alpha(F)$, we have, still by \eqref{eq:estmu} with $k=\alpha+N$,
\begin{equation}\label{eq:estv}
	V_\alpha(F)\le \left(1+\frac{N+\alpha}{N}2^{\frac{N+\alpha}{N}+1}u\right)V_\beta(E \cap B_{r_E})\le \left(1+\frac{N+\alpha}{N}2^{\frac{N+\alpha}{N}+1}u\right)V_\beta(E).
\end{equation}
Combining inequalities \eqref{eq:estper2}, \eqref{eq:estg} and \eqref{eq:estv} we get
\begin{equation*}
	\fE_m(F)\le \fE_m(E)+C_5u\fE_m(E)-\varepsilon(m)\left(1+\frac{N-s}{N}2^{\frac{N-1}{N}+1}u\right)\frac{\lambda^s|E\setminus B_{r_E}|}{C_4N\omega_N \eta^{\frac{s}{N}}},
\end{equation*}
where $C_5$ is a constant depending on $N,\alpha,\beta,s$. Being $m>\omega_N$, inequality \eqref{gammaest} together with the fact that $\fE_m(E)\le \gamma+V_\alpha(B[m])$ imply
\begin{equation*}
	\fE_m(E)\le 2\left(\frac{m}{\omega_N}\right)^{2+\frac{\beta}{N}}\fE_{\omega_N}(B).
\end{equation*}
Hence, recalling that $u=\frac{|E \setminus B_{r_E}|}{m}$, $\lambda=\omega_N^{\frac{1}{N}}m^{-\frac{1}{N}}$, $\varepsilon(m)=\left(\frac{m}{\omega_N}\right)^{1+\frac{\beta+s}{N}}$ and $\eta$ is defined in \eqref{eq:etadef} we get
\begin{align*}
	\fE_m(F)&\le \fE_m(E)+2C_5|E \setminus B_{r_E}|m^{1+\frac{\beta}{N}}\omega_N^{-2-\frac{\beta}{N}}\fE_{\omega_N}(B)-C_6m^{1+\frac{\beta}{N}+\frac{s(\beta+N-\alpha)}{2N^2}}|E\setminus B_{r_E}|\\
	&=\fE_m(E)+|E \setminus B_{r_E}|m^{1+\frac{\beta}{N}}\omega_N^{-2-\frac{\beta}{N}}\left(2C_5\fE_{\omega_N}(B)-C_6\omega^{2+\frac{\beta}{N}}m^{\frac{s(\beta+N-\alpha)}{2N^2}}\right),
\end{align*}
where $C_6$ is a suitable constant depending on $\alpha,\beta,s,N$. Choosing
\begin{equation*}
	m_0>\left(\frac{2C_5\fE_{\omega_N}(B)}{C_6\omega_N^{2+\frac{\beta}{N}}}\right)^{\frac{2N^2}{s(\beta+N-\alpha)}}
\end{equation*}
and $m>m_0$, we conclude that
\begin{equation*}
	\fE_m(F)\le \fE_m(E).
\end{equation*}
In particular it holds $F \subseteq B_{\mu r_E}$ with 
\begin{align*}
r_E=\frac{r_*}{\lambda}\le (1+C_3\eta^{\frac{1}{N}})\left(\frac{m}{\omega_N}\right)^{\frac{1}{N}}\le (1+C_3C_2^{\frac{1}{N}}m_0^{-\frac{\beta+N-\alpha}{2N^2}})\left(\frac{m}{\omega_N}\right)^{\frac{1}{N}}. 
\end{align*}
Moreover, we have, by \eqref{eq:estmu} with $k=1$,
\begin{equation*}
\mu r_E\le \left(1+\frac{1}{N}2^{\frac{1}{N}}\right)(1+C_3C_2^{\frac{1}{N}}m_0^{-\frac{\beta+N-\alpha}{2N^2}})\left(\frac{m}{\omega_N}\right)^{\frac{1}{N}}=:R_0\left(\frac{m}{\omega_N}\right)^{\frac{1}{N}}
\end{equation*}
where 
\begin{equation}\label{eq:R0}
	R_0=\left(1+\frac{1}{N}2^{\frac{1}{N}}\right)(1+C_3C_2^{\frac{1}{N}}m_0^{-\frac{\beta+N-\alpha}{2N^2}})
\end{equation} 
depends only on $N,\alpha,\beta,s$. Being also $|F|=m$, we have that
\begin{equation*}
	\gamma=\inf\left\{\fE_m: \ |E|=m, E \subseteq B_{\left(\frac{m}{\omega_N}\right)^{\frac{1}{N}}R_0} \right\}.
\end{equation*}
Now let us show that there exists a minimizer. To do this, let us consider a minimizing sequence $E_h \subseteq B_{\left(\frac{m}{\omega_N}\right)^{\frac{1}{N}}r_0}$ with
\begin{equation*}
	\fE_m(E_h)\le \gamma+1.
\end{equation*}
This implies in particular
\begin{equation*}
	P_s(E_h)\le \frac{\gamma+1}{\varepsilon(m)}.
\end{equation*}
By precompactness in $L^1$ of uniformly bounded sequences with respect to the classical perimeter (see \cite{evans2015measure}) for $s=1$ and with respect to the fractional perimeter (see \cite{cozzi2017regularity}) and lower semicontinuity of the involved shape functionals we conclude the proof.
\end{proof}
Now that we have prove that $\fE_m$ admits a minimizer for $m>m_0$, let us show that such minimizer is actually a ball when $m>m_1$ (where $m_1$ is a certain critical mass). To do this, we first need to show that, up to suitable rescaling, minimizers of $\fE_m$ are quasi-minimizers of the (possibly fractional) perimeter.
\begin{lem}\label{lem:quasimin}
	Let $m_0$ be as in Lemma \ref{constr}, $m>m_0$ and $\varepsilon(m)=\left(\frac{m}{\omega_N}\right)^{1+\frac{\beta+s}{N}}$. Consider $E$ a minimizer of $\fE_m$ with $|E|=m$ and such that $E \subseteq B_{\left(\frac{m}{\omega_N}\right)^{\frac{1}{N}}R_0}$, where $R_0$ is defined in Formula \eqref{eq:R0}. Then the set $E_*=\lambda E$ such that $|E_*|=\omega_N$ is a $\Lambda$-quasi minimizer of $P_s$ for some constant $\Lambda$ depending on $\beta,\alpha,s$, i.e. for any measurable set $F$ with $E_* \Delta F \subset \subset B_R(x)$ for some $x \in \R^N$ and $R \in (0,1)$ it holds
	\begin{equation}\label{quasimin}
	P_s(E_*)\le P_s(F)+\Lambda |E_* \Delta F|.
	\end{equation}
\end{lem}
\begin{proof}
	By $|E_*|=\omega_N$ we know that $\lambda=\left(\frac{\omega_N}{m}\right)^\frac{1}{N}\le 1$, since $m>m_0>\omega_N$, see the proof of Lemma \ref{constr}, and $\varepsilon(m)=\lambda^{-N-\beta-s}$. Being $E$ a minimizer of $\fE_m$ under the volume constraint $|E|=m$, we get
	\begin{equation*}
	\fE_m(E)\le \fE_m(B[m])= \lambda^{-2N-\beta}\fG_\beta(B)+\lambda^{-N-\alpha}V_\alpha(B)+\lambda^{-N+s}\varepsilon(m) P_s(B)
	\end{equation*}
	that implies
	\begin{align*}
	\varepsilon(m) P_s(E)&\le \lambda^{-2N-\beta}(\fG_\beta(B)-\fG_\beta(E_*))+\lambda^{-N-\alpha}(V_\alpha(B)-V_\alpha(E_*))+\lambda^{-N+s}\varepsilon(m) P_s(B)\\
	&\le \lambda^{-N-\alpha}V_\alpha(B)+\lambda^{-N+s}\varepsilon(m) P_s(B),
	\end{align*}
where we also used the isoperimetric inequality on $\fG_\beta$. Multiplying last relation by $\lambda^{2N+\beta}$ we have
	\begin{align*}
	P_s(E_*)\le \lambda^{N+\beta-\alpha}V_\alpha(B)+P_s(B)\le V_\alpha(B)+P_s(B).
	\end{align*}
	Now let us consider any measurable set $F \subseteq \R^N$ with $E_* \Delta F \subset \subset B_R(x)$ for some $R<1$. If $P_s(F)\ge P_s(E_*)$ equation \eqref{quasimin} is already verified. Thus, let us assume $P_s(F)<P_s(E_*)$. Then we have, from the fractional (or classical if $s=1$) isoperimetric inequality
	\begin{equation*}
	\frac{|F|}{|B|}\le \left(\frac{P_s(F)}{P_s(B)}\right)^{\frac{N}{N-s}}\le \left(\frac{P_s(E_*)}{P_s(B)}\right)^{\frac{N}{N-s}}\le  \left(1+\frac{V_\alpha(B)}{P_s(B)}\right)^{\frac{N}{N-s}}:=C_1.
	\end{equation*}
	Let us first consider the case $|F| \le \frac{\omega_N}{2}$. Then $|F \Delta E_*| \ge \frac{\omega_N}{2}$ and we can find $\Lambda$ big enough and independent of $m$ in such a way to obtain 
	\begin{equation}\label{eq:Lamba}
	\Lambda \frac{\omega_N}{2}\ge \frac{\lambda^{-s-\alpha}V_\alpha(B)}{\varepsilon(m)}+P_s(B).
	\end{equation}
	Indeed we have
	\begin{equation*}
	\frac{\lambda^{-s-\alpha}V_\alpha(B)}{\varepsilon(m)}+P_s(B)=\lambda^{N+\beta-\alpha}V_\alpha(B)+P_s(B)\le V_\alpha(B)+P_s(B),
	\end{equation*}
being $\lambda \le 1$, hence
	\begin{equation*}
	\Lambda \ge \frac{2(V_\alpha(B)+P_s(B))}{\omega_N}
	\end{equation*}
satisfies inequality \eqref{eq:Lamba}. In this case Equation \eqref{quasimin} follows.\\
	Assume now $|F|>\frac{\omega_N}{2}$ and observe recall that $E_* \Delta F\subset \subset B_1(x)$ for some $x \in \R^N$. Since $E_* \subseteq B_{R_0}$, if $B_1(x)\cap B_{R_0+1}=\emptyset$, then $E_* \Delta F$ is separated from $E_*$. If $s=1$, this implies $P(F)\ge P(E_*)$, which is a contradiction with the assumption $P(F)<P(E_*)$. Moreover, if $s<1$, we have, by definition of fractional perimeter,
	\begin{equation*}
		P_s(F)=P_s(E_*)+P_s(F \Delta E_*)-\frac{2(1-s)}{\omega_{N-1}}\int_{F \Delta E_*}\int_{E_*}|x-y|^{-s-N}dxdy.
	\end{equation*}
	Concerning the last integral, since $E_* \subseteq B_{R_0}$, $F \Delta E_* \subset \subset B_1(x)$ for some $x \in \R^N$ and $B_1(x)\cap B_{R_0+1}=\emptyset$, it holds $|z-y| \ge 1$ for any $z \in E_*$ and $y \in F \Delta E_*$. Hence, we get
	\begin{equation*}
		\int_{F \Delta E_*}\int_{E_*}|x-y|^{-s-N}dxdy\le \omega_N|F \Delta E_*|
	\end{equation*}
	and then
	\begin{equation*}
		P_s(E_*) \le P_s(F_*)+\frac{2(1-s)\omega_N}{\omega_{N-1}} |F \Delta E_*|,
	\end{equation*}
	thus Equation \eqref{quasimin} holds for $\Lambda \ge \frac{2(1-s)\omega_N}{\omega_{N-1}}$.\\
	On the other hand, if $B_1(x)\cap B_{R_0+1}\not = \emptyset$, then $F \subseteq B_{R_0+3}$. Now let us define $\mu=\left(\frac{m}{|F|}\right)^{\frac{1}{N}}$, in such a way that $|\mu F|=m$. Being $E$ a minimizer of $\fE_m$, we have
	\begin{align*}
	\varepsilon(m) P_s(E)&\le \varepsilon(m) \mu^{N-s}P_s(F)+\mu^{N+\alpha}V_\alpha(F)-V_\alpha(E)+\mu^{2N+\beta}\fG_\beta(F)-\fG_\beta(E)\\
	&= \varepsilon(m)\mu^{N-s}P_s(F)+\mu^{N+\alpha}(V_\alpha(F)-V_\alpha(E_*))+((\lambda \mu)^{N+\alpha}-1)V_\alpha(E)\\
	&+\mu^{2N+\beta}(\fG_\beta(F)-\fG_\beta(E_*))+((\lambda\mu)^{2N+\beta}-1)\fG_\beta(E).
	\end{align*}
	Multiplying last inequality by $\lambda^{2N+\beta}$ and using the fact that $\lambda \le 1$, we get
	\begin{align}\label{estfin5}
	\begin{split}
	P_s(E_*)&\le P_s(F)\\&+ ((\lambda\mu)^{N-s}-1)P_s(F)\\&+(\lambda\mu)^{N+\alpha}(V_\alpha(F)-V_\alpha(E_*))\\&+((\lambda \mu)^{N+\alpha}-1)V_\alpha(E_*)\\&+(\lambda\mu)^{2N+\beta}(\fG_\beta(F)-\fG_\beta(E_*))\\
	&+((\lambda\mu)^{2N+\beta}-1)\fG_\beta(E_*).
	\end{split}
	\end{align}
	First of all, observe that $|F \cap E_*|\le |E_*|=\omega_N$ and $|F \setminus E_*|\le |B_1(x)|=\omega_N$, hence $|F|\le 2\omega_N$ and then $-\omega_N \le \omega_N-|F|\le |E_* \Delta F|$. Thus, in particular, $-\frac{1}{2}\le \frac{\omega_N-|F|}{|F|}\le 1$. Moreover $|F|\ge \frac{\omega_N}{2}$, thus we get, for any $\nu\ge N$,
	\begin{align*}
	\begin{split}
	(\lambda \mu)^{\nu}&=\left(1+\frac{\omega_N-|F|}{|F|}\right)^{\frac{\nu}{N}}\le 1+(2^{\frac{\nu}{N}}-1)\left|\frac{\omega_N-|F|}{|F|}\right|\\&\le 1+2(2^{\frac{\nu}{N}}-1)\frac{1}{\omega_N}|E_* \Delta F|.
	\end{split}
	\end{align*}
	This leads to the following estimates
	\begin{align*}
	(\lambda \mu)^{N+\alpha}-1&\le 2(2^{1+\frac{\alpha}{N}}-1)\frac{1}{\omega_N}|E_* \Delta F|,\\
	(\lambda \mu)^{2N+\beta}-1&\le 2(2^{2+\frac{\beta}{N}}-1)\frac{1}{\omega_N}|E_* \Delta F|,
	\end{align*}
	that are respectively used to obtain an upper bound on the fourth and sixth summand of \eqref{estfin5}, together with the fact that $V_\alpha(E_*)\le V_\alpha(B)$ and $\fG_\beta(E_*)\le 2^\beta R_0^{\beta+2N}\omega_N^2$.\\
	On the other hand, we can use Bernoulli's inequality to achieve
	\begin{equation*}
		(\lambda \mu)^{N-s}-1\le2\frac{N-s}{N}\frac{|E_* \Delta F|}{\omega_N},
	\end{equation*}
	that is used to obtain an upper bound on the second summand of \eqref{estfin5}.\\
	Concerning the third summand of \eqref{estfin5}, we use \cite[Equation $2.11$]{muratov2014isoperimetric} (see \cite[Equation $5.9$]{figalli2015isoperimetry} for a more precise statement involving the perimeter), obtaining
	\begin{equation}\label{estfin3}
	V_\alpha(F)-V_\alpha(E_*)\le C_1|E_* \Delta F|,
	\end{equation}
where $C_1$ is a constant depending on $N$ and $\alpha$. Concerning the fifth summand, recall that $F \subseteq B_{R_0+3}$, where $\widetilde{R}=R_0+3$ and $E_* \subseteq B_{R_0}$, so that it holds $F \cup E_* \subseteq B_{R_0+3}$. Thus we have
	\begin{align}\label{estpass2}
		\begin{split}
			\fG_\beta(F)-\fG_\beta(E_*)&=\fG_\beta(F,F \setminus E_*)+\fG_\beta(F,F \cap E_*)-\fG_\beta(E_*)\\
			&\le \fG_\beta(F,F \setminus E_*)+\fG_\beta(F,E_*)-\fG_\beta(E_*)\\
			&=\fG_\beta(F,F \setminus E_*)+\fG_\beta(F\setminus E_*,E_*)+\fG_\beta(F\cap E_*,E_*)-\fG_\beta(E_*)\\
			&\le \fG_\beta(F,F \setminus E_*)+\fG_\beta(F\setminus E_*,E_*)\\ 
			&= \int_{F}\int_{F\setminus E_*}|x-y|^\beta dxdy+\int_{E_*}\int_{F\setminus E_*}|x-y|^\beta dxdy\\
			&\le |F||F\setminus E_*|(2(R_0+3))^\beta+\omega_N|F\setminus E_*|(2(R_0+3))^\beta\\
			&\le 2^\beta(R_0+3)^\beta((R_0+3)^{N}+1)\omega_N|F \setminus E_*|\\
			&\le C_2|E_*\Delta F|,
		\end{split}
	\end{align}
where $C_2:=2^\beta(R_0+3)^\beta((R_0+3)^{N}+1)\omega_N>0$ is a constant depending on $N,\beta,\alpha,s$. Hence, we conclude, by using all the upper bounds we obtained, that there exists a constant $C_3>0$ depending on $N,\beta,\alpha,s$ such that
\begin{equation*}
	P_s(E_*)\le P_s(F)+C_3|E_*\Delta F|,
\end{equation*}
thus equation \eqref{quasimin} is verified by taking $\Lambda \ge C_3$. Finally, if we take $$\Lambda \ge \max\left\{\frac{2(V_\alpha(B)+P_s(B))}{\omega_N},\frac{2(1-s)\omega_N}{\omega_{N-1}},C_3\right\}$$ we conclude the proof.
\end{proof}
Now that we have shown that minimizers of $\fE_m$ are quasi-minimizers of the perimeter (up to a rescaling), we can use this property to improve the regularity of the minimizers of $\fE_m$ and finally prove Theorem \ref{thm:minmix1}. 
\begin{proof}[Proof of Theorem \ref{thm:minmix1}]
	In this proof, given $E$ of finite measure,
	\begin{equation*} \widetilde{\delta}(E)=\min\left\{|E \Delta B_r(x)|; \ x \in \R^N, \ |B_r|=|E|\right\}.
	\end{equation*}
	First of all, let us observe that if the ball $B[m]$ is the minimizer of $\fE_m$ then, by the isoperimetric inequality, it is obviously the minimizer of $\fE$ for any $\varepsilon>\varepsilon(m)$. Thus, let us work directly with $\fE_m$.\\
	Let us argue by contradiction. Let us suppose there exists a sequence $m_h>m_0$ such that $m_h \to +\infty$ and $E_h$ are minimizers of $\fE_{m_h}$ for which $\widetilde{\delta}(E_h)>0$ for all $h \in \N$. By Lemma \ref{constr} we can assume $E_h \subseteq B_{\left(\frac{m_h}{\omega_N}\right)^\frac{1}{N}R_0}$ for any $h \in N$. Being $E_h$ minimizers of the $\fE_{m_h}$, we have
	\begin{align*}
	\frac{P_s(E_h)-P_s(B[m_h])}{ P_s(B[m_h])}&\le \frac{2\left(\frac{m_h}{\omega_N}\right)^{1+\frac{\alpha}{N}}V_\alpha(B)-V_\alpha(E_h)+(G_\beta(B[m_h])-G_\beta(E_h))}{\varepsilon(m_h) \left(\frac{m_h}{\omega_N}\right)^{1-\frac{s}{N}} P_s(B)}
	\\&\le \frac{2V_\alpha(B)}{ P_s(B)}\left(\frac{\omega_N}{m_h}\right)^{\frac{\beta+N-\alpha}{N}}.
	\end{align*}
	By the sharp quantitative isoperimetric inequality for the fractional perimeter (see \cite[Theorem $1.1$]{figalli2015isoperimetry}) when $s \in (0,1)$, or for the classical perimeter (see \cite{fusco2015quantitative} for a survey) when $s=1$, we also have
	\begin{equation*}
	\frac{P_s(E_h)-P_s(B[m_h])}{P_s(B[m_h])}\ge C \left(\frac{\widetilde{\delta}(E_h)}{|E_h|}\right)^2
	\end{equation*}
	where the constant $C$ depends only on $s$ and $N$. Thus, sending $h \to +\infty$ and observing that $\beta+N-\alpha>0$, we have $\frac{\widetilde{\delta}(E_h)}{|E_h|}\to 0$. Being all the quantities involved translation-invariant and scaling-invariant, we may assume $\frac{\widetilde{\delta}(E_h)}{|E_h|}=\frac{\delta(E_{h,*})}{\omega_N}=\frac{|E_{h,*} \Delta B|}{\omega_N}$, where $E_{h,*}$ is given by $\lambda_h E_h$ in such a way that $|E_{h,*}|=\omega_N$.\\
	Since $E_{h,*} \subset B_{R_0}$ for all $h$ and $P_s(E_{h,*})$ are equibounded, by precompactness of the perimeter, up to a not relabelled subsequence, we may assume that $E_{h,*} \to E$ with $|E|=\omega_N$. Moreover, since $\delta(E_{h,*}) \to 0$ we have that $E=B_1$. Then, since $E_{h,*}$ are $\Lambda$-quasi minimizers of $P_s$ by Lemma \ref{lem:quasimin}, by a well known regularity result (see \cite[Corollary $3.6$]{figalli2015isoperimetry} for $s \in (0,1)$ and \cite[Proposition $2.1$ and $2.2$]{cicalese2012selection} for $s=1$) we have that for $h$ large $E_{h,*}$ is an open set of class $C^1$ and that $E_{h,*}\to B_1$ in $C^1$. This means in particular that, for $h$ large, $E_{h,*}$ are nearly-spherical sets $E_{h,*}:=E_{1,u_{h}}$ as in \eqref{eq:nearlyspherical} with  $u_h \in C^1(S^{N-1})$ such that $\lim_{h \to \infty}\Norm{u_h}{C^1(S^{N-1})}=0 $.\\ 
	Now let us denote
	\begin{equation*}
	[[u]]^2_s:=\frac{(1-s)}{\omega_{N-1}}\int_{\partial B}\int_{\partial B}\frac{|u(x)-u(y)|^2}{|x-y|^{N+2s-1}}d\cH^{N-1}(x)d\cH^{N-1}(y)
	\end{equation*}
	and set, for any function $u \in W^{1,\infty}(S^{N-1})$, $[[u]]^2_{1}:=\Norm{\nabla_\tau u}{L^2(S^{N-1})}^2$. By \cite[Theorem $2.1$]{figalli2015isoperimetry} for $s \in (0,1)$ and \cite[Theorem $1.2$]{fuglede1989stability} for $s=1$, we know that
	\begin{equation}\label{def2}
	\frac{P_s(E_{h,*})-P_s(B)}{P_s(B)}\ge C([[u_h]]_{\frac{1+s}{2}}^2+\Norm{u_h}{L^2(S^{N-1})}^2)
	\end{equation}
	for some positive constant $C>0$.\\
	On the other hand, we have, by minimality of $E_h$ and the fact that $\fG_\beta(B[m_h])-\fG_\beta(E_h)\le 0$
	\begin{equation*}
	\varepsilon(m_h)(P_s(E_h)-P_s(B[m_h]))\le V_\alpha(B[m_h])-V_\alpha(E_h)
	\end{equation*}
	and then, we have
	\begin{align}\label{defest2}
	\begin{split}
	\frac{P_s(E_{h,*})-P_s(B)}{P_s(B)}&=\frac{P_s(E_{h})-P_s(B[m_h])}{P_s(B[m_h])}\\&\le \frac{V_\alpha(B[m_h])-V_\alpha(E_h)}{\varepsilon(m_h) P_s(B[m_h])}\\
	&\le C m_h^{\alpha-\beta-N}(V_\alpha(B)-V_\alpha(E_{h,*}))
	\end{split}
	\end{align}
	for some positive constant $C>0$.\\
	Now, by \cite[Lemma $5.3$]{figalli2015isoperimetry} we have
	\begin{equation*}
	V_\alpha(B)-V_\alpha(E_{h,*})\le C([[u_h]]_{\frac{1-\alpha}{2}}^2+\Norm{u_h}{L^2(S^{N-1})}^2).
\end{equation*}
	Let us observe that, for $s \in (0,1)$, it holds
	\begin{align*}
		[[u_h]]_{\frac{1-\alpha}{2}}^2&=\frac{1+\alpha}{2}\int_{S^{N-1}}\int_{S^{N-1}}\frac{|u(x)-u(y)|^2}{|x-y|^{N-\alpha}}d\cH^{N-1}(x)d\cH^{N-1}(y)\\&=\frac{1+\alpha}{2}\int_{S^{N-1}}\int_{S^{N-1}}|x-y|^{\alpha+s}\frac{|u(x)-u(y)|^2}{|x-y|^{N+s}}d\cH^{N-1}(x)d\cH^{N-1}(y)\\
		&\le \frac{1+\alpha}{1-s}2^{\alpha+s}[[u_h]]_{\frac{1+s}{2}}^2.
	\end{align*}
	For $s=1$ we have to use a different argument that we briefly discus here. Precisely, if we consider the operator $\widetilde{I}_\alpha$ such that for any $u \in C^{1}(S^{N-1})$ it holds
	\begin{equation*}
		\widetilde{I}_\alpha u(x)=2\int_{S^{N-1}}\frac{u(x)-u(y)}{|x-y|^{n-\alpha}}d\cH^{N-1}(y),
	\end{equation*}
	we can rewrite $$[[u]]_{\frac{1-\alpha}{2}}^2=\int_{S^{N-1}}u(x)\widetilde{I}_\alpha u(x)d\cH^{N-1}(x).$$
	Now fix $u_h$, denote by $\widetilde{\lambda}_{k,\alpha}$ the eigenvalues of $\widetilde{I}_\alpha$, such that for each $k \ge 0$ and $Y_k \in \cS_k$ it holds $\widetilde{I}_\alpha[Y_k](x)=\widetilde{\lambda}_{k,\alpha}Y_k$, and set, for each $\cS_k$, an orthonormal basis $\cY_k=\{Y_{k,j}\}_{j \le d(k)}$. We have
	\begin{equation*}
		u_h(x)=\sum_{k=0}^{+\infty}\sum_{j=1}^{d(k)}a_{h,k,j}Y_{k,j}(x)
	\end{equation*}
and then
	\begin{equation*}
		[[u_h]]_{\frac{1-\alpha}{2}}^2=\sum_{k=0}^{+\infty}\sum_{j=1}^{d(k)}\widetilde{\lambda}_{k,\alpha}a_{h,k,j}^2.
\end{equation*}
The explicit form of $\widetilde{\lambda}_{k,\alpha}$ is provided, for instance, in \cite[Equations $(7.4)$ to $(7.6)$]{figalli2015isoperimetry}. In particular, let us recall from \cite[Equation $(7.17)$]{figalli2015isoperimetry} that $\widetilde{\lambda}_{0,\alpha}=0$ and $\widetilde{\lambda}_{k+1,\alpha}>\widetilde{\lambda}_{k,\alpha}$. On the other hand, from \cite[Equation $(7.11)$]{figalli2015isoperimetry}, we have
\begin{equation*}
	[[u_h]]_{1}^2=\sum_{k=0}^{+\infty}\sum_{j=1}^{d(k)}\lambda^1_{k}a_{h,k,j}^2,
\end{equation*}
where $\lambda^1_{k}=k(k+N-2)$. In particular, it holds $\lim_{k \to +\infty}\frac{\lambda^1_{k}}{\widetilde{\lambda}_{k,\alpha}}=+\infty$ and $\frac{\lambda^1_{k}}{\widetilde{\lambda}_{k,\alpha}}>0$ for any $k \ge 1$, while $\lambda^1_{0}=\widetilde{\lambda}_{0,\alpha}=0$. Thus, there exists a constant $C>0$ depending only on $N$ and $\alpha$ such that $\widetilde{\lambda}_{k,\alpha}\le C\lambda^1_{k}$, leading to
\begin{equation*}
[[u_h]]_{\frac{1-\alpha}{2}}^2=\sum_{k=0}^{+\infty}\sum_{j=1}^{d(k)}\widetilde{\lambda}_{k,\alpha}a_{h,k,j}^2\le C\sum_{k=0}^{+\infty}\sum_{j=1}^{d(k)}\lambda^1_{k}a_{h,k,j}^2= C[[u_h]]_{1}^2.
\end{equation*}
In general, we conclude, for any $s \in (0,1]$, that
	\begin{equation*}
		V_\alpha(B)-V_\alpha(E_{h,*})\le C([[u_h]]_{\frac{1+s}{2}}^2+\Norm{u_h}{L^2(S^{N-1})}^2).
	\end{equation*}
	and then
	\begin{equation}\label{def1}
	\frac{P_s(E_{h,*})-P_s(B)}{P_s(B)}\le C m_h^{\alpha-\beta-N}([[u_h]]_{\frac{1+s}{2}}^2+\Norm{u_h}{L^2(S^{N-1})}^2).
	\end{equation}
	Thus, from \eqref{def2} and \eqref{def1} we get the following inequality
	\begin{equation*}
	0<C\le m_h^{\alpha-\beta-N}
	\end{equation*}
	where $C$ is a suitable constant. Sending $h \to +\infty$ we get a contradiction.\\
	Hence we know that there exists $m_1$ such that for any $m>m_1$ the minimizer must satisfy $\delta(E)=0$, i.e. it is equivalent to a ball.
\end{proof}
\section*{Acknowledgements}
The author would like to thank Prof. Nicola Fusco for his fundamental support. Moreover, he would like to thank Prof. Rupert Frank and the anonymous referee for their useful comments that really helped improving the paper.  
\bibliographystyle{abbrv}
\bibliography{biblio}
\end{document}